\documentclass{article}

\usepackage{amsmath}
\usepackage{amsfonts}
\usepackage{latexsym}
\usepackage{graphicx} %
\usepackage{float}
\usepackage{amssymb}
\usepackage{amsthm}
\usepackage[margin=2.5cm]{geometry}
\usepackage{url}
\usepackage{color}
\usepackage{enumerate}
\usepackage{hyperref}
\usepackage{cleveref}   %
\usepackage{stmaryrd}   %
\usepackage{mathtools}  %
\usepackage{mathrsfs}
\usepackage{nicematrix}
\usepackage{todonotes}

\usepackage{tikz}
\usepackage{tikz-cd}

\usepackage[T1]{fontenc}     %
\usepackage{lmodern}         %
\usepackage[utf8]{inputenc}  %

\usepackage{hyperref}

\newcommand{\doi}[1]{
    \href{http://dx.doi.org/#1}{\texttt{doi:#1}}
}
\newcommand{\mailto}[1]{\href{mailto:#1}{\texttt{#1}}}

\newtheorem{definition}{Definition}[section]
\newtheorem{lemma}[definition]{Lemma}
\newtheorem{proposition}[definition]{Proposition}
\newtheorem{example}[definition]{Example}
\newtheorem{question}[definition]{Question}
\newtheorem{corollary}[definition]{Corollary}

\newtheorem{theorem}[definition]{Theorem}
\newtheorem{remark}[definition]{Remark}

\newtheorem{maintheorem}{Theorem}

\newtheorem*{THEOREMI}{Theorem~\ref{thm:not-periodic}}
\newtheorem*{THEOREMII}{Theorem~\ref{mainthm:an-aperiodic-tileset-for-every-quadratic-number}}

\newtheoremstyle{named}{}{}{\itshape}{}{\bfseries}{.}{.5em}{#1\thmnote{ #3}}
\theoremstyle{named}

\numberwithin{equation}{section}

\newcommand{\N}{\mathbb{N}}
\newcommand{\Z}{\mathbb{Z}}
\newcommand{\Q}{\mathbb{Q}}
\newcommand{\R}{\mathbb{R}}
\newcommand{\C}{\mathbb{C}}
\newcommand{\I}{\mathbb{I}}

\newcommand{\Acal}{\mathcal{A}}
\newcommand{\Ccal}{\mathcal{C}}
\newcommand{\Fcal}{\mathcal{F}}

\newcommand{\Mcal}{\mathcal{M}}
\newcommand{\Ocal}{\mathcal{O}}
\newcommand{\Scal}{\mathcal{S}}
\newcommand{\Tcal}{\mathcal{T}}
\newcommand{\Ucal}{\mathcal{U}}
\newcommand{\Zcal}{\mathcal{Z}}

\newcommand{\be}{{\boldsymbol{e}}}
\newcommand{\bn}{{\boldsymbol{n}}}
\newcommand{\bm}{{\boldsymbol{m}}}
\newcommand{\zero}{{\boldsymbol{0}}}

\newcommand{\north}{\textsc{North}}
\newcommand{\east}{\textsc{East}}
\newcommand{\west}{\textsc{West}}
\newcommand{\south}{\textsc{South}}

\newcommand{\hh}{\textsc{h}}
\newcommand{\vv}{\textsc{v}}  %

\newcommand{\shiftclosure}[1]{{\overline{#1}^{\sigma}}}

\DeclareMathOperator{\rank}{rank}
\DeclareMathOperator{\nullity}{nullity}
\DeclareMathOperator{\colspan}{colspan}
\DeclareMathOperator{\nullspace}{nullspace}
\DeclareMathOperator{\Ker}{Ker}
\DeclareMathOperator{\Ima}{Im}
\DeclareMathOperator{\dist}{dist}

\newcommand{\fr}[1]{\left\{#1\right\}}

\newcommand{\langl}{\begin{picture}(4.5,7)
\put(1.1,2.5){\rotatebox{60}{\line(1,0){5.5}}}
\put(1.1,2.5){\rotatebox{300}{\line(1,0){5.5}}}
\end{picture}}
\newcommand{\rangl}{\begin{picture}(4.5,7)
\put(.9,2.5){\rotatebox{120}{\line(1,0){5.5}}}
\put(.9,2.5){\rotatebox{240}{\line(1,0){5.5}}}
\end{picture}}

\newcommand\smallvector[2]{
    \left(\begin{smallmatrix} #1\\#2 \end{smallmatrix}\right)
}

\newcommand\bal[2]{
bal_{#2}(#1)
}

\newcommand\labelinside[6]{
\draw[draw] (#1,#2) rectangle (#1+\size,#2+\size);
\node[rotate=0,black] at (#1+0.8, #2+0.5) {$#3$};
\node[rotate=0,black] at (#1+0.5, #2+0.8) {$#4$};
\node[rotate=0,black] at (#1+0.2, #2+0.5) {$#5$};
\node[rotate=0,black] at (#1+0.5, #2+0.2) {$#6$};
}
\newcommand\labeloutside[6]{
\draw[draw] (#1+\size,#2) -- node[swap]{$#3$} (#1+\size,#2+\size)
                          -- node[swap]{$#4$} (#1,#2+\size)
                          -- node[swap]{$#5$} (#1,#2)
                          -- node[swap]{$#6$} (#1+\size,#2);
}
\newcommand\tile[7]{
\begin{scope}
\draw[draw=none,fill=#1] (#2,#3) rectangle (#2+\size,#3+\size);
\labeloutside{#2}{#3}{#4}{#5}{#6}{#7}
\end{scope}}

\newcommand\tilelabelinside[7]{
\begin{scope}
\draw[draw=none,fill=#1] (#2,#3) rectangle (#2+\size,#3+\size);
\labelinside{#2}{#3}{#4}{#5}{#6}{#7}
\end{scope}}

\newcommand\tilelabelinsideboxplus[7]{
\draw[draw=none,fill=blue!20] (#2+.3*\size,#3) rectangle (#2+.7*\size,#3+\size);
\draw[draw=none,fill=blue!20] (#2,#3+.3*\size) rectangle (#2+\size,#3+.7*\size);
\tilelabelinside{#1}{#2}{#3}{#4}{#5}{#6}{#7}
}
\newcommand\tilelabelinsideboxbar[7]{
\draw[draw=none,fill=blue!20] (#2+.3*\size,#3) rectangle (#2+.7*\size,#3+\size);
\tilelabelinside{#1}{#2}{#3}{#4}{#5}{#6}{#7}
}
\newcommand\tilelabelinsideboxminus[7]{
\draw[draw=none,fill=blue!20] (#2,#3+.3*\size) rectangle (#2+\size,#3+.7*\size);
\tilelabelinside{#1}{#2}{#3}{#4}{#5}{#6}{#7}
}
\newcommand\tilelabelinsidesquare[7]{
\tilelabelinside{#1}{#2}{#3}{#4}{#5}{#6}{#7}
}

\def\size{1}

\newcommand\IdTwoThree{
        \left(\begin{smallmatrix}
        0&1&0\\
        0&0&1
        \end{smallmatrix}\right)}
\newcommand\OneOneMinusoneMinusOne{
        \left(\begin{smallmatrix}
        1\\ 1\\ -1\\ -1\\
        \end{smallmatrix}\right)}

\newcommand\ABCD[8]{
   \left(\begin{smallmatrix} 
       #1& #3& #5& #7\\
       #2& #4& #6& #8
   \end{smallmatrix}\right)
}

\title{An aperiodic set of Wang tiles for every quadratic irrational}
\author{Jarkko Kari\footnote{Department of Mathematics and Statistics, University of Turku, Turku, Finland, \mailto{jkari@utu.fi}}
~and Sébastien Labbé\footnote{CNRS – Université de Montréal CRM-CNRS, Montréal, Canada, \mailto{sebastien.labbe@cnrs.fr}} 
~and Pieter Mostert\footnote{Zithulele, South Africa, \mailto{pi.mostert@gmail.com}}}
\date{June 2025}

\begin{document}

\maketitle

\begin{abstract}
We propose a sufficient condition for the non-periodicity of a set of Wang
tiles. 
It applies to sets of Wang tiles whose tiles have vertical or
horizontal stripes. The proof is based on a geometric argument involving 
a quadrilateral circumscribed to a parabola from which we conclude the
irrationality of the densities of the vertical and horizontal stripes.
We apply the sufficient condition to propose new proofs of non-periodicity of
known sets of Wang tiles, including an encoding of Penrose tilings into 24
Wang tiles and the family of metallic mean Wang tiles.

Conversely, for every pair $(\alpha,\beta)\in[0,1]^2$ of irrational numbers in
the same quadratic number field, we construct a finite aperiodic set of
Wang tiles with stripes that admits a valid tiling whose density of
vertical stripes is $\alpha$ and density of horizontal stripes is $\beta$.


\end{abstract}

\setcounter{tocdepth}{1}
\tableofcontents

\section{Introduction}

A tiling of the plane by a set of tiles is a covering of the plane by isometric
copies of the tiles that do not intersect except maybe on their boundary. The
description of the tilings of the plane allowed by a finite set
of tiles is a rich subject within geometry \cite{MR857454} with strong connections
to dynamical systems \cite{MR1970385}, 
number theory \cite{MR2742574} and
computability \cite{MR0297572}. %
In physics, it has direct applications to the theory of quasicrystals
\cite{zbMATH00768067} and aperiodic order \cite{MR3136260}.
We say that a set of tiles is \emph{aperiodic} when two conditions are satisfied:
(\emph{existence}) there exists one tiling of the plane by the tiles and
(\emph{non-periodicity}) none of the tilings are periodic, i.e., invariant under a non-trivial translation.
Proving that a set of tiles is aperiodic is difficult problem,
because proving existence is undecidable in general \cite{MR0216954}.

One strategy to prove that a set of tiles is aperiodic involves the presence of 
a hierarchical structure in the tilings allowed and enforced by the tiles.
A hierarchical structure is described as a substitution rule allowing to construct larger
and larger patches and tilings at the limit, thus proving existence.
If the hierarchical decomposition is uniquely defined, it also proves non-periodicity.
It is also sometimes called the \emph{unique composition property} \cite{MR1156132}.
This is an old strategy \cite{penrose_role_1974} which is still used nowadays
\cite{akiyama_note_2012,MR4770585}. But it often results in a long and technical argument, as
it entails the verification of many combinatorial cases.

A promising approach was proposed in
\cite{MR1417578} where a short proof of non-periodicity was deduced from
multiplicative arithmetic and the fundamental theorem of arithmetics.
Basically, non-periodicity is deduced from the fact that 
the only integer solution to the equation $2^k 3^\ell=1$ is $k=\ell=0$.
This argument was used in~\cite{MR1417576} to design an aperiodic set of 13 Wang tiles, which remained the smallest known such set until the discovery of the 11 tile set and the proof of its minimality by Jeandel and Rao \cite{zbMATH07421483}. The argument was also used in~\cite{MR2369448} to cover a slightly more general family
using integers other than $2$ or $3$.
But, such short arguments for proving non-periodicity remain elusive.

    In this article, we propose a new proof of non-periodicity of sets of tiles 
    based on additive arithmetics. 
    The method works for certain sets of Wang tiles.
    Recall that a Wang tile is a unit square with labelled edges.
    Given a finite set of Wang tiles $\Tcal$, a 
    tiling of the plane by translated copies of the tiles is \emph{valid}
    if the label of the common edge of adjacent tiles is the same.
    We represent a valid tiling as a map $\Z^2\to\Tcal$ called a \emph{configuration}.
    The set $\Omega_\Tcal$ of valid configurations defines what is called a \emph{subshift}
    which is closed for the prodiscrete topology and for the $\Z^2$-action induced by the shift.
    A subshift $\Omega_\Tcal$ is said to be \emph{aperiodic} if it is non-empty and if every
    configuration in $\Omega_\Tcal$ is non-periodic.
    As mentioned above, proving that a subshift of finite type is non-empty is
    undecidable \cite{MR0216954} and proving that it is non-periodic is also
    difficult.

    The method proposed in this article works for sets of Wang tiles with stripes.
    See the precise definition in Section~\ref{sec:wang-tiles-with-stripes}.
    Sets of Wang tiles with stripes include many examples, the simplest non-trivial one
    being the Ammann set of 16 Wang tiles \cite{MR857454}.
    We will use this example to illustrate, in the remainder of the
    introduction, the main results of this article.
We say that the Ammann set $\Tcal$ of 16 Wang tiles 
has stripes because it can be partitioned into four non-empty subsets 
$ \Tcal = \Tcal_\square \cup \Tcal_\boxbar \cup \Tcal_\boxminus \cup \Tcal_\boxplus$
as follows:
\begin{align*}
	\Tcal_\square   &= 
		\left\{
		\begin{tikzpicture}[baseline=.4cm]
		\tilelabelinsidesquare{none}{0}{0}{1}{1}{2}{2}
		\end{tikzpicture}\right\},\\
	\Tcal_\boxminus &= 
		\left\{
		\begin{tikzpicture}[baseline=.4cm]
		\tilelabelinsideboxminus{none}{0*1.1}{0}{4}{2}{6}{1}
		\tilelabelinsideboxminus{none}{1*1.1}{0}{5}{2}{3}{1}
		\tilelabelinsideboxminus{none}{2*1.1}{0}{3}{2}{6}{2}
		\tilelabelinsideboxminus{none}{3*1.1}{0}{5}{1}{4}{1}
		\end{tikzpicture}\right\},\\
	\Tcal_\boxbar   &= 
		\left\{
		\begin{tikzpicture}[baseline=.4cm]
		\tilelabelinsideboxbar{none}{0*1.1}{0}{2}{4}{1}{6}
		\tilelabelinsideboxbar{none}{1*1.1}{0}{2}{5}{1}{3}
		\tilelabelinsideboxbar{none}{2*1.1}{0}{2}{3}{2}{6}
		\tilelabelinsideboxbar{none}{3*1.1}{0}{1}{5}{1}{4}
		\end{tikzpicture}\right\},\\
	\Tcal_\boxplus  &= 
		\left\{
		\begin{tikzpicture}[baseline=.4cm]
		\tilelabelinsideboxplus{none}{0*1.1}{0}{3}{3}{4}{4}
		\tilelabelinsideboxplus{none}{1*1.1}{0}{4}{4}{5}{5}
		\tilelabelinsideboxplus{none}{2*1.1}{0}{6}{6}{3}{3}
		\tilelabelinsideboxplus{none}{3*1.1}{0}{4}{3}{4}{5}
		\tilelabelinsideboxplus{none}{4*1.1}{0}{6}{3}{4}{3}
		\tilelabelinsideboxplus{none}{5*1.1}{0}{3}{4}{5}{4}
		\tilelabelinsideboxplus{none}{6*1.1}{0}{3}{6}{3}{4}
		\end{tikzpicture}\right\}.
\end{align*}
The alphabet of the labels of horizontal sides and vertical sides of
the Ammann set of 16 Wang tiles is $\{1,2,3,4,5,6\}$.
We may observe that for every pair $(u,v)$ of labels of opposite edges 
we have $u\in\{1,2\}$ if and only if $v\in\{1,2\}$
and $u\in\{3,4,5,6\}$ if and only if $v\in\{3,4,5,6\}$.
Thus, we associate the subset $\{3,4,5,6\}$ with the presence of a horizontal or vertical stripe
and the subset $\{1,2\}$ with the absence of such a stripe.
Tilings by Wang tiles with stripes are naturally divided into 
rows and columns with or without stripes, see Figure~\ref{fig:Ammann-tiling}. Note that the 
tile sets by Kari~\cite{MR1417578} and Culik~\cite{MR1417576} have analogous horizontal stripes but no vertical stripes. Our proofs essentially use stripes in both directions.

\begin{figure}
\begin{center}
\includegraphics{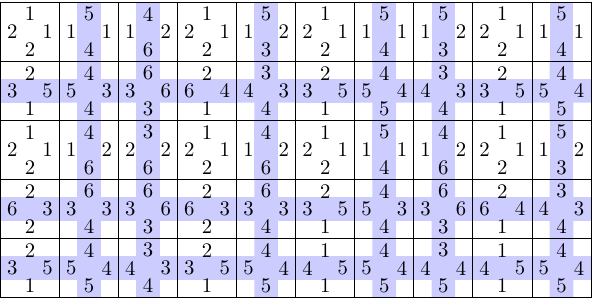}
\end{center}
    \caption{A tiling of a $10\times 5$ rectangle with the Ammann set of 16 Wang tiles.
             The tiling is made of rows and columns with or without stripes
             according to the partition $\{1,2\}\cup\{3,4,5,6\}$ of the
             alphabet of labels.
             The density of these stripes in a valid tiling
             of the plane is $\varphi^{-1}$ where $\varphi$ is the golden ratio.}
    \label{fig:Ammann-tiling}
\end{figure}

\subsection*{A new criteria for non-periodicity}

Our first result proposes a sufficient condition for the density of vertical
and horizontal stripes to be the solution of a system of quadratic equations.
We deduce non-periodicity of the set of tilings when the solutions to the
system are irrational. The sufficient condition depends on a quadrilateral
in the plane $\R^2$ that is determined by a set of Wang tiles with stripes.
Here is what it means for the Ammann set.
Suppose that we map the vertical labels
and the horizontal labels to vectors as follows:
\begin{align*}
    \vv:
1\mapsto \left(\begin{smallmatrix}0\\ 1\end{smallmatrix}\right),\quad
2\mapsto \left(\begin{smallmatrix}0\\ 0\end{smallmatrix}\right),\quad
3\mapsto \left(\begin{smallmatrix}0\\ 1\end{smallmatrix}\right),\quad
4\mapsto \left(\begin{smallmatrix}1\\ 0\end{smallmatrix}\right),\quad
5\mapsto \left(\begin{smallmatrix}1\\ 1\end{smallmatrix}\right),\quad
6\mapsto \left(\begin{smallmatrix}0\\ 0\end{smallmatrix}\right),\\
    \hh:
1\mapsto \left(\begin{smallmatrix}1\\ 0\end{smallmatrix}\right),\quad
2\mapsto \left(\begin{smallmatrix}0\\ 1\end{smallmatrix}\right),\quad
3\mapsto \left(\begin{smallmatrix}1\\ 0\end{smallmatrix}\right),\quad
4\mapsto \left(\begin{smallmatrix}1\\ 1\end{smallmatrix}\right),\quad
5\mapsto \left(\begin{smallmatrix}2\\ 0\end{smallmatrix}\right),\quad
6\mapsto \left(\begin{smallmatrix}0\\ 1\end{smallmatrix}\right).
\end{align*}
For every tile 
$\tau=
    \begin{tikzpicture}[scale=.7,auto,baseline=0.30cm]
    \tilelabelinside{white}{0}{0}{a}{b}{c}{d}
    \end{tikzpicture}
\in\Tcal$, we observe that
$\vv(a) + \hh(b) - \vv(c) - \hh(d)$
takes only 4 distinct values
and the value depends only on the presence or absence of the stripes.
More precisely, for every tile $\tau\in\Tcal$, we have
\begin{equation*}
    \vv(a) + \hh(b) - \vv(c) - \hh(d)
        =
\begin{cases}
    A & \text{ if }\tau\in\Tcal_\square,\\
    B & \text{ if }\tau\in\Tcal_\boxminus,\\
    C & \text{ if }\tau\in\Tcal_\boxbar,\\
    D & \text{ if }\tau \in\Tcal_\boxplus,
\end{cases}
\end{equation*}
where
\[
A=\left(\begin{smallmatrix}1\\0\end{smallmatrix}\right),\quad
B=\left(\begin{smallmatrix}0\\1\end{smallmatrix}\right),\quad
C=\left(\begin{smallmatrix}1\\-1\end{smallmatrix}\right)\quad\text{ and }\quad
D=\left(\begin{smallmatrix}-1\\0\end{smallmatrix}\right).
\]
We see the points $A$, $B$, $D$ and $C$ (in that particular order) as the
consecutive vertices of a quadrilateral in the plane $\R^2$.
Note that the quadrilateral can be simple (not self-intersecting) or complex
(self-intersecting).
When the equation above holds for every tile, we say that the set  $\Tcal$ of Wang tiles
with stripes \emph{determines the quadrilateral $ABDC$}
(see Definition~\ref{def:determines-a-quadrilateral}).
In this situation, we prove the following result.

\newcommand\StatementMainTHEOREMI{
Let $\Tcal$ be a set of Wang tiles
having horizontal and vertical stripes with partition
$ \Tcal = \Tcal_\square \cup \Tcal_\boxbar \cup \Tcal_\boxminus \cup \Tcal_\boxplus$.
Suppose that $\Tcal$ determines the quadrilateral $ABDC$ for some integer points
$A=\left(\begin{smallmatrix} A_1\\A_2 \end{smallmatrix}\right),
 B=\left(\begin{smallmatrix} B_1\\B_2 \end{smallmatrix}\right),
 C=\left(\begin{smallmatrix} C_1\\C_2 \end{smallmatrix}\right),
 D=\left(\begin{smallmatrix} D_1\\D_2 \end{smallmatrix}\right)\in\Z^2$.
If
    \[
    \det\left(\begin{smallmatrix} A_1&D_1\\A_2&D_2 \end{smallmatrix}\right)=0,
        \qquad
    \det\left(\begin{smallmatrix} B_1&D_1\\B_2&D_2 \end{smallmatrix}\right)
    =-\det\left(\begin{smallmatrix} C_1&D_1\\C_2&D_2 \end{smallmatrix}\right)\neq0
    \]
    and there exists $i\in\{1,2\}$ such that $(B_i+C_i)^2-4A_iD_i>0$ is not a perfect square,
    then the set of Wang tiles $\Tcal$ admits no periodic tiling.
}

\begin{maintheorem}\label{thm:not-periodic}
    \StatementMainTHEOREMI
\end{maintheorem}

We observe that the hypotheses of Theorem~\ref{thm:not-periodic} are satisfied
for the Ammann set of 16 Wang tiles presented above
with $(B_1+C_1)^2-4A_1D_1=5$. Thus, it provides a new proof that the Ammann set
of Wang tiles admits no periodic tiling, without using the classical argument
based on unique decomposition property or self-similarity.
Also, Theorem~\ref{thm:not-periodic}
provides a positive answer to Question 2 in \cite{MR4963140}
which was asking if one can directly prove that a Wang shift is non-periodic
using equations satisfied by labels of Wang tiles given by integer vectors 
following the short arithmetical argument for the non-periodicity of Kari's
tile set.

The proof of Theorem~\ref{thm:not-periodic} is short and is done in
Section~\ref{sec:aperiodicityproof}. In Section~\ref{sec:applications},
we apply Theorem~\ref{thm:not-periodic} to obtain new proofs of non-periodicity of
many known sets of Wang tiles including the encoding of Penrose tilings into 24 Wang tiles
and the recent family of metallic mean Wang tiles.
We also propose an algorithm based on linear algebra 
to find the quadrilateral determined by a set of Wang tiles with stripes.
This is done in Section~\ref{sec:the-kernel-method}.

\subsection*{Stripe densities are solutions to a quadratic equation}

\tikzset{blackdot/.style={fill=black,inner sep=0.4mm,circle}}
\tikzset{whitedot/.style={fill=white,draw=black,inner sep=0.4mm,circle,line width=.3mm}}
\tikzset{reddot/.style  ={fill=white,draw=red,inner sep=0.4mm,circle,line width=.3mm}}

\begin{figure}
\begin{center}
\begin{tikzpicture}[scale=.8]
    \clip (-3.7, -4.4) rectangle (10.4,4.4);
    \draw[dotted,thick] (-3,4) -- (6,-5); %
    \draw[dotted,thick] (4 ,5) -- (-6,-5); %
    \draw[dotted,thick] (-11,5) -- (9,-5); %
    \draw[dotted,thick] (1 ,5) -- (1,-5); %
    \draw[fill=blue!20] (1,0)  -- (0,1)  -- (-1,0) -- (1,-1) ;
    \draw[red,very thick,dashed] (-5,-5*1.618) -- node[sloped,pos=.3,above] {$y=\varphi x$} 
                                 (5,5*1.618);
    \draw[red,very thick,dashed] (-5,-5*-.618) -- node[sloped,pos=.15,above] {$y=-\varphi^{-1} x$}
                                 (9,9*-.618);
    =
    \node[whitedot,label={right:$A=\smallvector{1}{0}$}]  (A) at (1,0) {};
    \node[whitedot,label={above:$B=\smallvector{0}{1}$}]  (B) at (0,1) {};
    \node[whitedot,label={below:$C=\smallvector{1}{-1}$}] (C) at (1,-1) {};
    \node[whitedot,label={left:$D=\smallvector{-1}{0}$}]  (D) at (-1,0) {};
    \node[whitedot,label={right:$O$}]   (O) at (0,0) {};
    \draw[draw=blue,very thick] (A) -- (B) -- (D) -- (C) -- (A) -- cycle;
    \node[reddot,label={left:$R$}]    (R) at (.382,.618) {};
    \node[reddot,label={left:$S$}]    (S) at (-.618+.382,-.382) {};
    \node[reddot,label={left:$P$}]    (P) at (1,-.618) {};
    \node[reddot,label={left:$Q$}]    (Q) at (-.618,.382) {};
    \draw[very thick,domain = -5:5, variable = \y]  plot ({.25*\y*\y-.5*\y+1.25},{\y});
    \node[right,align=left] at (3,3.5) {parabola $x=\frac{1}{4}(y-1)^2+1$};
\end{tikzpicture}
\end{center}
    \caption{
    The Ammann set of 16 Wang tiles determines the quadrilateral $ABDC$ with
    $\ABCD{A_1}{A_2}{B_1}{B_2}{C_1}{C_2}{D_1}{D_2}
    =\ABCD{1}{0}{0}{1}{1}{-1}{-1}{0}$.
        The parabola tangent to the four lines supported by the sides of the
        quadrilateral $ABDC$ 
        has equation $x=\frac{1}{4}(y-1)^2+1$.
        The density of tiles with stripes in a tiling by Ammann tiles
        satisfy a system of Diophantine equations whose solutions
        are represented by the two lines passing through the origin $O$
        tangent to the parabola.
        The slope of these tangent lines are irrational
        from which we deduce the non-periodicity of the Ammann set of Wang tiles.
        One may interact with this figure at 
        \url{https://www.geogebra.org/m/a4vmxgft}.
    }
    \label{fig:the-parabola-for-Ammann-case}
\end{figure}
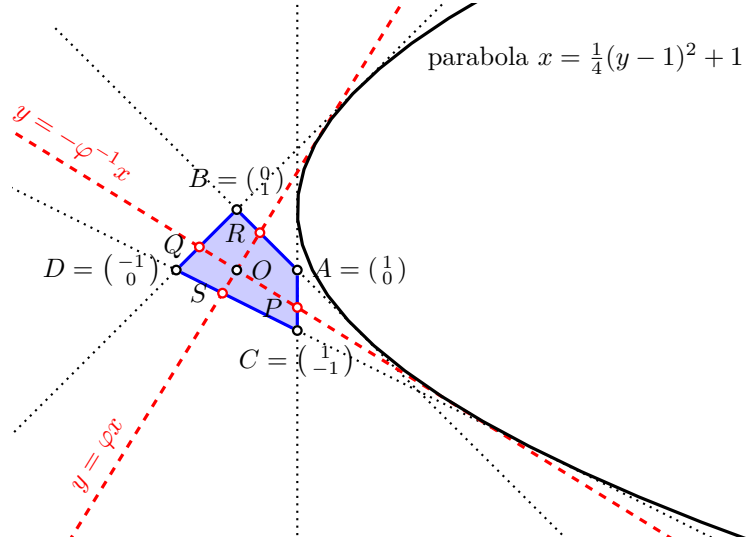

The proof of Theorem~\ref{thm:not-periodic} is based on an elegant equation
satisfied by the densities of vertical and horizontal stripes.
Suppose that the vertical stripes (\emph{i.e.}, tiles in $\Tcal_\boxbar\cup
\Tcal_\boxplus$) have density $\alpha$, and that the horizontal stripes (tiles in
$\Tcal_\boxminus\cup \Tcal_\boxplus$) have density $\beta$. 
We show in Lemma~\ref{lem:doubly-periodic-densities-satisfy-equation} 
for every doubly-periodic configuration
(and in Proposition~\ref{prop:sufficient-condition-quadratic-equation}
for all configurations),
that the densities satisfy the following equation:
\begin{equation}
 \label{alpha_beta_relations}
    \left(\begin{smallmatrix} 0\\0 \end{smallmatrix}\right)=
     (1-\alpha)\left((1-\beta)A 
                    +\beta    B \right)
        +\alpha\left((1-\beta)C 
                    +\beta    D \right).
\end{equation}
Therefore, if the quadrilateral $ABDC$ determined by a set of Wang tiles
$\Tcal$ is such that the solutions to Equation~\eqref{alpha_beta_relations}
are quadratic irrationals, we conclude that the set $\Tcal$ admits no periodic
tiling.
This is how Theorem~\ref{thm:not-periodic} is proved.

The solutions to Equation~\eqref{alpha_beta_relations}
are studied in detail in Section~\ref{sec:solution-to-alpha-beta-equation}.
Assuming that the quadrilateral $ABDC$ is in a generic position,
where consecutive and opposite sides are not parallel,
it is known that there exists a unique parabola tangent
to the four lines defined by the sides of the quadrilateral $ABDC$
\cite{zbMATH02521565}.
The right hand side of Equation~\eqref{alpha_beta_relations}
describes a set of tangent lines to this parabola.
Thus, a solution $(\alpha,\beta)$ corresponds to a tangent line
passing through the origin $O$.
For example, for the quadrilateral $ABDC$
determined above by the Ammann set of Wang tiles,
we find that the parabola tangent to the four lines supported by the sides
of the quadrilateral has equation $x=\frac{1}{4}(y-1)^2+1$,
see Figure~\ref{fig:the-parabola-for-Ammann-case}.
One of the two tangent lines passing through the origin 
intersects the segments $AB$ and $CD$ of the quadrilateral in the points
$R=(1-\beta)A +\beta B$ and
$S=(1-\beta)C +\beta D$ for some $\beta\in\R$.
The other tangent line to the parabola passing through the origin 
intersects the segments $AC$ and $BD$ of the quadrilateral
in the points
$P=(1-\alpha)A +\alpha C$ and
$Q=(1-\alpha)B +\alpha D$ for some $\alpha\in\R$.
We may recover the value of $\alpha$ and $\beta$ from this geometric construction
because the points $P$, $Q$, $R$ and $S$ satisfy
\begin{equation}\label{eq:alpha_beta_relations_with_points_PQRS}
\smallvector{0}{0}
=O
=(1-\beta)P+\beta Q
=(1-\alpha)R +\alpha S
\end{equation}
when $\alpha=\beta=\varphi^{-1}$
where $\varphi$ is the golden ratio.
We observe that Equation~\eqref{eq:alpha_beta_relations_with_points_PQRS}
is equivalent to Equation~\eqref{alpha_beta_relations}.

The solutions to the equation for every quadrilateral $ABDC$
are described very precisely in
Proposition~\ref{prop:alpha_beta_solutions}.
Also, in Lemma~\ref{lem:a-solution-for-every-pair-of-quadratic-numbers},
we show that for every pair $(\alpha, \beta)$ of elements in the
same quadratic number field, there exist four points $A,B,C,D \in \Z^2$ such
that $(\alpha,\beta)$ is a solution to Equation~\eqref{alpha_beta_relations}. 

\subsection*{A set of Wang tiles for every pair of quadratic numbers}

Many aperiodic sets of tiles are related to the golden ratio:
Penrose tiles \cite{penrose_role_1974},
Ammann tiles \cite{MR857454},
Jeandel-Rao tiles \cite{zbMATH07421483} and
the Hat tile \cite{MR4770585}, just to list a few.
Aperiodic sets of tiles exists whose structure is explained by other algebraic numbers
\cite{MR3791847} or even non-alebraic \cite{MR4822285} but examples remain sparse.
Recently, a new family of aperiodic Wang tiles was introduced whose tilings are non-periodic
and associated with metallic means \cite{MR4944989,MR4963140}. 
For every integer $n\geq1$, the $n$-th metallic mean is the positive root
$\beta$ of the polynomial $x^2-nx-1$.
Tilings with sets of Wang tiles from this family have horizontal and vertical stripes
and the density of these stripes are equal to the inverse of $\beta$.

\begin{question}\label{question:for-which-stripe-densities}
For which vertical and horizontal stripe densities can we construct an
    \emph{as simple as possible} set of Wang tiles with stripes whose tilings
    achieve these densities?
\end{question}

Other questions like these were listed in the last section of \cite{MR4963140}.
Note that there are partial answers to that question.
Indeed, for every two-dimensional (non-uniform) rectangular substitution, 
we know there exists a general construction of a set of Wang tiles 
whose set of valid configurations are representatives of configurations
obtained by the substitution rule \cite{MR1014984}.
Considering Mozes' construction for direct products of two distinct
one-dimensional substitutions already give lots of examples.
But, the construction proposed by Mozes is complicated as there is a lot of
information stored in each Wang tile. As a consequence, 
the construction is very hardly used in practice. 
Here, as in the family of metallic mean Wang tiles, we are interested in
solutions which are as simple as possible.

In this article, we provide a simple constructive answer to
Question~\ref{question:for-which-stripe-densities} for algebraic numbers of
degree 2.
Namely, for every pair $(\alpha,\beta)\in[0,1]^2$ of elements in the same quadratic number field
such that both are rational or irrational,
we construct a finite set of Wang tiles with stripes whose 
density of vertical stripes is $\alpha$ and 
density of horizontal stripes is $\beta$.

\newcommand\StatementMainTHEOREMII{
    Let $(\alpha,\beta)\in[0,1]^2$
    be a pair of elements of a quadratic number field $K$
    such that
    $\alpha$ and $\beta$ are both irrational or both rational.
    There exists a finite set 
    $\Tcal = \Tcal_\square \cup \Tcal_\boxbar \cup \Tcal_\boxminus \cup \Tcal_\boxplus$
    of Wang tiles with vertical and horizontal stripes 
    and there exist four non-collinear points $A,B,C,D \in \Z^2$
    such that
\begin{enumerate}
    \item[(i)] $(\alpha,\beta)$ is a solution to Equation~\eqref{alpha_beta_relations}, 

    \item[(ii)] $\Tcal$ determines the quadrilateral $ABDC$
        whose opposite sides are not parallel,

    \item[(iii)] ${\cal T}$ admits a tiling $w:\Z^2\longrightarrow {\cal T}$
        where vertical stripes (\emph{i.e.}, tiles in $\Tcal_\boxbar\cup
        \Tcal_\boxplus$) have density $\alpha$, and the horizontal stripes
        (tiles in $\Tcal_\boxminus\cup \Tcal_\boxplus$) have density $\beta$,
        and

    \item[(iv)] in every valid tiling by $\Tcal$, the horizontal and the
        vertical densities of stripes exist and are solutions of
        Equation~\eqref{alpha_beta_relations}.

\end{enumerate}
In particular, if $\alpha$ and $\beta$ are irrational, then the tile set $\Tcal$ is aperiodic.}

\begin{maintheorem}\label{mainthm:an-aperiodic-tileset-for-every-quadratic-number}
    \StatementMainTHEOREMII
\end{maintheorem}

Theorem~\ref{mainthm:an-aperiodic-tileset-for-every-quadratic-number} is proved in 
Section~\ref{sec:new-constructions}
based on results proved in 
Section~\ref{sec:solution-to-alpha-beta-equation}
and Section~\ref{sec:stripes-densities-satisfy-equation}.

\subsection*{Structure of the article}

In Section~\ref{sec:preliminaries}, we recall the notion
of Wang tiles and present some specific notation used in this article.
In Section~\ref{sec:aperiodicityproof}, we prove
Theorem~\ref{thm:not-periodic}, that is, we present a short argument for
concluding the non-periodicity of a set of Wang tiles with stripes provided
some sufficient condition is satisfied.
In Section~\ref{sec:applications},
we apply the sufficient condition to propose 
new proofs of non-periodicity of known sets of Wang tiles, including
an encoding of Penrose tilings into 24 Wang tiles and the family of metallic mean Wang tiles.
In Section~\ref{sec:the-kernel-method}, we present a method based on linear
algebra to compute a quadrilateral determined by a set of Wang tiles with
vertical and horizontal stripes if one exists.
We use this method to give new proofs of non-periodicity for three additional
known aperiodic sets of Wang tiles.
In Section~\ref{sec:solution-to-alpha-beta-equation},
we present a complete description of the set of solutions to 
Equation~\eqref{alpha_beta_relations}.
In Section~\ref{sec:stripes-densities-satisfy-equation},
we show that densities of stripes in valid
configurations over a set of tiles determining a quadrilateral $ABDC$ must
exist and be a solution to Equation~\eqref{alpha_beta_relations}.
In Section~\ref{sec:new-constructions}, we prove
Theorem~\ref{mainthm:an-aperiodic-tileset-for-every-quadratic-number}. We conclude with some open questions in Section~\ref{sec:open}.

\section{Preliminaries}
\label{sec:preliminaries}

\subsection{Topological dynamical systems}

Most of the notions introduced here can be found in \cite{MR648108}.
A \emph{dynamical system} is
a triple $(X,G,T)$, where $X$ is a topological space, $G$ is a topological
group and $T$ is a continuous function $G\times X\to X$ defining a left action
of $G$ on $X$:
if $x\in X$, $e$ is the identity element of $G$ and $g,h\in G$, then using
additive notation for the operation in $G$ we have $T(e,x)=x$
and $T(g+h,x)=T(g,T(h,x))$.
In other words, if one denotes the transformation $x\mapsto T(g,x)$
by $T^g$, then $T^{g+h}=T^g T^h$.
In this work, we consider the Abelian group $G=\Z\times\Z$.

If $Y\subset X$, let $\overline{Y}$ denote the topological closure of $Y$ and
let $\overline{Y}^T:=\cup_{g\in G}T^g(Y)$ denote the $T$-closure of $Y$.
A subset $Y\subset X$ is \emph{$T$-invariant} if $\overline{Y}^T=Y$.
A dynamical system $(X,G,T)$ is called \emph{minimal} if $X$ does
not contain any nonempty, proper, closed $T$-invariant subset.
The left action of $G$ on $X$ is \emph{free}
if $g=e$ whenever there exists $x\in X$ such that $T^g(x)=x$.

Let $(X,G,T)$ and $(Y,G,S)$ be two dynamical systems with
the same topological group $G$.
A \emph{homomorphism} $\theta:(X,G,T)\to(Y,G,S)$ is a continuous
function $\theta:X\to Y$ satisfying the commuting property
that $S^g\circ\theta=\theta\circ T^g$ for every $g\in G$.
A homomorphism $\theta:(X,G,T)\to(Y,G,S)$ is called an \emph{embedding}
if it is one-to-one, a \emph{factor map} if it is onto, and a \emph{topological
conjugacy} if it is both one-to-one and onto and its inverse map is continuous.
If $\theta:(X,G,T)\to(Y,G,S)$ is a factor map,
then $(Y,G,S)$ is called a \emph{factor} of $(X,G,T)$
and $(X,G,T)$ is called an \emph{extension} of $(Y,G,S)$.
Two dynamical systems are \emph{topologically conjugate} if there is a
topological conjugacy between them.

\subsection{Subshifts and shifts of finite type}\label{sec:subshift-SFT}

In this section, we introduce multidimensional subshifts,
a particular type of dynamical system 
\cite[\S 13.10]{MR1369092},
\cite{MR1861953,MR2078846,MR3525488}.
Let $\Acal$ be a finite set, $d\geq 1$, and let $\Acal^{\Z^d}$ be the set of all maps
$x:\Z^d\to\Acal$, equipped with the compact product topology. 
An element $x\in\Acal^{\Z^d}$ is called a \emph{configuration}
and we write it as $x=(x_\bm)=(x_\bm:\bm\in\Z^d)$,
where $x_\bm\in\Acal$ denotes the value of $x$ at $\bm$. 
The topology on $\Acal^{\Z^d}$ is compatible with the metric defined for all
configurations $x,x'\in\Acal^{\Z^d}$ by $\dist(x,x')=2^{-\min\left\{\Vert\bn\Vert\,:\,
x_\bn\neq x'_\bn\right\}}$
where $\Vert\bn\Vert = |n_1| + \dots + |n_d|$.
The \emph{shift action} $\sigma:\bn\mapsto
\sigma^\bn$ of the additive group $\Z^d$ on $\Acal^{\Z^d}$ is defined by
\begin{equation}\label{eq:shift-action}
    (\sigma^\bn(x))_\bm = x_{\bm+\bn}
\end{equation}
for every $x=(x_\bm)\in\Acal^{\Z^d}$ and $\bn\in\Z^d$. 
If $X\subset \Acal^{\Z^d}$,
let $\overline{X}$ denote the topological closure of $X$
and let $\shiftclosure{X}:=\{\sigma^\bn(x)\mid x\in X, \bn\in\Z^d\}$
denote the shift-closure of $X$.
A subset $X\subset
\Acal^{\Z^d}$ is \emph{shift-invariant} if 
$\shiftclosure{X}=X$. A closed, shift-invariant subset
$X\subset\Acal^{\Z^d}$ is a \emph{subshift}. 
If $X\subset\Acal^{\Z^d}$ is a subshift we write
$\sigma=\sigma^X$ for the restriction of the shift action
\eqref{eq:shift-action} to $X$. 
When $X$ is a subshift,
the triple $(X,\Z^d,\sigma)$ is a dynamical system
and the notions presented in the previous section hold.

A configuration $x\in X$ is \emph{periodic} if there is a nonzero vector
$\bn\in\Z^d\setminus\{\zero\}$ such that $x=\sigma^\bn(x)$,
and otherwise it is \emph{nonperiodic}.
We say that a nonempty subshift $X$ is \emph{aperiodic}
if the shift action $\sigma$ on $X$ is free.

For any subset $S\subset\Z^d$ let $\pi_S:\Acal^{\Z^d}\to\Acal^S$ denote the
projection map which restricts every $x\in\Acal^{\Z^d}$ to $S$. 
A \emph{pattern} is a function $p\in\Acal^S$ for some finite subset
$S\subset\Z^d$.
To every pattern $p\in\Acal^S$ corresponds
a subset $\pi_S^{-1}(p)\subset\Acal^{\Z^d}$ called \emph{cylinder}.
A nonempty set $X\subset\Acal^{\Z^d}$ is a
\emph{subshift} if and only if there exists a set $\Fcal$
of \emph{forbidden} patterns such that
\begin{equation}\label{eq:SFT}
    X = \{x\in\Acal^{\Z^d} \mid \pi_S\circ\sigma^\bn(x)\notin\Fcal
    \text{ for every } \bn\in\Z^d \text{ and } S\subset\Z^d\},
\end{equation}
see \cite[Prop.~9.2.4]{MR3525488}.
A subshift $X\subset\Acal^{\Z^d}$ is a 
\emph{subshift of finite type} (SFT) if there exists a finite set $\Fcal$ such that \eqref{eq:SFT} holds.
In this article, we consider subshifts of finite type on $\Z\times\Z$, that is, the case
$d=2$.

\subsection{Wang tiles}
\label{sec:wang-tiles}

A \emph{Wang tile} 
is a tuple of four colors $(a,b,c,d)\in I\times J\times
I\times J$
where $I$
is a finite set of vertical colors
and $J$
is a finite set of horizontal colors; see
\cite{wang_proving_1961,MR0297572}.
A Wang tile is represented as a unit square with colored edges:
\begin{center}
    \raisebox{-9.5mm}{
    \begin{tikzpicture}[auto]
    \tile{white}{0}{0}{a}{b}{c}{d}
    \end{tikzpicture}}
\end{center}
For each Wang tile $\tau=(a,b,c,d)$, let
$\east(\tau)=a$,
$\north(\tau)=b$,
$\west(\tau)=c$,
$\south(\tau)=d$
denote respectively the \emph{colors} or \emph{labels} of the right, top, left
and bottom edges of the tile $\tau$.

Let $\Tcal=\{t_0,\dots,t_{m-1}\}$ be a set of Wang tiles.
A configuration $x:\Z^2\to\{0,\dots,m-1\}$ is \emph{valid} with respect to $\Tcal$ if
it assigns a tile in $\Tcal$ to each position of $\Z^2$, so that contiguous edges
of adjacent tiles have the same color, that is,
\begin{align}
    \east(t_{x(\bn)})&=\west(t_{x(\bn+\be_1)})\label{eq:validwangtiling1}\\
    \north(t_{x(\bn)})&=\south(t_{x(\bn+\be_2)})\label{eq:validwangtiling2}
\end{align}
for every $\bn\in\Z^2$ where $\be_1=(1,0)$ and $\be_2=(0,1)$.

Let $\Omega_\Tcal=\{x\in\{0,\dots,m-1\}^{\Z^2}| x \text{ is a valid configuration with respect to }\Tcal\}$.
Together with the shift action $\sigma$ of $\Z^2$,
the set $\Omega_\Tcal$ is a subshift which we called the \emph{Wang shift} of $\Tcal$. 
The subshift $\Omega_\Tcal$ is of finite type satisfying \eqref{eq:SFT}
since there exists a finite set of
forbidden patterns made of all horizontal and vertical dominoes of two tiles
whose adjacent edge is not of the same color.
To a configuration $x\in\Omega_\Tcal$ corresponds a tiling of the plane $\R^2$ by
the tiles $\Tcal$ where the unit square Wang tile $t_{x(\bn)}$ is placed at position $\bn$ for every
$\bn\in\Z^2$. %

A configuration $x\in\Omega_\Tcal$ is \emph{periodic} if there exists
$\bn\in\Z^2\setminus\{0\}$ such that $x=\sigma^\bn(x)$.
A set $\Tcal$ of Wang tiles is \emph{periodic} if there exists a periodic configuration
$x\in\Omega_\Tcal$. 
Originally, Wang thought that every set $\Tcal$ of Wang tiles is periodic 
as soon as $\Omega_\Tcal$ is nonempty \cite{wang_proving_1961}.
Wang noticed that if this statement were true, 
it would imply the existence of an algorithm 
solving the \emph{domino problem}, that is, taking as input a set of Wang tiles
and returning \textit{yes} or \textit{no} whether there exists a valid
configuration with these tiles. 
Berger, a student of Wang, later proved that the domino problem is undecidable
and he also provided a first example of an aperiodic set of Wang tiles
\cite{MR0216954}.
A set $\Tcal$ of Wang tiles is \emph{aperiodic} if
the Wang shift $\Omega_\Tcal$ is a nonempty aperiodic subshift.

\subsection{Wang tiles with stripes}
\label{sec:wang-tiles-with-stripes}

We say that a set of Wang tiles 
$\Tcal\subset I\times J\times I\times J$
\emph{has horizontal stripes}
if there exists a partition $I=I_0\cup I_1$
of the vertical labels such that
\[
    \Tcal\subset 
    \left( I_0\times J\times I_0\times J \right) \cup
    \left( I_1\times J\times I_1\times J \right).
\]
We say that a set of Wang tiles 
$\Tcal\subset I\times J\times I\times J$
\emph{has vertical stripes}
if there exists a partition 
$J=J_0\cup J_1$
of the horizontal labels
such that
\[
    \Tcal\subset 
    \left( I\times J_0\times I\times J_0 \right) \cup
    \left( I\times J_1\times I\times J_1 \right).
\]
When drawing the Wang tiles,
we associate the subset $I_1$ with horizontal stripes
and the subset $J_1$ with vertical stripes,
whereas the subsets $I_0$ and $J_0$ are drawn with no stripe.
If a set of Wang tiles $\Tcal$ has horizontal and vertical stripes
with partitions $I=I_0\cup I_1$ and $J=J_0\cup J_1$,
then we define the following subsets
\begin{align*}
    \Tcal_\square   &= \Tcal\cap \left( I_0\times J_0\times I_0\times J_0 \right),\\
    \Tcal_\boxbar   &= \Tcal\cap \left( I_0\times J_1\times I_0\times J_1 \right),\\
    \Tcal_\boxminus &= \Tcal\cap \left( I_1\times J_0\times I_1\times J_0 \right),\\
    \Tcal_\boxplus  &= \Tcal\cap \left( I_1\times J_1\times I_1\times J_1 \right).
\end{align*}
which form a partition of
$\Tcal = \Tcal_\square \cup \Tcal_\boxbar \cup \Tcal_\boxminus \cup \Tcal_\boxplus$.

\section{A sufficient condition for non-periodicity}
\label{sec:aperiodicityproof}

In this article, we construct a relation between sets of Wang tiles
and planar Euclidean geometry. Under some condition, we can associate 
a quadrilateral in the plane with a set of Wang tiles with stripes.
The relation is given precisely by the following definition
which is used in the statement of most of the results in this article.

\begin{definition}[when a set of Wang tiles with stripes determines a quadrilateral]
    \label{def:determines-a-quadrilateral}
Let $\Tcal\subset I\times J\times I\times J$ be a set of Wang tiles
having horizontal and vertical stripes with partition
$ \Tcal = \Tcal_\square \cup \Tcal_\boxbar \cup \Tcal_\boxminus \cup \Tcal_\boxplus$.
Suppose the labels are encoded into vectors
using two maps $\vv:I\to\R^2$ and $\hh:J\to\R^2$.
If there exist $A,B,C,D\in\R^2$ such that
\begin{equation}
\label{eq:fourbulletpoints}
    \vv(a) + \hh(b) - \vv(c) - \hh(d)
        =
\begin{cases}
    A & \text{ if }\tau\in\Tcal_\square,\\
    B & \text{ if }\tau\in\Tcal_\boxminus,\\
    C & \text{ if }\tau\in\Tcal_\boxbar,\\
    D & \text{ if }\tau \in\Tcal_\boxplus,
\end{cases}
\end{equation}
for every tile 
$\tau=
    \begin{tikzpicture}[scale=.8,auto,baseline=0.30cm]
    \tilelabelinside{white}{0}{0}{a}{b}{c}{d}
    \end{tikzpicture}
\in\Tcal$,
    then we say that the set $\Tcal$ 
    (and the functions $\vv$ and $\hh$)
    \emph{determines the quadrilateral $ABDC$}.
\end{definition}

Note that the quadrilateral can be degenerate as vertices may coincide or three
or more of them can be collinear.
When a set of Wang tiles with stripes determines a quadrilateral, then
we can show that the densities of vertical and horizontal stripes
must satisfy a system of equations related to the vertices of the quadrilateral.

\begin{lemma}\label{lem:doubly-periodic-densities-satisfy-equation}
Let $\Tcal\subset I\times J\times I\times J$ be a set of Wang tiles
having horizontal and vertical stripes with partition
$ \Tcal = \Tcal_\square \cup \Tcal_\boxbar \cup \Tcal_\boxminus \cup \Tcal_\boxplus$.
Suppose that $\Tcal$ determines the quadrilateral $ABDC$ for some $A,B,C,D\in\R^2$.
If there exists a doubly-periodic valid configuration $w:\Z^2\to\Tcal$ where
vertical stripes (\emph{i.e.}, tiles in $\Tcal_\boxbar\cup
\Tcal_\boxplus$) have density $\alpha$, and the horizontal stripes (tiles in
$\Tcal_\boxminus\cup \Tcal_\boxplus$) have density $\beta$, then
the densities satisfy Equation~\eqref{alpha_beta_relations}:
\[
    \left(\begin{smallmatrix} 0\\0 \end{smallmatrix}\right)=
     (1-\alpha)\left((1-\beta)A 
                    +\beta    B \right)
        +\alpha\left((1-\beta)C 
                    +\beta    D \right).
\]
\end{lemma}

\begin{proof}
    Suppose that
    there exists a valid configuration
    $w:\Z^2\to\Tcal$ over the set of Wang tiles $\Tcal$
    which has two linearly independent periods.
    Following the paragraph that follows Proposition~5.9 in \cite{MR3136260},
    we may suppose that the periods are horizontal
    and vertical and of the same length $K\in\N\setminus\{0\}$.

    Let $S_K=\{(i,j)\in\Z^2\colon0\leq i,j< K\}$.
    Since tiles in the same row either all have a horizontal stripe
    or none of them have one (and similarly with columns and vertical stripes),
    the density of tiles belonging to each subset of the partition 
    $ \Tcal = \Tcal_\square \cup \Tcal_\boxbar \cup \Tcal_\boxminus \cup \Tcal_\boxplus$
    is given by the following products:
    \begin{align*}
        \frac{\#\left(w^{-1}(\Tcal_\boxminus)\cap S_K\right)}{\# S_K}
                         &=(1-\alpha)\beta,
        &
        \frac{\#\left(w^{-1}(\Tcal_\square)\cap S_K\right)}{\# S_K}
        &=(1-\alpha)(1-\beta),\\
        \frac{\#\left(w^{-1}(\Tcal_\boxplus)\cap S_K\right)}{\# S_K}
        &=\alpha\beta,    
        &\frac{\#\left(w^{-1}(\Tcal_\boxbar)\cap S_K\right)}{\# S_K}
        &=\alpha(1-\beta).
    \end{align*}
    On the one hand, we have
    \begin{equation}\label{eq:first-equation-in-proof-KxK_domain}
    \begin{aligned}
        &\frac{1}{\# S_K}
        \sum_{n\in S_K}
        \vv(\east(w_{n})) + \hh(\north(w_n)) - \vv(\west(w_n)) - \hh(\south(w_n))
        \\
    &\qquad=
    \frac{1}{\# S_K}
    \Big(
        \left(\begin{smallmatrix}
        A_1\\A_2
        \end{smallmatrix}\right)
        \cdot \#\{n\in S_K\colon w_n\in\Tcal_\square\}  
        +
        \left(\begin{smallmatrix}
        B_1\\B_2
        \end{smallmatrix}\right)
        \cdot \#\{n\in S_K\colon w_n\in\Tcal_\boxminus\}  \\
        &\qquad\qquad+
        \left(\begin{smallmatrix}
        C_1\\C_2
        \end{smallmatrix}\right)
        \cdot \#\{n\in S_K\colon w_n\in\Tcal_\boxbar\}
        +
        \left(\begin{smallmatrix}
        D_1\\D_2
        \end{smallmatrix}\right)
        \cdot \#\{n\in S_K\colon w_n\in\Tcal_\boxplus\}\Big)\\
    &\qquad= 
     \left(\begin{smallmatrix} A_1\\A_2 \end{smallmatrix}\right)     (1-\alpha)(1-\beta)
    +\left(\begin{smallmatrix} B_1\\B_2 \end{smallmatrix}\right) (1-\alpha)\beta
    +\left(\begin{smallmatrix} C_1\\C_2 \end{smallmatrix}\right) \alpha(1-\beta)
    +\left(\begin{smallmatrix} D_1\\D_2 \end{smallmatrix}\right)     \alpha\beta\\
    &\qquad= 
     (1-\alpha)\left[\left(\begin{smallmatrix} A_1\\A_2 \end{smallmatrix}\right)     (1-\beta)
     +\left(\begin{smallmatrix} B_1\\B_2 \end{smallmatrix}\right) \beta\right]
     +\alpha\left[\left(\begin{smallmatrix} C_1\\C_2 \end{smallmatrix}\right) (1-\beta)
         +\left(\begin{smallmatrix} D_1\\D_2  \end{smallmatrix}\right) \beta\right].
\end{aligned}
\end{equation}

    On the other hand, since $w$ is a valid configuration, the following sum is
    telescoping and simplifies to zero because $K$ is a horizontal and vertical
    period of $w$:
    \begin{equation}\label{eq:sum_is_zero_in_KxK_domain}
    \begin{aligned}
        &\sum_{n\in S_K}
        \vv(\east(w_{n})) + \hh(\north(w_n)) - \vv(\west(w_n)) - \hh(\south(w_n))
        \\
        &\qquad=
        \sum_{0\leq i,j <K}
        \vv(\east(w_{i,j})) - \vv(\west(w_{i,j})) + \hh(\north(w_{i,j})) - \hh(\south(w_{i,j}))\\
        &\qquad=
        \sum_{0\leq i,j <K}
        \vv(\east(w_{i,j})) - \vv(\east(w_{i-1,j})) + \hh(\north(w_{i,j})) - \hh(\north(w_{i,j-1}))\\
        &\qquad=
        \sum_{j=0}^{K-1}
        \sum_{i=0}^{K-1}
        \left(
          \east(w_{i,j})
        - \east(w_{i-1,j})
        \right)
        +
        \sum_{i=0}^{K-1}
        \sum_{j=0}^{K-1}
        \left(
          \north(w_{i,j})
        - \north(w_{i,j-1})
        \right)\\
        &\qquad=
        \sum_{j=0}^{K-1}
        \left(
          \east(w_{K-1,j})
        - \east(w_{0-1,j})
        \right)
        +
        \sum_{i=0}^{K-1}
        \left(
          \north(w_{i,K-1})
        - \north(w_{i,0-1})
        \right)\\
        &\qquad=
        \sum_{j=0}^{K-1} 0
        +
        \sum_{i=0}^{K-1} 0
        =0.
    \end{aligned}
    \end{equation}
Combining \eqref{eq:first-equation-in-proof-KxK_domain} and \eqref{eq:sum_is_zero_in_KxK_domain}, 
we obtain the equation given in \eqref{alpha_beta_relations}.
\end{proof}

Under some conditions on determinants which are satisfied in many cases (in
fact all examples described in Section~\ref{sec:applications} satisfy these
conditions), Equation~\eqref{alpha_beta_relations} has a solution which is easy
to describe.

\begin{lemma}\label{lem:det-conditions-implies-alpha-equal-beta-are-quadratic}
    Suppose that $\alpha,\beta\in\C$ is a solution
    to Equation~\eqref{alpha_beta_relations}
    for some
    $\left(\begin{smallmatrix} A_1\\A_2 \end{smallmatrix}\right),
     \left(\begin{smallmatrix} B_1\\B_2 \end{smallmatrix}\right),
     \left(\begin{smallmatrix} C_1\\C_2 \end{smallmatrix}\right),
     \left(\begin{smallmatrix} D_1\\D_2 \end{smallmatrix}\right)\in\R^2$.
    If
    \begin{equation}
    \label{eq:determinantcondition}
    \det\left(\begin{smallmatrix} A_1&D_1\\A_2&D_2 \end{smallmatrix}\right)=0
        \quad
        \text{and}
        \quad
    \det\left(\begin{smallmatrix} B_1&D_1\\B_2&D_2 \end{smallmatrix}\right)
    =-\det\left(\begin{smallmatrix} C_1&D_1\\C_2&D_2 \end{smallmatrix}\right)\neq0,
    \end{equation}
    then
    $\alpha=\beta$ and $\alpha$ is a solution of the quadratic equation
    \begin{equation}\label{eq:quadratic-equations-ABCD}
    \begin{aligned}
        0 =
        (A_i-B_i-C_i+D_i)x^2+(B_i+C_i-2A_i)x + A_i
    \end{aligned}
    \end{equation}
whose discriminant is $\Delta=(B_i+C_i)^2-4A_iD_i$
for both $i=1$ and $i=2$.
\end{lemma}

\begin{proof}
Multiplying Equation~\eqref{alpha_beta_relations} by
the matrix
$\left(\begin{smallmatrix} D_2&0\\0&D_1 \end{smallmatrix}\right)$
we obtain the following two equations:
\[
    \begin{cases}
        0=
        (1-\alpha)\left(D_2A_1(1-\beta) +D_2B_1 \beta\right)
        +\alpha   \left(D_2C_1(1-\beta) +D_2D_1 \beta\right),\\
        0=
        (1-\alpha)\left(D_1A_2(1-\beta)  +D_1B_2 \beta\right)
        +\alpha   \left(D_1C_2 (1-\beta) +D_1D_2   \beta\right).
    \end{cases}
\]
Since
$\det\left(\begin{smallmatrix} A_1&D_1\\A_2&D_2 \end{smallmatrix}\right)=0$,
the difference of the two equations yields
    \begin{align*}
        0
        &= 
        (D_2B_1-D_1B_2)(1-\alpha)\beta
        +
        (D_2C_1-D_1C_2)\alpha(1-\beta) \\
        &= 
        \det\left(\begin{smallmatrix} B_1&D_1\\B_2&D_2 \end{smallmatrix}\right)
        (1-\alpha)\beta
        +
        \det\left(\begin{smallmatrix} C_1&D_1\\C_2&D_2 \end{smallmatrix}\right)
        \alpha(1-\beta) \\
        &= 
        \det\left(\begin{smallmatrix} B_1&D_1\\B_2&D_2 \end{smallmatrix}\right)
        \left( \beta-\alpha\beta - \alpha-\alpha\beta\right)
        = 
        \det\left(\begin{smallmatrix} B_1&D_1\\B_2&D_2 \end{smallmatrix}\right)
        ( \beta - \alpha)
    \end{align*}
        Since $\det\left(\begin{smallmatrix} B_1&D_1\\B_2&D_2 \end{smallmatrix}\right)\neq0$,
    we must have equality of the densities, that is $\alpha=\beta$, and
    \begin{align*}
        0 &=
     A_i (1-\alpha)^2
    +(B_i+C_i) (1-\alpha)\alpha
        +D_i \alpha^2,\\
        &=
        (A_i-B_i-C_i+D_i)\alpha^2+(B_i+C_i-2A_i)\alpha + A_i
    \end{align*}
whose discriminant is $\Delta=(B_i+C_i)^2-4A_iD_i$
for both $i=1$ and $i=2$.
\end{proof}

The solutions to Equation~\eqref{alpha_beta_relations}
can be interpreted geometrically as the lines passing through the origin and
tangent to a parabola (see Figure~\ref{fig:the-parabola-for-Ammann-case}).
The complete description of the solutions for all cases
is postponed to Section~\ref{sec:solution-to-alpha-beta-equation}.
However, Lemma~\ref{lem:det-conditions-implies-alpha-equal-beta-are-quadratic}
is already sufficient to conclude nonperiodicity of many sets of Wang tiles.

\begin{THEOREMI}
    \StatementMainTHEOREMI
\end{THEOREMI}

\begin{proof}
    By contradiction, assume there exists a 
    valid configuration $w:\Z^2\to\Tcal$ over the set of Wang tiles $\Tcal$
    which is periodic.
    From a standard argument \cite[Proposition~5.9]{MR3136260},
    there exists another valid configuration
    $w':\Z^2\to\Tcal$ over the set of Wang tiles $\Tcal$
    which has two linearly independent periods.
    In particular, the densities of all tiles in the configuration $w'$ exist,
    as well as the density $\alpha$ of tiles with vertical stripes (tiles in
    $\Tcal_\boxbar\cup\Tcal_\boxplus$)
    and the density $\beta$ of tiles with horizontal stripes (tiles in
    $\Tcal_\boxminus\cup\Tcal_\boxplus$) in $w'$.
    Since $w'$ is doubly-periodic, these densities must be rational numbers.
    Thus, $\alpha,\beta\in\Q$.

    From Lemma~\ref{lem:doubly-periodic-densities-satisfy-equation},
    the densities $\alpha$ and $\beta$ satisfy
    Equation~\eqref{alpha_beta_relations}.
    From Lemma~\ref{lem:det-conditions-implies-alpha-equal-beta-are-quadratic},
    we have $\alpha=\beta$ and
    $\alpha$ and $\beta$
    also satisfy the quadratic Equation~\eqref{eq:quadratic-equations-ABCD}
    for both $i\in\{1,2\}$.
    Since there exists $i\in\{1,2\}$ such that
    $(B_i+C_i)^2-4A_iD_i>0$ is not a perfect square,
    then at least one of the two equations is of degree 2 
    (because equality $A_i+D_i= B_i+C_i$ implies that the discriminant is 
    $(A_i-D_i)^2$ which is a perfect square)
    and has no rational solution.
    In particular, $\alpha=\beta\notin\Q$.
    This is a contradiction and we conclude that 
    the set of Wang tiles $\Tcal$ admits no periodic tiling.
\end{proof}

\begin{corollary}\label{cor:Ammann-non-periodic}
    The set $\Tcal_{Ammann}$ 
    \begin{center}
        \includegraphics{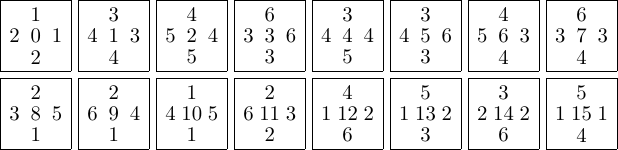}
    \end{center}
    of 16 Wang tiles admits no periodic tiling.
\end{corollary}

\begin{proof}
    As observed in the introduction,
the Ammann set of Wang tiles determines the quadrilateral $ABDC$ with
$\ABCD{A_1}{A_2}{B_1}{B_2}{C_1}{C_2}{D_1}{D_2}
=\ABCD{1}{0}{0}{1}{1}{-1}{-1}{0}$
satisfying
    \begin{align*}
    \det\left(\begin{smallmatrix} A_1&D_1\\A_2&D_2 \end{smallmatrix}\right)
        &= \det\left(\begin{smallmatrix} 1&-1\\0&0 \end{smallmatrix}\right) =0,\\
    \det\left(\begin{smallmatrix} B_1&D_1\\B_2&D_2 \end{smallmatrix}\right)
        &=\det\left(\begin{smallmatrix} 0&-1\\1&0 \end{smallmatrix}\right)
        =1 \neq0,\\
    \det\left(\begin{smallmatrix} C_1&D_1\\C_2&D_2 \end{smallmatrix}\right)
       &=\det\left(\begin{smallmatrix} 1&-1\\-1&0 \end{smallmatrix}\right)
       =-1
      =-\det\left(\begin{smallmatrix} B_1&D_1\\B_2&D_2 \end{smallmatrix}\right).
    \end{align*}
The discriminant
$\Delta=(B_1+C_1)^2-4A_1D_1=(0+1)^2+4=5$
is positive and is not a perfect square.
Thus, we conclude from Theorem~\ref{thm:not-periodic} that
every valid configuration over these Wang tiles is non-periodic.
\end{proof}

\begin{remark}
     Let
     $A=\left(\begin{smallmatrix} A_1\\A_2 \end{smallmatrix}\right)$,
     $B=\left(\begin{smallmatrix} B_1\\B_2 \end{smallmatrix}\right)$,
     $C=\left(\begin{smallmatrix} C_1\\C_2 \end{smallmatrix}\right)$ and
     $D=\left(\begin{smallmatrix} D_1\\D_2 \end{smallmatrix}\right)$.
     The hypothesis on the determinants in the statement of the theorem
     means that $A$, $D$ and $O=(0,0)$ are collinear
     and that the points $B$ and $C$ are equidistant to the line $AOD$
     but on two different side of that line.
     Also, the quadrilateral $OBDC$ has area
     $\det\left(\begin{smallmatrix} B_1&D_1\\B_2&D_2 \end{smallmatrix}\right)
    =-\det\left(\begin{smallmatrix} C_1&D_1\\C_2&D_2 \end{smallmatrix}\right)\neq0$.
\end{remark}

Observe that the Theorem holds but is inconclusive if we use the trivial partition 
$I_0=J_0=\varnothing$ or the trivial partition  $I_1=J_1=\varnothing$.
In these cases, the hypotheses of the theorem are not satisfied, since
$I_0=J_0=\varnothing$ implies $\Tcal=\Tcal_\boxplus$ implies $\alpha=\beta=1$ implies $D_1=D_2=0$,
while
$I_1=J_1=\varnothing$ implies $\Tcal=\Tcal_\square$ implies $\alpha=\beta=0$ implies $A_1=A_2=0$.
The discriminant is a perfect square in these cases.

    In the section that follows,
    we apply Theorem~\ref{thm:not-periodic}
    to conclude nonperiodicity of many well-known sets of Wang tiles.

\section{Applications: new proofs of non-periodicity of known sets of Wang tiles}
\label{sec:applications}

    Theorem~\ref{thm:not-periodic} may be used to prove the non-existence of
    periodic tilings for many other well-known sets of Wang tiles.
    We apply this method to 
    the encoding of the Penrose tilings into 24 Wang tiles
    and to the family of metallic mean Wang tiles.
    First, we present an encoding of Wang tiles as matrices which we use thereafter.

\subsection{Wang tiles encoded as matrices}
\label{sec:wang-tiles-encoded-as-nx4-matrices}

Let $\Tcal\subset I\times J\times I\times J$ be a set of Wang tiles
having horizontal and vertical stripes
with partitions $I=I_0\cup I_1$ and $J=J_0\cup J_1$
and
$ \Tcal = \Tcal_\square \cup \Tcal_\boxbar \cup \Tcal_\boxminus \cup \Tcal_\boxplus$.
Suppose the labels are encoded into vectors
using two maps $\vv$ and $\hh$:
\[
\vv:I\to\R^n
\qquad
\text{ and }
\qquad
\hh:J\to\R^n.
\]
Let $\I_{A}$ denote the indicator function of a set $A$ defined as
\[
    \I_A(x)=
    \begin{cases}
        1 & \text{ if } x\in A,\\
        0 & \text{ if } x\notin A.
    \end{cases}
\]
In this work, 
we encode a Wang tile $\tau=(a,b,c,d)$
as a $(n+1)\times 4$ matrix whose 
first row encodes the presence of stripes (vertical label in $I_1$ or
horizontal label in $J_1$) and whose columns of the remaining part are the
images of its labels under the maps $\vv$ and $\hh$:
\[
    M_\tau 
    :=
    M^{\vv,\hh}(\tau)
    =
    \left(
    \begin{array}{c|c|c|c}
        \I_{I_1}(a) & \I_{J_1}(b) & \I_{I_1}(c) &  \I_{J_1}(d)\\
        &&&\\[-3mm]
        \hline
        &&&\\[-3mm]
        \vv(a) & \hh(b) & \vv(c) & \hh(d)\\
    \end{array}
    \right)
    =
    \left(
    \begin{array}{c|c|c|c}
        a_0 & b_0 & c_0 & d_0\\
        a_1 & b_1 & c_1 & d_1\\
        \vdots & \vdots & \vdots & \vdots\\
        a_n & b_n & c_n & d_n
    \end{array}
    \right).
\]

If $e_1=(1,0,\dots,0)\in\R^{n+1}$, the partition
$ \Tcal = \Tcal_\square \cup \Tcal_\boxbar \cup \Tcal_\boxminus \cup \Tcal_\boxplus$
is equivalently defined by the first row of the matrices:
\begin{align*}
    \Tcal_\square   &= \{\tau\in\Tcal\colon e_1M_\tau= (0,0,0,0)\},\\
    \Tcal_\boxbar   &= \{\tau\in\Tcal\colon e_1M_\tau= (0,1,0,1)\},\\
    \Tcal_\boxminus &= \{\tau\in\Tcal\colon e_1M_\tau= (1,0,1,0)\},\\
    \Tcal_\boxplus  &= \{\tau\in\Tcal\colon e_1M_\tau= (1,1,1,1)\}.
\end{align*}

Now, suppose that $n=2$ and let $U=\IdTwoThree$.
In this situation, the last two lines of the matrix $M_\tau$ 
are the image of the labels under the maps $\vv$ and $\hh$:
\[
    U\cdot M_\tau 
    =
    \left(
    \begin{array}{c|c|c|c}
        \vv(a) & \hh(b) & \vv(c) & \hh(d)\\
    \end{array}
    \right).
\]
Therefore,
the left hand side of Equation~\eqref{eq:fourbulletpoints}
is equal to
\begin{equation}\label{eq:matrix-notation-of-va-hb-vc-hd}
    U\cdot M_\tau \cdot \left(\begin{smallmatrix} 1\\1\\-1\\-1 \end{smallmatrix}\right)
        =
    \vv(a) + \hh(b) - \vv(c) - \hh(d).
\end{equation}
We use this matrix notation of Wang tiles in
this section in the proof of 
Theorem~\ref{thm:Penrose-non-periodic},
Theorem~\ref{thm:metallic-mean-tiles-non-periodic}
and later in the proof of
Proposition~\ref{prop:sufficient-condition-quadratic-equation}.

Validity of configurations translates into the following relation on the matrix
encoding of Wang tiles.

\begin{remark}\label{rem:matrix-encoding-of-valid-configuration}
If $w:\Z^2\to\Tcal$ is a valid configuration over the tiles $\Tcal$,
then
\[
    M^{\vv,\hh}(w_n)
    \left(
    \begin{array}{c}
        1 \\ 0 \\ 0 \\ 0
    \end{array}
    \right)
    =
    M^{\vv,\hh}(w_{n+e_1})
    \left(
    \begin{array}{c}
        0 \\ 0 \\ 1 \\ 0
    \end{array}
    \right)
    \qquad
    \text{ and }
    \qquad
    M^{\vv,\hh}(w_{n})
    \left(
    \begin{array}{c}
        0 \\ 1 \\ 0 \\ 0
    \end{array}
    \right)
    =
    M^{\vv,\hh}(w_{n+e_2})
    \left(
    \begin{array}{c}
        0 \\ 0 \\ 0 \\ 1
    \end{array}
    \right)
\]
for every $n\in\Z^2$
where $\{e_1,e_2\}$ denotes the canonical base of $\Z^2$.
\end{remark}

\subsection{Penrose encoded into 24 Wang tiles}
\label{section:24 Penrose}

As observed in \cite[Exercise 11.1.2]{MR857454},
Penrose tilings can be encoded into a set $\Tcal_{24}$ of 24 Wang tiles.
This was studied in more details in
\cite{jang_directional_2021,jang_directional_2025}.
Here, we use the labelling of the Wang tiles proposed in this blog post\footnote{\url{http://www.slabbe.org/blogue/2025/05/on-the-bifurcation-diagram-proposed-by-jang-and-robinson/}}.

Using the sufficient condition in Theorem~\ref{thm:not-periodic},
we may prove that the set $\Tcal_{24}$ admits no periodic tiling of the plane.

\begin{theorem}\label{thm:Penrose-non-periodic}
    The set $\Tcal_{24}$ 
    \begin{center}
        \includegraphics{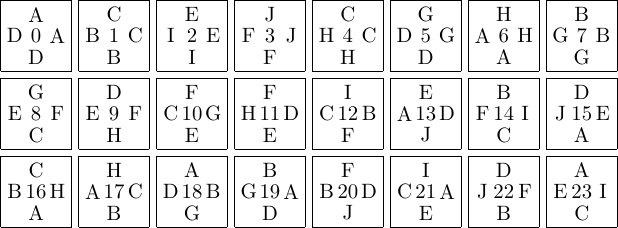}
    \end{center}
    of 24 Wang tiles encoding Penrose tilings
    admits no periodic tiling.
\end{theorem}

\begin{proof}
We map the vertical and horizontal labels to vectors as follows:
\[
    \vv:\left\{\begin{array}{cccc}
            A\mapsto (1, 3, 1), &F\mapsto (0, 1, 1), \\
            B\mapsto (1, 3, 2), &G\mapsto (1, 4, 2), \\
            C\mapsto (1, 0, 0), &H\mapsto (1, 2, 1), \\
            D\mapsto (1, 6, 3), &I\mapsto (0, 0, 0), \\
            E\mapsto (0, 3, 2), &J\mapsto (0, 4, 3),
    \end{array}\right.
    \quad
    \text{ and }
    \quad
    \hh:\left\{\begin{array}{cccc}
            A\mapsto (1, 2, 2), &F\mapsto (0, 1, 2), \\
            B\mapsto (1, 0, 1), &G\mapsto (1, 1, 1), \\
            C\mapsto (1, 1, 3), &H\mapsto (1, 1, 2), \\
            D\mapsto (1, 1, 0), &I\mapsto (0, 2, 3), \\
            E\mapsto (0, 1, 1), &J\mapsto (0, 0, 0).
    \end{array}\right.
\]
Using the vector $e=(1,0,0)$, the partition of the set of Wang tiles is:
\begin{align*}
    \Tcal_\square   &= \{2, 3\}, \\
    \Tcal_\boxminus &= \{10, 11, 12, 13, 20, 21\}, \\
    \Tcal_\boxbar   &= \{8, 9, 14, 15, 22, 23\}, \\
    \Tcal_\boxplus  &= \{0, 1, 4, 5, 6, 7, 16, 17, 18, 19\}.
\end{align*}
Now take $U=\IdTwoThree$.
For each of the Wang tiles, we have
\begin{itemize}
\item if $\tau \in\Tcal_\square$,  then
    \begin{align*}
        M_\tau&\in\{M_{2},M_{3}\}
        =
        \left\{
            \left(\begin{smallmatrix}%
            0 & 0 & 0 & 0 \\
            3 & 1 & 0 & 2 \\
            2 & 1 & 0 & 3
            \end{smallmatrix}\right),
            \left(\begin{smallmatrix}%
            0 & 0 & 0 & 0 \\
            4 & 0 & 1 & 1 \\
            3 & 0 & 1 & 2
            \end{smallmatrix}\right)
        \right\}
    \end{align*}
    which all satisfy $\IdTwoThree
        M_\tau \OneOneMinusoneMinusOne=
        \left(\begin{smallmatrix}%
        2\\
        0
        \end{smallmatrix}\right)$,
\item if $\tau\in\Tcal_\boxminus$,  then
    \begin{align*}
        M_\tau&\in\{M_{10},M_{11},M_{12},M_{13},M_{20},M_{21}\}\\
        &=\left\{
            \left(\begin{smallmatrix}%
            1 & 0 & 1 & 0 \\
            4 & 1 & 0 & 1 \\
            2 & 2 & 0 & 1
            \end{smallmatrix}\right),
            \left(\begin{smallmatrix}%
            1 & 0 & 1 & 0 \\
            6 & 1 & 2 & 1 \\
            3 & 2 & 1 & 1
            \end{smallmatrix}\right),
            \left(\begin{smallmatrix}%
            1 & 0 & 1 & 0 \\
            3 & 2 & 0 & 1 \\
            2 & 3 & 0 & 2
            \end{smallmatrix}\right),
            \left(\begin{smallmatrix}%
            1 & 0 & 1 & 0 \\
            6 & 1 & 3 & 0 \\
            3 & 1 & 1 & 0
            \end{smallmatrix}\right),
            \left(\begin{smallmatrix}%
            1 & 0 & 1 & 0 \\
            6 & 1 & 3 & 0 \\
            3 & 2 & 2 & 0
            \end{smallmatrix}\right),
            \left(\begin{smallmatrix}%
            1 & 0 & 1 & 0 \\
            3 & 2 & 0 & 1 \\
            1 & 3 & 0 & 1
            \end{smallmatrix}\right)
        \right\}
    \end{align*}
    which all satisfy $\IdTwoThree
        M_\tau \OneOneMinusoneMinusOne=
        \left(\begin{smallmatrix}%
        4\\
        3
        \end{smallmatrix}\right)$,
\item if $\tau\in\Tcal_\boxbar$,  then
    \begin{align*}
        M_\tau&\in\{M_{8},M_{9},M_{14},M_{15},M_{22},M_{23}\}\\
        &=
        \left\{
            \left(\begin{smallmatrix}%
            0 & 1 & 0 & 1 \\
            1 & 1 & 3 & 1 \\
            1 & 1 & 2 & 3
            \end{smallmatrix}\right),
            \left(\begin{smallmatrix}%
            0 & 1 & 0 & 1 \\
            1 & 1 & 3 & 1 \\
            1 & 0 & 2 & 2
            \end{smallmatrix}\right),
            \left(\begin{smallmatrix}%
            0 & 1 & 0 & 1 \\
            0 & 0 & 1 & 1 \\
            0 & 1 & 1 & 3
            \end{smallmatrix}\right),
            \left(\begin{smallmatrix}%
            0 & 1 & 0 & 1 \\
            3 & 1 & 4 & 2 \\
            2 & 0 & 3 & 2
            \end{smallmatrix}\right),
            \left(\begin{smallmatrix}%
            0 & 1 & 0 & 1 \\
            1 & 1 & 4 & 0 \\
            1 & 0 & 3 & 1
            \end{smallmatrix}\right),
            \left(\begin{smallmatrix}%
            0 & 1 & 0 & 1 \\
            0 & 2 & 3 & 1 \\
            0 & 2 & 2 & 3
            \end{smallmatrix}\right)
        \right\}
    \end{align*}
    which all satisfy $\IdTwoThree
        M_\tau \OneOneMinusoneMinusOne=
        \left(\begin{smallmatrix}%
        -2\\
        -3
        \end{smallmatrix}\right)$,
\item if $\tau\in\Tcal_\boxplus$,  then
    \begin{align*}
        M_\tau&\in\{M_{0},M_{1},M_{4},M_{5},M_{6},M_{7},M_{16},M_{17},M_{18},M_{19}\}\\
        &=
        \left\{
            \begin{array}{l}
            \left(\begin{smallmatrix}%
            1 & 1 & 1 & 1 \\
            3 & 2 & 6 & 1 \\
            1 & 2 & 3 & 0
            \end{smallmatrix}\right),
            \left(\begin{smallmatrix}%
            1 & 1 & 1 & 1 \\
            0 & 1 & 3 & 0 \\
            0 & 3 & 2 & 1
            \end{smallmatrix}\right),
            \left(\begin{smallmatrix}%
            1 & 1 & 1 & 1 \\
            0 & 1 & 2 & 1 \\
            0 & 3 & 1 & 2
            \end{smallmatrix}\right),
            \left(\begin{smallmatrix}%
            1 & 1 & 1 & 1 \\
            4 & 1 & 6 & 1 \\
            2 & 1 & 3 & 0
            \end{smallmatrix}\right),
            \left(\begin{smallmatrix}%
            1 & 1 & 1 & 1 \\
            2 & 1 & 3 & 2 \\
            1 & 2 & 1 & 2
            \end{smallmatrix}\right),\\[3mm]
            \left(\begin{smallmatrix}%
            1 & 1 & 1 & 1 \\
            3 & 0 & 4 & 1 \\
            2 & 1 & 2 & 1
            \end{smallmatrix}\right),
            \left(\begin{smallmatrix}%
            1 & 1 & 1 & 1 \\
            2 & 1 & 3 & 2 \\
            1 & 3 & 2 & 2
            \end{smallmatrix}\right),
            \left(\begin{smallmatrix}%
            1 & 1 & 1 & 1 \\
            0 & 1 & 3 & 0 \\
            0 & 2 & 1 & 1
            \end{smallmatrix}\right),
            \left(\begin{smallmatrix}%
            1 & 1 & 1 & 1 \\
            3 & 2 & 6 & 1 \\
            2 & 2 & 3 & 1
            \end{smallmatrix}\right),
            \left(\begin{smallmatrix}%
            1 & 1 & 1 & 1 \\
            3 & 0 & 4 & 1 \\
            1 & 1 & 2 & 0
            \end{smallmatrix}\right)
            \end{array}
        \right\}
    \end{align*}
    which all satisfy $\IdTwoThree
        M_\tau \OneOneMinusoneMinusOne=
        \left(\begin{smallmatrix}%
        -2\\
        0
        \end{smallmatrix}\right)$.
\end{itemize}
Thus, we have that the set of Wang tiles determines the quadrilateral $ABDC$
with
$\ABCD{A_1}{A_2}{B_1}{B_2}{C_1}{C_2}{D_1}{D_2}
=\ABCD{2}{0}{4}{3}{-2}{-3}{-2}{0}$
satisfying
    \begin{align*}
    \det\left(\begin{smallmatrix} A_1&D_1\\A_2&D_2 \end{smallmatrix}\right)
        &= \det\left(\begin{smallmatrix} 2&-2\\0&0 \end{smallmatrix}\right) =0,\\
    \det\left(\begin{smallmatrix} B_1&D_1\\B_2&D_2 \end{smallmatrix}\right)
        &=\det\left(\begin{smallmatrix} 4&-2\\3&0 \end{smallmatrix}\right)
        =6 \neq0,\\
    \det\left(\begin{smallmatrix} C_1&D_1\\C_2&D_2 \end{smallmatrix}\right)
       &=\det\left(\begin{smallmatrix} -2&-2\\-3&0 \end{smallmatrix}\right)
       =-6
      =-\det\left(\begin{smallmatrix} B_1&D_1\\B_2&D_2 \end{smallmatrix}\right).
    \end{align*}
The discriminant
$\Delta=(B_1+C_1)^2-4A_1D_1=(4-2)^2+16=20$
is positive and is not a perfect square.
Thus, we conclude from Theorem~\ref{thm:not-periodic} that
every valid configuration over these Wang tiles is non-periodic.
\end{proof}

\subsection{Metallic mean Wang tiles}
\label{sec:metallic-mean-computer-chip}

A family of aperiodic set of Wang tiles was introduced in \cite{MR4944989,MR4963140}
associated with metallic mean numbers. 
Although it was proved in \cite{MR4963140} that tilings by metallic mean Wang tiles
are associated with irrational rotations on a torus, the proof of non-periodicity
of these tilings was deduced from its self-similarity and its minimality
proved in \cite{MR4944989}. A less tedious proof of non-periodicity of these tilings 
based on a short Kari-like argument \cite{MR1417578} was left open. 
In this section, using Theorem~\ref{thm:not-periodic}, we obtain
a short proof of non-periodicity of every tiling of the plane by metallic mean Wang tiles.

We recall from \cite{MR4963140} the definition of the metallic mean Wang tiles
as instances of a computer chip.
For every integer $n\geq1$, let $V_n\subset\N^3$ be a finite subset of vectors
\[
    V_n = \{(v_0,v_1,v_2)\in\N^3\colon 
                       0\leq v_0\leq v_1\leq 1 
                       \text{ and } 
                       v_1\leq v_2\leq n+1\}
\]
with nondecreasing entries where the middle entry is at most 1.
For every integer $n\geq1$, a function 
\[
\begin{array}{rccl}
    \theta_n:&V_n\times V_n & \to & \Z^3\\
    &(u_0,u_1,u_2), (v_0,v_1,v_2) 
    & \mapsto & (r_0,r_1,r_2),
\end{array}
\]
taking two vectors as input and returning one vector
is defined by the rule
\begin{equation}\label{eq:defintion-map-theta}
\left\{
\begin{array}{ll}
    &r_0=u_0,\\
    &r_1=\begin{cases}
            v_2-n       & \text{ if } u_0 = 0,\\
            1  & \text{ if } u_0 = 1,
         \end{cases}\\
    &r_2=\begin{cases}
            v_1+u_0             & \text{ if } v_0 = 0,\\
            u_2+1  & \text{ if } v_0 = 1.
         \end{cases}
\end{array}
\right.
\end{equation}
The set of metallic mean Wang tiles can be defined as the set of instances of a
square-shape computer chip outputing the value of map $\theta_n$ on the right
and on the top depending on the left input $u$ and right input $v$ \cite{MR4963140}.  
More precisely, it is the set
\begin{equation}\label{eq:Ccal}
    \Ccal_n=
    \left\{
    \raisebox{-9.5mm}{
    \begin{tikzpicture}[auto]
        \tile{white}{0}{0}{\theta_n(u,v)}{\theta_n(v,u)}{u=(u_0,u_1,u_2)}{v=(v_0,v_1,v_2)}
    \end{tikzpicture}}
    \middle|\,
        u,v\in V_n
        \text{ such that }
        \theta_n(u,v),\theta_n(v,u)\in V_n
    \right\}
\end{equation}
which is the finite set of all possible instances of the $\theta_n$-chip.

\begin{theorem}\label{thm:metallic-mean-tiles-non-periodic}
    For every $n\geq 1$, the set $\Ccal_n$ of metallic mean Wang tiles 
    admits no periodic tiling.
\end{theorem}

\begin{proof}
    We use the description of the metallic mean Wang tiles as instances of a
    computer chip recalled in \eqref{eq:Ccal}.
    But, when encoding the Wang tiles into matrices, we use the vectors in
    $V_n$ for the vertical labels:
    \[
        \vv:(u_0,u_1,u_2) \mapsto (1-u_0,u_1,u_2),
    \]
    and we swap the last two entries of the vectors in $V_n$ for the horizontal labels:
    \[
        \hh:(v_0,v_1,v_2) \mapsto (1-v_0,v_2,v_1).
    \]
Thus, we have the following matrix representation of the Wang tiles
\begin{align*}
    M_{(\theta_n(u,v),\theta_n(v,u),u,v)}
    &=
    \left(
    \begin{array}{c|c|c|c}
        \vv(\theta_n(u,v)) & \hh(\theta_n(v,u)) & \vv(u) & \hh(v)\\
    \end{array}
    \right).
\end{align*}
Now take $U=\IdTwoThree$.
Depending on the value of $u_0,v_0\in\{0,1\}$, we have the following four subsets
of Wang tiles:
\begin{itemize}
\item if $\tau \in\Tcal_\square$,   then 
    \[
        M_\tau
        = \left(\begin{array}{cccc}
            1-u_0   & 1-v_0   & 1-u_0 & 1-v_0\\
            1     & v_2+1 & u_1 & v_2\\
            u_2+1 & 1     & u_2 & v_1
        \end{array}\right)
        = \left(\begin{array}{cccc}
            0 & 0 & 0 & 0\\
            1 & v_2+1 & 1 & v_2\\
            u_2+1 & 1 & u_2 & 1
        \end{array}\right)
    \]
    which satisfy
    $ \IdTwoThree M_\tau \OneOneMinusoneMinusOne=
        \left(\begin{smallmatrix}%
        1\\
        1
        \end{smallmatrix}\right)$.
\item if $\tau \in\Tcal_\boxbar$,   then 
    \[
        M_\tau
        = \left(\begin{array}{cccc}
            1-u_0   & 1-v_0     & 1-u_0 & 1-v_0\\
            1     & v_2+1   & u_1 & v_2\\
            v_1+u_0 & u_2-n & u_2 & v_1
        \end{array}\right)
        = \left(\begin{array}{cccc}
            0   & 1     & 0 & 1\\
            1     & v_2+1   & 1 & v_2\\
            v_1+1 & u_2-n & u_2 & v_1
        \end{array}\right)
    \]
    which satisfy
        $ \IdTwoThree M_\tau \OneOneMinusoneMinusOne=
        \left(\begin{smallmatrix}%
        1\\
        1-n
        \end{smallmatrix}\right)$,
\item if $\tau \in\Tcal_\boxminus$, then 
    \[
        M_\tau
        = \left(\begin{array}{cccc}
            1-u_0 & 1-v_0     & 1-u_0 & 1-v_0\\
            v_2-n & u_1+v_0 & u_1 & v_2\\
            u_2+1 & 1 & u_2 & v_1
        \end{array}\right)
        = \left(\begin{array}{cccc}
            1 & 0     & 1 & 0\\
            v_2-n & u_1+1 & u_1 & v_2\\
            u_2+1 & 1 & u_2 & 1
        \end{array}\right)
    \]
    which satisfy
        $\IdTwoThree M_\tau \OneOneMinusoneMinusOne=
        \left(\begin{smallmatrix}%
        1-n\\
        1
        \end{smallmatrix}\right)$,
\item if $\tau \in\Tcal_\boxplus$,  then 
    \[
        M_\tau
        = \left(\begin{array}{cccc}
            1-u_0   & 1-v_0     & 1-u_0 & 1-v_0\\
            v_2-n & u_1+v_0   & u_1 & v_2\\
            v_1+u_0 & u_2-n & u_2 & v_1
        \end{array}\right)
        = \left(\begin{array}{cccc}
            1   & 1     & 1 & 1\\
            v_2-n & u_1 & u_1 & v_2\\
            v_1 & u_2-n & u_2 & v_1
        \end{array}\right)
    \]
    which satisfy
        $ \IdTwoThree M_\tau \OneOneMinusoneMinusOne=
        \left(\begin{smallmatrix}%
        -n\\
        -n
        \end{smallmatrix}\right)$,
\end{itemize}
Thus, we have that the set of Wang tiles determines the quadrilateral $ABDC$ where
$\ABCD{A_1}{A_2}{B_1}{B_2}{C_1}{C_2}{D_1}{D_2}
=\ABCD{1}{1}{1}{1-n}{1-n}{1}{-n}{-n}$
satisfies
    \begin{align*}
    \det\left(\begin{smallmatrix} A_1&D_1\\A_2&D_2 \end{smallmatrix}\right)
        &= \det\left(\begin{smallmatrix} 1&-n\\1&-n \end{smallmatrix}\right) =0,\\
    \det\left(\begin{smallmatrix} B_1&D_1\\B_2&D_2 \end{smallmatrix}\right)
       &=\det\left(\begin{smallmatrix} 1&-n\\1-n&-n \end{smallmatrix}\right)
           =-n+n(1-n) =-n^2 \neq0,\\
    \det\left(\begin{smallmatrix} C_1&D_1\\C_2&D_2 \end{smallmatrix}\right)
        &=\det\left(\begin{smallmatrix} 1-n&-n\\1&-n \end{smallmatrix}\right)
        =n^2
      =-\det\left(\begin{smallmatrix} B_1&D_1\\B_2&D_2 \end{smallmatrix}\right)
    \end{align*}
since $n\geq1$ is a positive integer.
The discriminant
$\Delta=(B_1+C_1)^2-4A_1D_1=n^2+4$
is positive and is not a perfect square when $n\geq1$.
Thus, we conclude from Theorem~\ref{thm:not-periodic} that
every valid configuration over these metallic mean Wang tiles is non-periodic.
\end{proof}

It is interesting to observe that the maps $\vv$ and $\hh$ used in the proof of
Theorem~\ref{thm:metallic-mean-tiles-non-periodic} almost correspond to the
labels of the metallic mean Wang tiles as 3-dimensional vectors.
This is all that was missing in order to provide a positive answer
to Question 8.5 in \cite{MR4963140}.

\section{Finding the quadrilateral determined by a set of Wang tiles with stripes}  
\label{sec:the-kernel-method}

In this section, we present a method to compute 
the maps $\vv$, $\hh$ and the quadrilateral $ABDC$ for a set
of Wang tiles with horizontal and vertical stripes.
This method was used for determining the maps $\vv$ and $\hh$ 
and the quadrilaterals for Ammann and Penrose examples in the previous section.
The strategy is based on linear algebra.

\subsection{The pullback of two linear transformations}
\label{sec:pullback-of-linear-transformations}

Let $U_1$, $U_2$ and $V$ be vector spaces
and $\lambda_1:U_1\to V$ and $\lambda_2:U_2\to V$
be two linear transformations.
Together with the natural projections
$\pi_1:U_1\times U_2\to U_1$
and
$\pi_2:U_1\times U_2\to U_2$,
we have the following diagram
\[
\begin{tikzcd}
U_2 \arrow[r, "\lambda_2"]  & V \\
U_1\times U_2 \arrow[u, "\pi_2"]\arrow[r, "\pi_1"] & U_1 \arrow[u, "\lambda_1"]
\end{tikzcd}.
\]
Let $F:U_1 \times U_2 \to V$ be the map defined as
    $F = \lambda_2 \circ \pi_2 - \lambda_1 \circ \pi_1$.
The kernel of $F$ is
\[
    X = \{(u_1,u_2) \in U_1 \times U_2 \,|\, \lambda_1(u_1) = \lambda_2(u_2)\},
\]
which is also known as the \emph{pullback} of $\lambda_1$ and $\lambda_2$,
namely, 
the subspace of $U_1 \times U_2$ where the images of both components are equal
\cite{zbMATH03367095}.
When restricted to $X\subset U_1 \times U_2$, the above diagram is commutative.
Note that the image of $F$ is the vector space $\Ima(\lambda_2)+\Ima(\lambda_1)$,
and the image of $X$ under $\lambda_1 \circ \pi_1 = \lambda_2 \circ \pi_2$ is
\begin{equation}\label{eq:imagelambda1capimagelamda2}
        \lambda_1 \circ \pi_1 (X)
        = \lambda_2 \circ \pi_2(X)
        = \Ima(\lambda_1) \cap \Ima(\lambda_2).
\end{equation}

Recall that $\nullity(\lambda)$ is the dimension of the kernel of a linear transformation
$\lambda:U\to V$ and $\rank(\lambda)$ is the dimension of its image.
They satisfy the equation $\nullity(\lambda)+\rank(\lambda)=\dim(U)$.

\begin{lemma}\label{lem:linear-algebra-nullity-observation-abstract}
Let $\lambda_1:U_1\to V$ and $\lambda_2:U_2\to V$
be two linear transformations where
$U_1$, $U_2$ and $V$ are vector spaces.
Assume that $\lambda_1$ is injective.
    Let $d\geq0$ be an integer.
The following conditions are equivalent:
\begin{enumerate}[(i)]
    \item $\nullity(F) - \nullity(\lambda_2) \geq d$,
    \item $\rank(F) - \rank(\lambda_2) \leq \dim(U_1)-d$,
    \item $\dim\Ima(\lambda_1) \cap \Ima(\lambda_2)\geq d$,
    \item there exists a $d$-dimensional vector space $X'\subseteq X$ such
        that $\dim(\pi_1(X'))=d$.
\end{enumerate}
\end{lemma}

\begin{proof}
    (i) $\implies$ (ii)
    We have
    \begin{align*}
           \rank(F) - \rank(\lambda_2)
           &=\dim(U_1) + \dim(U_2) -\nullity(F)  - (\dim(U_2) -\nullity(\lambda_2))\\
           &=\dim(U_1) -\nullity(F)  +\nullity(\lambda_2)\\
           &\leq \dim(U_1)-d.
    \end{align*}

    (ii) $\implies$ (iii) 
    Since $\lambda_1$ is injective, we have
    \begin{align*}
         \dim\Ima(\lambda_1) \cap \Ima(\lambda_2)
        &= \dim\Ima(\lambda_1) +\dim\Ima(\lambda_2) -\dim(\Ima(\lambda_2)+\Ima(\lambda_1))\\
        &= \dim\Ima(\lambda_1) +\dim\Ima(\lambda_2) -\dim\Ima(F)\\
        &= \rank(\lambda_1) +\rank(\lambda_2) -\rank(F)\\
        &\geq \rank(\lambda_1) + d - \dim(U_1) = d.
    \end{align*}

    (iii) $\implies$ (iv)
    Since $\lambda_1$ is injective and using Equation~\eqref{eq:imagelambda1capimagelamda2},
    we have
    \[
        \dim \pi_1(X) 
        = \dim \lambda_1 \circ \pi_1(X) 
        = \dim \Ima(\lambda_1) \cap \Ima(\lambda_2) \geq d.
    \]
    Thus, there exists a $d$-dimensional vector space $X'\subseteq X$ such
    that $\dim(\pi_1(X'))=d$.

    (iv) $\implies$ (i)
    We have $X'\subset X= \Ker(F)$.
    Since $\lambda_1$ is injective, we have
    \[
        d
        =\dim(\pi_1(X'))
        =\dim(\lambda_1\circ\pi_1(X'))
        =\dim(\lambda_2\circ\pi_2(X')).
    \]
    Thus $\dim(\pi_2(X'))= d$,
    so that $\pi_2(X')\cap \Ker(\lambda_2)=\{\mathbf{0}\}$.
    We have $\mathbf{0}\times\Ker(\lambda_2)\subseteq\Ker(F)$,
    so we conclude that
    \begin{align*}
       \nullity(F) 
        &= \dim\Ker(F) 
        \geq \dim(\mathbf{0}\times\Ker(\lambda_2) + X')\\
        &\geq \dim\pi_2(\mathbf{0}\times\Ker(\lambda_2) + X')\\
        &=    \dim(\Ker(\lambda_2) + \pi_2(X'))\\
        &=    \dim(\Ker(\lambda_2)) + \dim(\pi_2(X'))
             -\dim(\Ker(\lambda_2) \cap \pi_2(X'))\\
        &=    \dim(\Ker(\lambda_2)) + \dim(\pi_2(X')) - 0
         = \nullity(\lambda_2) + d.
    \end{align*}
\end{proof}

\subsection{Sets of Wang tiles encoded as linear transformations}

In what follows, we encode a set of Wang tiles
as a pair of linear transformations.
More precisely, to every set 
$\Tcal\subset I\times J\times I\times J$
of Wang tiles with horizontal and vertical stripes
with partitions $I=I_0\cup I_1$ and $J=J_0\cup J_1$,
we construct two linear transformations.

One linear transformation is $\sigma:\R^\Scal \to \R^{\Tcal}$,
encoding the stripe type 
(a value in the set $\Scal = \{\square, \boxbar, \boxminus,\boxplus\}$)
of every tile is defined as follows.
For a stripe type $s\in \Scal$ and tile $\tau\in\Tcal$, define $\sigma(\delta_s)(\tau) =
\delta_s(\text{stripe type}(\tau))$. 
Note that if the partition 
$ \Tcal = \Tcal_\square \cup \Tcal_\boxbar \cup \Tcal_\boxminus \cup \Tcal_\boxplus$
is such that each of the four subsets is non-empty (which we assume from here on),
then the map $\sigma$ is injective.

Another linear transformation is $\mu : \R^{I \sqcup J} \to \R^\Tcal$ is defined as follows. 
For an edge label $e \in I\sqcup J$, $\mu$ maps $\delta_e$ to the function sending a tile $\tau = (a,b,c,d)$ to $\delta_e(a) + \delta_e(b) - \delta_e(c) - \delta_e(d)$. 
Note that $\mu$ is not a bijective encoding of the set $\Tcal$ as it
loses information when the left and right color of a tile is the same or when
the bottom and top color of a tile is the same, in which case the corresponding
image is zero.

As in Section~\ref{sec:pullback-of-linear-transformations},
let $X\subset\R^\Scal\times\R^{I \sqcup J}$ be the 
kernel of $\mu \circ \pi_2 - \sigma \circ \pi_1$.
We have the following commutative diagram:
\[
\begin{tikzcd}
\R^{I \sqcup J} \arrow[r, "\mu"]  & \R^{\Tcal} \\
X \arrow[u, "\pi_2"]\arrow[r, "\pi_1"] & \R^\Scal \arrow[u, "\sigma"]
\end{tikzcd}
\]

We denote the canonical basis of
$\R^\Scal$,
$\R^\Tcal$ and
$\R^{I\sqcup J}$ 
as 
$\{\delta_s\}_{s\in\Scal}$,
$\{\delta_\tau\}_{\tau\in\Tcal}$ and
$\{\delta_e\}_{e\in I\sqcup J}$, respectively.
Using these canonical basis,
we can represent
the linear transformations $\mu$ and $\sigma$ 
as matrices:
$[\mu]=\Mcal_\Tcal$ and $[\sigma]=S_\Tcal$.
With this notation, the linear transformation 
$F=\mu \circ \pi_2 - \sigma \circ \pi_1$
is represented as the augmented matrix
\[
    [F] = (-S_\Tcal|\Mcal_\Tcal).
\]

The matrix $\Mcal_\Tcal$ has $\#\Tcal$ rows and $\#I + \#J$ columns
with entries in the set $\{-1,0,1\}$
while the matrix $S_\Tcal$ is made of $\#\Tcal$ rows and $\#\Scal=4$ columns
with entries in the set $\{0,1\}$.
Note that the matrix $\Mcal_\Tcal$ is not a bijective encoding of the set $\Tcal$ as it
loses information when the left and right color of a tile is the same or when
the bottom and top color of a tile is the same, in which case the corresponding
entry is zero in the matrix.

Using Lemma~\ref{lem:linear-algebra-nullity-observation-abstract}, we may deduce
the existence of a quadrilateral determined by a set of Wang tiles.

\begin{lemma}\label{lem:the-kernel-method}
Let $\Tcal\subset I\times J\times I\times J$ be a set of Wang tiles.
Suppose that $\Tcal$ has horizontal and vertical stripes with partition
$ \Tcal = \Tcal_\square \cup \Tcal_\boxbar \cup \Tcal_\boxminus \cup \Tcal_\boxplus$
into four non-empty subsets.
Let
$\mu : \R^{I \sqcup J} \to \R^\Tcal$ and $\sigma: \R^\Scal \to \R^{\Tcal}$
defined as above from the set $\Tcal$
and $\Mcal_\Tcal=[\mu]$ and $S_\Tcal=[\sigma]$
be their matrix representations.
The following conditions are equivalent:
\begin{enumerate}[(i)]
    \item $\nullity(-S_\Tcal|\Mcal_\Tcal) - \nullity(\Mcal_\Tcal) \geq 2$,
    \item $\rank(-S_\Tcal|\Mcal_\Tcal) - \rank(\Mcal_\Tcal) \leq 2$,
    \item $\dim\colspan(\Mcal_\Tcal) \cap \colspan(S_\Tcal)\geq 2$,
    \item there exist matrices $V\in\Z^{2\times\#I}$, $H\in\Z^{2\times\#J}$
        and a rank 2 matrix
        $\ABCD{A_1}{A_2}{B_1}{B_2}{C_1}{C_2}{D_1}{D_2}\in\Z^{2\times4}$,
        such that
            \begin{equation}\label{eq:VHABCD-is-in-the-kernel}
                (-S_\Tcal|\Mcal_\Tcal)
                    \cdot \left(\ABCD{A_1}{A_2}{B_1}{B_2}{C_1}{C_2}{D_1}{D_2}|V|H\right)^T
                =\mathbf{0}^{\#\Tcal\times2},
            \end{equation}
    \item there exist two maps $\vv:I\to\Z^2$ and $\hh:J\to\Z^2$ and there exist
        four points $A, B, C, D\in\Z^2$ such that 
        $(0,0), A, B, C, D$ are non-collinear
        and $\Tcal$ determines the quadrilateral $ABDC$.
\end{enumerate}
\end{lemma}

\begin{proof}
    Using $U_1=\R^\Scal$,
    $U_2=\R^{I \sqcup J}$,
    $V=\R^\Tcal$,
    $\lambda_1=\sigma$ and
    $\lambda_2=\mu$,
    it follows from Lemma~\ref{lem:linear-algebra-nullity-observation-abstract}
    that (i), (ii), (iii) and (iv) are equivalent,
    since $\dim(U_1)=\dim(\R^\Scal)=4$.

    Note that a base of the 2-dimensional vector space $X'$ can be described
    by the two lines of a $2\times (4+\#I+\#J)$ matrix
                    $\left(\ABCD{A_1}{A_2}{B_1}{B_2}{C_1}{C_2}{D_1}{D_2}|V|H\right)$.
    The matrix
    $\ABCD{A_1}{A_2}{B_1}{B_2}{C_1}{C_2}{D_1}{D_2}$
    has rank 2 if and only if $\dim(\pi_1(X'))=2$.
    Equation~\eqref{eq:VHABCD-is-in-the-kernel} holds if and only if $X'\subseteq \Ker F=X$.

    (iv) $\iff$ (v)
    The two maps $\vv:I\to\Z^2$ and $\hh:J\to\Z^2$ and the
    four points $A, B, C, D\in\Z^2$ 
    can be encoded into matrices
    $V\in\Z^{2\times\#I}$, $H\in\Z^{2\times\#J}$
    and $\ABCD{A_1}{A_2}{B_1}{B_2}{C_1}{C_2}{D_1}{D_2}\in\Z^{2\times4}$, and conversely, such matrices determine the maps and points.
    Equation~\eqref{eq:VHABCD-is-in-the-kernel} is satisfied if and only if the set $\Tcal$ determines the quadrilateral $ABDC$.
    Also, the matrix
    $\ABCD{A_1}{A_2}{B_1}{B_2}{C_1}{C_2}{D_1}{D_2}$
    having rank 2 is equivalent to $(0,0), A, B, C, D$ being non-collinear.
\end{proof}

In what follows, we suppose that the matrix $\Mcal_\Tcal$ is constructed such that 
the first $\#I_0$ columns from left to right are associated with the vertical
labels in the set $I_0$,
the next $\#I_1$ columns are associated with the vertical labels in the set $I_1$,
the next $\#J_0$ columns are associated with the horizontal labels in the set $J_0$
and
the last $\#J_1$ columns are associated with the horizontal labels in the set $J_1$.
Also, we assume that the four columns of the matrix $S_\Tcal$
are associated with the stripe types in the set
$\{\square, \boxminus, \boxbar, \boxplus\}$ in that order.

A nice consequence of Lemma~\ref{lem:the-kernel-method} is that it provides
an algorithm based on linear algebra 
to find a quadrilateral in order to apply 
Lemma~\ref{lem:doubly-periodic-densities-satisfy-equation}
and Theorem~\ref{thm:not-periodic}.
Indeed, let $\eta: \Ker\mu\to X: w \mapsto (0,w)$ so that the following is an exact sequence:
\[
    \begin{tikzcd}
        0           \arrow[r]
        & \Ker\mu   \arrow[r, "\eta"] 
        & X         \arrow[r, "\pi_1"] 
        & \pi_1(X)  \arrow[r] %
        & 0.
    \end{tikzcd}
\]
The sequence of linear maps being exact, we have the isomorphism
$\pi_1(X) \simeq X/\eta(\Ker\mu)$.
Thus, in practice, to compute a $d$-dimensional subspace 
$X'\subseteq X$ such that $\dim(\pi_1(X'))=d$
as in (iv) of Lemma
\ref{lem:linear-algebra-nullity-observation-abstract}, it is convenient to find
the orthogonal complement of $\eta(\Ker\mu) =\{\mathbf{0}\} \times \Ker(\mu)$ 
in $X$ and then pick a $d$-dimensional
subspace. Translated into the setting of Lemma \ref{lem:the-kernel-method} this
amounts to finding a matrix $K$ whose rows are a basis for the nullspace of
$M_\Tcal$, and then computing the nullspace of 
$\left(\begin{smallmatrix}
    -S_\Tcal & M_\Tcal\\ 0 & K
 \end{smallmatrix}\right)$, 
since this consists of elements in $X=\Ker(F) = \nullspace(-S_\Tcal | M_\Tcal)$
which are orthogonal to $(0 | K)$.

The matrix $K$ contains at least four rows which can be defined easily as follows.
Indeed, the following four vectors
\[
    k_{I_0} = \sum_{i\in I_0} e_{i},\qquad
    k_{I_1} = \sum_{i\in I_1} e_{i},\qquad
    k_{J_0} = \sum_{i\in J_0} e_{i},\qquad
    k_{J_1} = \sum_{i\in J_1} e_{i}
\]
are in the right kernel of the matrix $\Mcal_\Tcal$.
In other words, we have
$\Mcal_\Tcal k_{I_0} =
 \Mcal_\Tcal k_{I_1} =
 \Mcal_\Tcal k_{J_0} =
 \Mcal_\Tcal k_{J_1} =0$,
since the left and right labels of a tile belongs to the same subset $I_a$ for some $a\in\{0,1\}$
and the top and bottom labels of a tile belongs to the same subset $J_a$ for some $a\in\{0,1\}$.

Let $K_1=\left(
\begin{array}{cccc}
k_{I_0} &
k_{I_1} &
k_{J_0} &
k_{J_1}
\end{array}\right)^T$ be a  $4\times(\#I + \#J)$ matrix. 
Since the rows of $K_1$ are in the right kernel of $\Mcal_\Tcal$, we have that
$\nullity(\Mcal_\Tcal) \geq 4$.
If $\nullity(\Mcal_\Tcal) = 4$, then we can choose $K=K_1$.
If $\nullity(\Mcal_\Tcal) > 4$, then
we can choose $K$ so that the first four rows are $K_1$, and write $K = \left(\begin{smallmatrix}
K_1\\ K_2
 \end{smallmatrix}\right)$
 where $K_2$ is a basis of the nullspace of $\left(\begin{smallmatrix}
\Mcal\\ K_1
 \end{smallmatrix}\right)$.
 This is needed in Example~\ref{example:kernelmethod-Penrose} below when
 dealing with the encoding of Penrose tilings into 24 Wang tiles.

If the columns of $S_\Tcal$ are linearly independent and also linearly
independent with the columns of $\Mcal_\Tcal$, then
$\nullity(-S_\Tcal|\Mcal_\Tcal) = \nullity(\Mcal_\Tcal)$.
But, if for some reason we have $\nullity(-S_\Tcal|\Mcal_\Tcal) - \nullity(\Mcal_\Tcal) \geq 2$,
then Lemma~\ref{lem:the-kernel-method} applies and this hypothesis is verified
in the examples that follow later in this section.

As explained in the following remark, we can use the matrix $K_1$ to make the entries
of $V$ and $H$ become nonnegative.

\begin{remark}
Using the rows of the matrix $K_1$,
it is possible to find a matrix $P\in\Z^{2\times 4}$
such that
\[
\left(\ABCD{A_1}{A_2}{B_1}{B_2}{C_1}{C_2}{D_1}{D_2}|V'|H'\right)
=
    \left(\ABCD{A_1}{A_2}{B_1}{B_2}{C_1}{C_2}{D_1}{D_2}|V|H\right) 
    + \left(\ABCD{0}{0}{0}{0}{0}{0}{0}{0}|PK_1\right) 
\]
where both submatrices $V'\in\N^{2\times\#I}$ and 
$H'\in\N^{2\times\#J}$ have nonnegative entries.
Also, we can subtract the vectors in the set
$\{k_{I_0},k_{I_1},k_{J_0},k_{J_1}\}$ to get the entries minimal
with at least one entry equal to zero in each subset of coordinates
$\{I_0,I_1,J_0,J_1\}$.
Since the rows of $K_1$ are in the right kernel of $\Mcal_\Tcal$,
    this matrix is also in the right kernel of $(-S_\Tcal|\Mcal_\Tcal)$:
\begin{equation}\label{eq:V'H'ABCD-is-in-the-kernel}
(-S_\Tcal|\Mcal_\Tcal)\cdot 
\left(\ABCD{A_1}{A_2}{B_1}{B_2}{C_1}{C_2}{D_1}{D_2}|V'|H'\right)
=\mathbf{0}^{\#\Tcal\times2}.
\end{equation}
The columns of $V'$ define some map $\vv:I\to\N^2$ and
the columns of $H'$ define some map $\hh:J\to\N^2$
whose image are vectors with nonnegative entries.
\end{remark}

\begin{remark}
    If $\nullity(\Mcal_\Tcal)>4$, we claim that any row of $K_2$ can be used to identify at least one pair of edge colours of the same stripe type. Let $\ell$ be a row in $K_2$. Since it is not in the span of $ k_{I_0}, k_{I_1},k_{J_0}, k_{J_1}$, we can find distinct edges $e_1,e_2$ of the same class with $\ell(e_1) \neq \ell(e_2)$. We are free to modify the matrix $(V|H)$ by adding multiples of $\ell$ to any row, so we can use this to remove differences in the columns corresponding to $e_1$ and $e_2$.
\end{remark}

In the remainder of this section,
we apply this method to many previously known examples.
In all examples except one,
the nullity of the matrix $\Mcal_\Tcal$ is 4
and that of the matrix $(-S_\Tcal|\Mcal_\Tcal)$ is 6,
see Table~\ref{table:kernel-method-statistics}.
Only in the Penrose example do we have
$\nullity(\Mcal_\Tcal)=6$
and $\nullity(-S_\Tcal|\Mcal_\Tcal)=8$.

\begin{table}
\[
\begin{array}{l|cccccc}
    \text{Example} & \#\Tcal           & \#I & \#J 
    & \nullity(\Mcal_\Tcal) 
    & \nullity(-S_\Tcal|\Mcal_\Tcal)\\
    \hline
\text{Penrose 24 Wang tiles}   & 24 & 10 & 10  & 6 & 8\\
\text{Ammann 16 Wang tiles}    & 16 & 6  & 6   & 4 & 6\\
\text{19 self-similar example} & 19 & 10 & 6   & 4 & 6\\
\text{16 self-similar example} & 16 & 8  & 6   & 4 & 6\\
\text{$T_7$ Jeandel-Rao}       & 20 & 13 & 4   & 4 & 6 \\
\text{Metallic mean Wang tiles}& (n+3)^2 & 3(n+1) & 3(n+1) & 4 & 6\\
\end{array}
\]
\caption{Size and nullity of matrices encoding the sets of Wang tiles.}
\label{table:kernel-method-statistics}
\end{table}

\subsection{The strategy applied to Ammann and Penrose sets of Wang tiles}

In the next two examples,
we present how Lemma~\ref{lem:the-kernel-method}
can be applied to Ammann and Penrose Wang tiles
considered in the introduction and in 
Section~\ref{section:24 Penrose}.

\begin{example}
For the Ammann 16 tiles ordered as in 
    Corollary~\ref{cor:Ammann-non-periodic},
    we have the following matrices\\[5mm]
    \[
(-S|\Mcal)=
\begin{pNiceArray}{rrrr|rr|rrrr|rr|rrrr}[small,first-row,first-col]
& \square & \boxminus & \boxbar & \boxplus & 1 & 2 & 3 & 4 & 5 & 6 & 1 & 2 & 3 & 4 & 5 & 6 \\
\tau_{0} & -1 & 0 & 0 & 0 & 1 & -1 & 0 & 0 & 0 & 0 & 1 & -1 & 0 & 0 & 0 & 0 \\
\tau_{1} & 0 & 0 & 0 & -1 & 0 & 0 & 1 & -1 & 0 & 0 & 0 & 0 & 1 & -1 & 0 & 0 \\
\tau_{2} & 0 & 0 & 0 & -1 & 0 & 0 & 0 & 1 & -1 & 0 & 0 & 0 & 0 & 1 & -1 & 0 \\
\tau_{3} & 0 & 0 & 0 & -1 & 0 & 0 & -1 & 0 & 0 & 1 & 0 & 0 & -1 & 0 & 0 & 1 \\
\tau_{4} & 0 & 0 & 0 & -1 & 0 & 0 & 0 & 0 & 0 & 0 & 0 & 0 & 1 & 0 & -1 & 0 \\
\tau_{5} & 0 & 0 & 0 & -1 & 0 & 0 & 0 & -1 & 0 & 1 & 0 & 0 & 0 & 0 & 0 & 0 \\
\tau_{6} & 0 & 0 & 0 & -1 & 0 & 0 & 1 & 0 & -1 & 0 & 0 & 0 & 0 & 0 & 0 & 0 \\
\tau_{7} & 0 & 0 & 0 & -1 & 0 & 0 & 0 & 0 & 0 & 0 & 0 & 0 & 0 & -1 & 0 & 1 \\
\tau_{8} & 0 & -1 & 0 & 0 & 0 & 0 & -1 & 0 & 1 & 0 & -1 & 1 & 0 & 0 & 0 & 0 \\
\tau_{9} & 0 & -1 & 0 & 0 & 0 & 0 & 0 & 1 & 0 & -1 & -1 & 1 & 0 & 0 & 0 & 0 \\
\tau_{10} & 0 & -1 & 0 & 0 & 0 & 0 & 0 & -1 & 1 & 0 & 0 & 0 & 0 & 0 & 0 & 0 \\
\tau_{11} & 0 & -1 & 0 & 0 & 0 & 0 & 1 & 0 & 0 & -1 & 0 & 0 & 0 & 0 & 0 & 0 \\
\tau_{12} & 0 & 0 & -1 & 0 & -1 & 1 & 0 & 0 & 0 & 0 & 0 & 0 & 0 & 1 & 0 & -1 \\
\tau_{13} & 0 & 0 & -1 & 0 & -1 & 1 & 0 & 0 & 0 & 0 & 0 & 0 & -1 & 0 & 1 & 0 \\
\tau_{14} & 0 & 0 & -1 & 0 & 0 & 0 & 0 & 0 & 0 & 0 & 0 & 0 & 1 & 0 & 0 & -1 \\
\tau_{15} & 0 & 0 & -1 & 0 & 0 & 0 & 0 & 0 & 0 & 0 & 0 & 0 & 0 & -1 & 1 & 0 \\
\CodeAfter
\OverBrace[shorten,yshift=3pt]{0-5}{0-6}{I_0}
\OverBrace[shorten,yshift=3pt]{0-7}{0-10}{I_1}
\OverBrace[shorten,yshift=3pt]{0-11}{0-12}{J_0}
\OverBrace[shorten,yshift=3pt]{0-13}{0-16}{J_1}
\end{pNiceArray}
\]
satisfying
$\rank(-S|\Mcal)=10$,
$\rank(\Mcal)=8$,
$\nullity(-S|\Mcal)=6$ and
$\nullity(\Mcal)=4$.
The right kernel of $\Mcal$ is generated by the rows of the matrix $K$ where
\[
(0|K) =
\footnotesize
\def\arraycolsep{1mm}
\left(\begin{array}{rrrr|rr|rrrr|rr|rrrr}
0 & 0 & 0 & 0 & 1 & 1 & 0 & 0 & 0 & 0 & 0 & 0 & 0 & 0 & 0 & 0 \\
0 & 0 & 0 & 0 & 0 & 0 & 1 & 1 & 1 & 1 & 0 & 0 & 0 & 0 & 0 & 0 \\
0 & 0 & 0 & 0 & 0 & 0 & 0 & 0 & 0 & 0 & 1 & 1 & 0 & 0 & 0 & 0 \\
0 & 0 & 0 & 0 & 0 & 0 & 0 & 0 & 0 & 0 & 0 & 0 & 1 & 1 & 1 & 1
\end{array}\right)
\]
The right kernel of 
$\left(\begin{smallmatrix}
-S & \Mcal\\ 0 & K
 \end{smallmatrix}\right)$
has dimension 2
and is generated by the rows of 
\[
\left(\ABCD{A_1}{A_2}{B_1}{B_2}{C_1}{C_2}{D_1}{D_2}|V|H\right)
=
\footnotesize
\def\arraycolsep{1mm}
\left(\begin{array}{rrrr|rr|rrrr|rr|rrrr}
2 & 0 & 2 & -2 & 0 & 0 & -1 & 1 & 1 & -1 & 1 & -1 & 0 & 0 & 2 & -2 \\
0 & 2 & -2 & 0 & 1 & -1 & 1 & -1 & 1 & -1 & -1 & 1 & -1 & 1 & -1 & 1
\end{array}\right)
\]
from which we recover the quadrilateral $ABDC$ where
$\ABCD{A_1}{A_2}{B_1}{B_2}{C_1}{C_2}{D_1}{D_2}
=
\left(\begin{smallmatrix}
2 & 0 & 2 & -2 \\
0 & 2 & -2 & 0
\end{smallmatrix}\right)
$ and the maps $\vv$ and $\hh$.
The codomain of the maps $\vv$ and $\hh$ can be made nonnegative
by adding multiples of the first four rows of $(0|K)$ to them.
After dividing by the gcd of the entries of each row, we obtain: 
\[
\footnotesize
\def\arraycolsep{1mm}
\left(\ABCD{A_1}{A_2}{B_1}{B_2}{C_1}{C_2}{D_1}{D_2}|V'|H'\right)
=
\left(\begin{array}{rrrr|rr|rrrr|rr|rrrr}
1 & 0 & 1 & -1 & 0 & 0 & 0 & 1 & 1 & 0 & 1 & 0 & 1 & 1 & 2 & 0 \\
0 & 1 & -1 & 0 & 1 & 0 & 1 & 0 & 1 & 0 & 0 & 1 & 0 & 1 & 0 & 1
\end{array}\right)
.\]

\end{example}

\begin{remark}
If one adds the trivial Wang tile
    $(1,1,1,1)$ as a 17-th tile to the Ammann set,
    then we still have $\nullity(\Mcal_\Tcal)=4$
    but the nullity of $(-S_\Tcal|\Mcal_\Tcal)$ decreases by 1 
    to become $\nullity(-S_\Tcal|\Mcal_\Tcal)=5$,
    because the rank of $(-S_\Tcal|\Mcal_\Tcal)$ increases by 1.
    Thus, Lemma~\ref{lem:the-kernel-method}
    does not apply anymore
    and we can't find a rank 2 matrix
    $\ABCD{A_1}{A_2}{B_1}{B_2}{C_1}{C_2}{D_1}{D_2}\in\Z^{2\times4}$
    satisfying Equation~\eqref{eq:VHABCD-is-in-the-kernel}.
\end{remark}

\begin{example}
    \label{example:kernelmethod-Penrose}
For the Penrose 24 tiles given in Theorem~\ref{thm:Penrose-non-periodic},
    we have the following matrices\\[5mm]
    \[
(-S|\Mcal)=
\begin{pNiceArray}{rrrr|rrrr|rrrrrr|rrrr|rrrrrr}[small,first-row,first-col]
& \square & \boxminus & \boxbar & \boxplus & E & F & I & J & A & B & C & D & G & H & E & F & I & J & A & B & C & D & G & H \\
\tau_{0} & 0 & 0 & 0 & -1 & 0 & 0 & 0 & 0 & 1 & 0 & 0 & -1 & 0 & 0 & 0 & 0 & 0 & 0 & 1 & 0 & 0 & -1 & 0 & 0 \\
\tau_{1} & 0 & 0 & 0 & -1 & 0 & 0 & 0 & 0 & 0 & -1 & 1 & 0 & 0 & 0 & 0 & 0 & 0 & 0 & 0 & -1 & 1 & 0 & 0 & 0 \\
\tau_{2} & -1 & 0 & 0 & 0 & 1 & 0 & -1 & 0 & 0 & 0 & 0 & 0 & 0 & 0 & 1 & 0 & -1 & 0 & 0 & 0 & 0 & 0 & 0 & 0 \\
\tau_{3} & -1 & 0 & 0 & 0 & 0 & -1 & 0 & 1 & 0 & 0 & 0 & 0 & 0 & 0 & 0 & -1 & 0 & 1 & 0 & 0 & 0 & 0 & 0 & 0 \\
\tau_{4} & 0 & 0 & 0 & -1 & 0 & 0 & 0 & 0 & 0 & 0 & 1 & 0 & 0 & -1 & 0 & 0 & 0 & 0 & 0 & 0 & 1 & 0 & 0 & -1 \\
\tau_{5} & 0 & 0 & 0 & -1 & 0 & 0 & 0 & 0 & 0 & 0 & 0 & -1 & 1 & 0 & 0 & 0 & 0 & 0 & 0 & 0 & 0 & -1 & 1 & 0 \\
\tau_{6} & 0 & 0 & 0 & -1 & 0 & 0 & 0 & 0 & -1 & 0 & 0 & 0 & 0 & 1 & 0 & 0 & 0 & 0 & -1 & 0 & 0 & 0 & 0 & 1 \\
\tau_{7} & 0 & 0 & 0 & -1 & 0 & 0 & 0 & 0 & 0 & 1 & 0 & 0 & -1 & 0 & 0 & 0 & 0 & 0 & 0 & 1 & 0 & 0 & -1 & 0 \\
\tau_{8} & 0 & 0 & -1 & 0 & -1 & 1 & 0 & 0 & 0 & 0 & 0 & 0 & 0 & 0 & 0 & 0 & 0 & 0 & 0 & 0 & -1 & 0 & 1 & 0 \\
\tau_{9} & 0 & 0 & -1 & 0 & -1 & 1 & 0 & 0 & 0 & 0 & 0 & 0 & 0 & 0 & 0 & 0 & 0 & 0 & 0 & 0 & 0 & 1 & 0 & -1 \\
\tau_{10} & 0 & -1 & 0 & 0 & 0 & 0 & 0 & 0 & 0 & 0 & -1 & 0 & 1 & 0 & -1 & 1 & 0 & 0 & 0 & 0 & 0 & 0 & 0 & 0 \\
\tau_{11} & 0 & -1 & 0 & 0 & 0 & 0 & 0 & 0 & 0 & 0 & 0 & 1 & 0 & -1 & -1 & 1 & 0 & 0 & 0 & 0 & 0 & 0 & 0 & 0 \\
\tau_{12} & 0 & -1 & 0 & 0 & 0 & 0 & 0 & 0 & 0 & 1 & -1 & 0 & 0 & 0 & 0 & -1 & 1 & 0 & 0 & 0 & 0 & 0 & 0 & 0 \\
\tau_{13} & 0 & -1 & 0 & 0 & 0 & 0 & 0 & 0 & -1 & 0 & 0 & 1 & 0 & 0 & 1 & 0 & 0 & -1 & 0 & 0 & 0 & 0 & 0 & 0 \\
\tau_{14} & 0 & 0 & -1 & 0 & 0 & -1 & 1 & 0 & 0 & 0 & 0 & 0 & 0 & 0 & 0 & 0 & 0 & 0 & 0 & 1 & -1 & 0 & 0 & 0 \\
\tau_{15} & 0 & 0 & -1 & 0 & 1 & 0 & 0 & -1 & 0 & 0 & 0 & 0 & 0 & 0 & 0 & 0 & 0 & 0 & -1 & 0 & 0 & 1 & 0 & 0 \\
\tau_{16} & 0 & 0 & 0 & -1 & 0 & 0 & 0 & 0 & 0 & -1 & 0 & 0 & 0 & 1 & 0 & 0 & 0 & 0 & -1 & 0 & 1 & 0 & 0 & 0 \\
\tau_{17} & 0 & 0 & 0 & -1 & 0 & 0 & 0 & 0 & -1 & 0 & 1 & 0 & 0 & 0 & 0 & 0 & 0 & 0 & 0 & -1 & 0 & 0 & 0 & 1 \\
\tau_{18} & 0 & 0 & 0 & -1 & 0 & 0 & 0 & 0 & 0 & 1 & 0 & -1 & 0 & 0 & 0 & 0 & 0 & 0 & 1 & 0 & 0 & 0 & -1 & 0 \\
\tau_{19} & 0 & 0 & 0 & -1 & 0 & 0 & 0 & 0 & 1 & 0 & 0 & 0 & -1 & 0 & 0 & 0 & 0 & 0 & 0 & 1 & 0 & -1 & 0 & 0 \\
\tau_{20} & 0 & -1 & 0 & 0 & 0 & 0 & 0 & 0 & 0 & -1 & 0 & 1 & 0 & 0 & 0 & 1 & 0 & -1 & 0 & 0 & 0 & 0 & 0 & 0 \\
\tau_{21} & 0 & -1 & 0 & 0 & 0 & 0 & 0 & 0 & 1 & 0 & -1 & 0 & 0 & 0 & -1 & 0 & 1 & 0 & 0 & 0 & 0 & 0 & 0 & 0 \\
\tau_{22} & 0 & 0 & -1 & 0 & 0 & 1 & 0 & -1 & 0 & 0 & 0 & 0 & 0 & 0 & 0 & 0 & 0 & 0 & 0 & -1 & 0 & 1 & 0 & 0 \\
\tau_{23} & 0 & 0 & -1 & 0 & -1 & 0 & 1 & 0 & 0 & 0 & 0 & 0 & 0 & 0 & 0 & 0 & 0 & 0 & 1 & 0 & -1 & 0 & 0 & 0 \\
\CodeAfter
\OverBrace[shorten,yshift=3pt]{0-5}{0-8}{I_0}
\OverBrace[shorten,yshift=3pt]{0-9}{0-14}{I_1}
\OverBrace[shorten,yshift=3pt]{0-15}{0-18}{J_0}
\OverBrace[shorten,yshift=3pt]{0-19}{0-24}{J_1}
\end{pNiceArray}
\]
satisfying
$\rank(-S|\Mcal)=16$,
$\rank(\Mcal)=14$,
$\nullity(-S|\Mcal)=8$ and
$\nullity(\Mcal)=6$.
The right kernel of $\Mcal$ is generated by the rows of the matrix $K$ where
\[
(0|K) =
\footnotesize
\def\arraycolsep{1mm}
\left(\begin{array}{rrrr|rrrr|rrrrrr|rrrr|rrrrrr}
0 & 0 & 0 & 0 & 1 & 1 & 1 & 1 & 0 & 0 & 0 & 0 & 0 & 0 & 0 & 0 & 0 & 0 & 0 & 0 & 0 & 0 & 0 & 0 \\
0 & 0 & 0 & 0 & 0 & 0 & 0 & 0 & 1 & 1 & 1 & 1 & 1 & 1 & 0 & 0 & 0 & 0 & 0 & 0 & 0 & 0 & 0 & 0 \\
0 & 0 & 0 & 0 & 0 & 0 & 0 & 0 & 0 & 0 & 0 & 0 & 0 & 0 & 1 & 1 & 1 & 1 & 0 & 0 & 0 & 0 & 0 & 0 \\
0 & 0 & 0 & 0 & 0 & 0 & 0 & 0 & 0 & 0 & 0 & 0 & 0 & 0 & 0 & 0 & 0 & 0 & 1 & 1 & 1 & 1 & 1 & 1 \\
0 & 0 & 0 & 0 & -1 & 1 & 0 & 0 & 1 & -1 & 0 & 0 & 2 & -2 & 1 & -1 & 0 & 0 & -1 & 1 & 0 & 0 & -2 & 2 \\
0 & 0 & 0 & 0 & 3 & 3 & -3 & -3 & -4 & -4 & 2 & 2 & 2 & 2 & -3 & -3 & 3 & 3 & 4 & 4 & -2 & -2 & -2 & -2
\end{array}\right)
\]
The right kernel of 
$\left(\begin{smallmatrix}
-S & \Mcal\\ 0 & K
 \end{smallmatrix}\right)$
has dimension 2
and is generated by the rows of 
\[
\left(\ABCD{A_1}{A_2}{B_1}{B_2}{C_1}{C_2}{D_1}{D_2}|V|H\right)
=
\footnotesize
\def\arraycolsep{1mm}
\left(\begin{array}{rrrr|rrrr|rrrrrr|rrrr|rrrrrr}
2 & 4 & -2 & -2 & 1 & -1 & -2 & 2 & 0 & 0 & -3 & 3 & 1 & -1 & 0 & 0 & 1 & -1 & 1 & -1 & 0 & 0 & 0 & 0 \\
0 & 6 & -6 & 0 & 1 & -1 & -3 & 3 & -1 & 1 & -3 & 3 & 1 & -1 & -1 & 1 & 3 & -3 & 1 & -1 & 3 & -3 & -1 & 1
\end{array}\right)
\]
from which we recover the quadrilateral $ABDC$ where
$\ABCD{A_1}{A_2}{B_1}{B_2}{C_1}{C_2}{D_1}{D_2}
=
\left(\begin{smallmatrix}
2 & 4 & -2 & -2 \\
0 & 6 & -6 & 0
\end{smallmatrix}\right)
$ and the maps $\vv$ and $\hh$.
The codomain of the maps $\vv$ and $\hh$ can be made nonnegative
by adding multiples of the first four rows of $(0|K)$ to them.
After dividing by the gcd of the entries of each row, we obtain: 
\[
\footnotesize
\def\arraycolsep{1mm}
\left(\ABCD{A_1}{A_2}{B_1}{B_2}{C_1}{C_2}{D_1}{D_2}|V'|H'\right)
=
\left(\begin{array}{rrrr|rrrr|rrrrrr|rrrr|rrrrrr}
2 & 4 & -2 & -2 & 3 & 1 & 0 & 4 & 3 & 3 & 0 & 6 & 4 & 2 & 1 & 1 & 2 & 0 & 2 & 0 & 1 & 1 & 1 & 1 \\
0 & 3 & -3 & 0 & 2 & 1 & 0 & 3 & 1 & 2 & 0 & 3 & 2 & 1 & 1 & 2 & 3 & 0 & 2 & 1 & 3 & 0 & 1 & 2
\end{array}\right)
.\]

\end{example}

In the next sections, we use the method provided by
Lemma~\ref{lem:doubly-periodic-densities-satisfy-equation}
to conclude the non-periodicity of other known sets of Wang tiles.

\subsection{The 19 self-similar example}

In this section, we consider the set $\Ucal$ of 19 Wang tiles from \cite{MR3978536}.

\begin{theorem}
    The set $\Ucal$ 
    \begin{center}
    \includegraphics{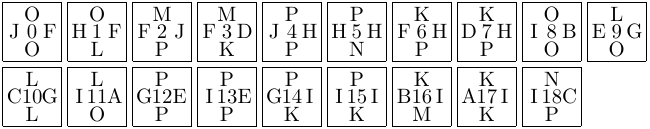}
    \end{center}
    of 19 Wang tiles admits no periodic tiling.
\end{theorem}

\begin{proof}
For the self-similar 19 tiles example, using Lemma~\ref{lem:the-kernel-method}, 
    we have the following matrices\\[5mm]
    \[
(-S|\Mcal)=
\begin{pNiceArray}{rrrr|rrrrrr|rrrr|rrrr|rr}[small,first-row,first-col]
& \square & \boxminus & \boxbar & \boxplus & A & B & C & E & G & I & D & F & H & J & K & M & N & P & L & O \\
\tau_{0} & 0 & 0 & 0 & -1 & 0 & 0 & 0 & 0 & 0 & 0 & 0 & 1 & 0 & -1 & 0 & 0 & 0 & 0 & 0 & 0 \\
\tau_{1} & 0 & 0 & 0 & -1 & 0 & 0 & 0 & 0 & 0 & 0 & 0 & 1 & -1 & 0 & 0 & 0 & 0 & 0 & -1 & 1 \\
\tau_{2} & 0 & -1 & 0 & 0 & 0 & 0 & 0 & 0 & 0 & 0 & 0 & -1 & 0 & 1 & 0 & 1 & 0 & -1 & 0 & 0 \\
\tau_{3} & 0 & -1 & 0 & 0 & 0 & 0 & 0 & 0 & 0 & 0 & 1 & -1 & 0 & 0 & -1 & 1 & 0 & 0 & 0 & 0 \\
\tau_{4} & 0 & -1 & 0 & 0 & 0 & 0 & 0 & 0 & 0 & 0 & 0 & 0 & 1 & -1 & 0 & 0 & 0 & 0 & 0 & 0 \\
\tau_{5} & 0 & -1 & 0 & 0 & 0 & 0 & 0 & 0 & 0 & 0 & 0 & 0 & 0 & 0 & 0 & 0 & -1 & 1 & 0 & 0 \\
\tau_{6} & 0 & -1 & 0 & 0 & 0 & 0 & 0 & 0 & 0 & 0 & 0 & -1 & 1 & 0 & 1 & 0 & 0 & -1 & 0 & 0 \\
\tau_{7} & 0 & -1 & 0 & 0 & 0 & 0 & 0 & 0 & 0 & 0 & -1 & 0 & 1 & 0 & 1 & 0 & 0 & -1 & 0 & 0 \\
\tau_{8} & 0 & 0 & -1 & 0 & 0 & 1 & 0 & 0 & 0 & -1 & 0 & 0 & 0 & 0 & 0 & 0 & 0 & 0 & 0 & 0 \\
\tau_{9} & 0 & 0 & -1 & 0 & 0 & 0 & 0 & -1 & 1 & 0 & 0 & 0 & 0 & 0 & 0 & 0 & 0 & 0 & 1 & -1 \\
\tau_{10} & 0 & 0 & -1 & 0 & 0 & 0 & -1 & 0 & 1 & 0 & 0 & 0 & 0 & 0 & 0 & 0 & 0 & 0 & 0 & 0 \\
\tau_{11} & 0 & 0 & -1 & 0 & 1 & 0 & 0 & 0 & 0 & -1 & 0 & 0 & 0 & 0 & 0 & 0 & 0 & 0 & 1 & -1 \\
\tau_{12} & -1 & 0 & 0 & 0 & 0 & 0 & 0 & 1 & -1 & 0 & 0 & 0 & 0 & 0 & 0 & 0 & 0 & 0 & 0 & 0 \\
\tau_{13} & -1 & 0 & 0 & 0 & 0 & 0 & 0 & 1 & 0 & -1 & 0 & 0 & 0 & 0 & 0 & 0 & 0 & 0 & 0 & 0 \\
\tau_{14} & -1 & 0 & 0 & 0 & 0 & 0 & 0 & 0 & -1 & 1 & 0 & 0 & 0 & 0 & -1 & 0 & 0 & 1 & 0 & 0 \\
\tau_{15} & -1 & 0 & 0 & 0 & 0 & 0 & 0 & 0 & 0 & 0 & 0 & 0 & 0 & 0 & -1 & 0 & 0 & 1 & 0 & 0 \\
\tau_{16} & -1 & 0 & 0 & 0 & 0 & -1 & 0 & 0 & 0 & 1 & 0 & 0 & 0 & 0 & 1 & -1 & 0 & 0 & 0 & 0 \\
\tau_{17} & -1 & 0 & 0 & 0 & -1 & 0 & 0 & 0 & 0 & 1 & 0 & 0 & 0 & 0 & 0 & 0 & 0 & 0 & 0 & 0 \\
\tau_{18} & -1 & 0 & 0 & 0 & 0 & 0 & 1 & 0 & 0 & -1 & 0 & 0 & 0 & 0 & 0 & 0 & 1 & -1 & 0 & 0 \\
\CodeAfter
\OverBrace[shorten,yshift=3pt]{0-5}{0-10}{I_0}
\OverBrace[shorten,yshift=3pt]{0-11}{0-14}{I_1}
\OverBrace[shorten,yshift=3pt]{0-15}{0-18}{J_0}
\OverBrace[shorten,yshift=3pt]{0-19}{0-20}{J_1}
\end{pNiceArray}
\]
satisfying
$\rank(-S|\Mcal)=14$,
$\rank(\Mcal)=12$,
$\nullity(-S|\Mcal)=6$ and
$\nullity(\Mcal)=4$.
The right kernel of $\Mcal$ is generated by the rows of the matrix $K$ where
\[
(0|K) =
\footnotesize
\def\arraycolsep{1mm}
\left(\begin{array}{rrrr|rrrrrr|rrrr|rrrr|rr}
0 & 0 & 0 & 0 & 1 & 1 & 1 & 1 & 1 & 1 & 0 & 0 & 0 & 0 & 0 & 0 & 0 & 0 & 0 & 0 \\
0 & 0 & 0 & 0 & 0 & 0 & 0 & 0 & 0 & 0 & 1 & 1 & 1 & 1 & 0 & 0 & 0 & 0 & 0 & 0 \\
0 & 0 & 0 & 0 & 0 & 0 & 0 & 0 & 0 & 0 & 0 & 0 & 0 & 0 & 1 & 1 & 1 & 1 & 0 & 0 \\
0 & 0 & 0 & 0 & 0 & 0 & 0 & 0 & 0 & 0 & 0 & 0 & 0 & 0 & 0 & 0 & 0 & 0 & 1 & 1
\end{array}\right)
\]
The right kernel of 
$\left(\begin{smallmatrix}
-S & \Mcal\\ 0 & K
 \end{smallmatrix}\right)$
has dimension 2
and is generated by the rows of 
\[
\left(\ABCD{A_1}{A_2}{B_1}{B_2}{C_1}{C_2}{D_1}{D_2}|V|H\right)
=
\footnotesize
\def\arraycolsep{1mm}
\left(\begin{array}{rrrr|rrrrrr|rrrr|rrrr|rr}
2 & 0 & -2 & -2 & -2 & -2 & 2 & 2 & 0 & 0 & -1 & -1 & 1 & 1 & -1 & -1 & 1 & 1 & 0 & 0 \\
0 & 4 & -4 & 0 & 0 & -4 & 4 & 0 & 0 & 0 & -1 & -1 & 3 & -1 & 0 & 4 & -4 & 0 & -2 & 2
\end{array}\right)
\]
from which we recover the quadrilateral $ABDC$ where
$\ABCD{A_1}{A_2}{B_1}{B_2}{C_1}{C_2}{D_1}{D_2}
=
\left(\begin{smallmatrix}
2 & 0 & -2 & -2 \\
0 & 4 & -4 & 0
\end{smallmatrix}\right)
$ and the maps $\vv$ and $\hh$.
The codomain of the maps $\vv$ and $\hh$ can be made nonnegative
by adding multiples of the first four rows of $(0|K)$ to them.
After dividing by the gcd of the entries of each row, we obtain: 
\[
\footnotesize
\def\arraycolsep{1mm}
\left(\ABCD{A_1}{A_2}{B_1}{B_2}{C_1}{C_2}{D_1}{D_2}|V'|H'\right)
=
\left(\begin{array}{rrrr|rrrrrr|rrrr|rrrr|rr}
1 & 0 & -1 & -1 & 0 & 0 & 2 & 2 & 1 & 1 & 0 & 0 & 1 & 1 & 0 & 0 & 1 & 1 & 0 & 0 \\
0 & 1 & -1 & 0 & 1 & 0 & 2 & 1 & 1 & 1 & 0 & 0 & 1 & 0 & 1 & 2 & 0 & 1 & 0 & 1
\end{array}\right)
.\]

Thus, the set of Wang tiles determines the quadrilateral $ABDC$ satisfying
    \begin{align*}
    \det\left(\begin{smallmatrix} A_1&D_1\\A_2&D_2 \end{smallmatrix}\right)
        &= \det\left(\begin{smallmatrix} 1&-1\\0&0 \end{smallmatrix}\right) =0,\\
    \det\left(\begin{smallmatrix} B_1&D_1\\B_2&D_2 \end{smallmatrix}\right)
        &=\det\left(\begin{smallmatrix} 0&-1\\1&0 \end{smallmatrix}\right)
        =1 \neq0,\\
    \det\left(\begin{smallmatrix} C_1&D_1\\C_2&D_2 \end{smallmatrix}\right)
       &=\det\left(\begin{smallmatrix} -1&-1\\-1&0 \end{smallmatrix}\right)
      =-1=-\det\left(\begin{smallmatrix} B_1&D_1\\B_2&D_2 \end{smallmatrix}\right).
    \end{align*}
The discriminant
$\Delta=(B_1+C_1)^2-4A_1D_1=(0-1)^2+4=5$
is positive and is not a perfect square.
Thus, we conclude from Theorem~\ref{thm:not-periodic} that
every valid configuration over these Wang tiles is non-periodic.
\end{proof}

\subsection{The 16 self-similar example}\label{section:16_self_similar}

In this section, we consider the 16 Wang tiles from \cite{labbe_three_2020}.
It was obtained as a simplification \cite{Lepsova-2024}
of the set of 19 Wang tiles from \cite{MR3978536}
after identification of labels $F$ and $D$ and labels $G$ and $I$.
This simplification can be deduced by noting that the $2 \times (4+ \#I + \#J)$
matrix 
$\left(\ABCD{A_1}{A_2}{B_1}{B_2}{C_1}{C_2}{D_1}{D_2}|V|H\right)$
in the previous example has two duplicate columns in the $I_0$ section,
as well as two in the $I_1$ section. 

\begin{theorem}
    The set $\Zcal$ 
    \begin{center}
    \includegraphics{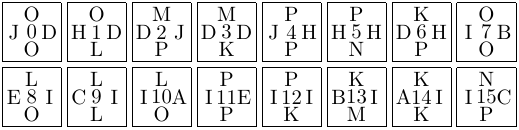}
    \end{center}
    of 16 Wang tiles admits no periodic tiling.
\end{theorem}

\begin{proof}
For the self-similar 16 tiles example, using Lemma~\ref{lem:the-kernel-method}, 
    we have the following matrices\\[5mm]
    \[
(-S|\Mcal)=
\begin{pNiceArray}{rrrr|rrrrr|rrr|rrrr|rr}[small,first-row,first-col]
& \square & \boxminus & \boxbar & \boxplus & A & B & C & E & I & D & H & J & K & M & N & P & L & O \\
\tau_{0} & 0 & 0 & 0 & -1 & 0 & 0 & 0 & 0 & 0 & 1 & 0 & -1 & 0 & 0 & 0 & 0 & 0 & 0 \\
\tau_{1} & 0 & 0 & 0 & -1 & 0 & 0 & 0 & 0 & 0 & 1 & -1 & 0 & 0 & 0 & 0 & 0 & -1 & 1 \\
\tau_{2} & 0 & -1 & 0 & 0 & 0 & 0 & 0 & 0 & 0 & -1 & 0 & 1 & 0 & 1 & 0 & -1 & 0 & 0 \\
\tau_{3} & 0 & -1 & 0 & 0 & 0 & 0 & 0 & 0 & 0 & 0 & 0 & 0 & -1 & 1 & 0 & 0 & 0 & 0 \\
\tau_{4} & 0 & -1 & 0 & 0 & 0 & 0 & 0 & 0 & 0 & 0 & 1 & -1 & 0 & 0 & 0 & 0 & 0 & 0 \\
\tau_{5} & 0 & -1 & 0 & 0 & 0 & 0 & 0 & 0 & 0 & 0 & 0 & 0 & 0 & 0 & -1 & 1 & 0 & 0 \\
\tau_{6} & 0 & -1 & 0 & 0 & 0 & 0 & 0 & 0 & 0 & -1 & 1 & 0 & 1 & 0 & 0 & -1 & 0 & 0 \\
\tau_{7} & 0 & 0 & -1 & 0 & 0 & 1 & 0 & 0 & -1 & 0 & 0 & 0 & 0 & 0 & 0 & 0 & 0 & 0 \\
\tau_{8} & 0 & 0 & -1 & 0 & 0 & 0 & 0 & -1 & 1 & 0 & 0 & 0 & 0 & 0 & 0 & 0 & 1 & -1 \\
\tau_{9} & 0 & 0 & -1 & 0 & 0 & 0 & -1 & 0 & 1 & 0 & 0 & 0 & 0 & 0 & 0 & 0 & 0 & 0 \\
\tau_{10} & 0 & 0 & -1 & 0 & 1 & 0 & 0 & 0 & -1 & 0 & 0 & 0 & 0 & 0 & 0 & 0 & 1 & -1 \\
\tau_{11} & -1 & 0 & 0 & 0 & 0 & 0 & 0 & 1 & -1 & 0 & 0 & 0 & 0 & 0 & 0 & 0 & 0 & 0 \\
\tau_{12} & -1 & 0 & 0 & 0 & 0 & 0 & 0 & 0 & 0 & 0 & 0 & 0 & -1 & 0 & 0 & 1 & 0 & 0 \\
\tau_{13} & -1 & 0 & 0 & 0 & 0 & -1 & 0 & 0 & 1 & 0 & 0 & 0 & 1 & -1 & 0 & 0 & 0 & 0 \\
\tau_{14} & -1 & 0 & 0 & 0 & -1 & 0 & 0 & 0 & 1 & 0 & 0 & 0 & 0 & 0 & 0 & 0 & 0 & 0 \\
\tau_{15} & -1 & 0 & 0 & 0 & 0 & 0 & 1 & 0 & -1 & 0 & 0 & 0 & 0 & 0 & 1 & -1 & 0 & 0 \\
\CodeAfter
\OverBrace[shorten,yshift=3pt]{0-5}{0-9}{I_0}
\OverBrace[shorten,yshift=3pt]{0-10}{0-12}{I_1}
\OverBrace[shorten,yshift=3pt]{0-13}{0-16}{J_0}
\OverBrace[shorten,yshift=3pt]{0-17}{0-18}{J_1}
\end{pNiceArray}
\]
satisfying
$\rank(-S|\Mcal)=12$,
$\rank(\Mcal)=10$,
$\nullity(-S|\Mcal)=6$ and
$\nullity(\Mcal)=4$.
The right kernel of $\Mcal$ is generated by the rows of the matrix $K$ where
\[
(0|K) =
\footnotesize
\def\arraycolsep{1mm}
\left(\begin{array}{rrrr|rrrrr|rrr|rrrr|rr}
0 & 0 & 0 & 0 & 1 & 1 & 1 & 1 & 1 & 0 & 0 & 0 & 0 & 0 & 0 & 0 & 0 & 0 \\
0 & 0 & 0 & 0 & 0 & 0 & 0 & 0 & 0 & 1 & 1 & 1 & 0 & 0 & 0 & 0 & 0 & 0 \\
0 & 0 & 0 & 0 & 0 & 0 & 0 & 0 & 0 & 0 & 0 & 0 & 1 & 1 & 1 & 1 & 0 & 0 \\
0 & 0 & 0 & 0 & 0 & 0 & 0 & 0 & 0 & 0 & 0 & 0 & 0 & 0 & 0 & 0 & 1 & 1
\end{array}\right)
\]
The right kernel of 
$\left(\begin{smallmatrix}
-S & \Mcal\\ 0 & K
 \end{smallmatrix}\right)$
has dimension 2
and is generated by the rows of 
\[
\left(\ABCD{A_1}{A_2}{B_1}{B_2}{C_1}{C_2}{D_1}{D_2}|V|H\right)
=
\footnotesize
\def\arraycolsep{1mm}
\left(\begin{array}{rrrr|rrrrr|rrr|rrrr|rr}
2 & 2 & -4 & -2 & -2 & -4 & 4 & 2 & 0 & -2 & 2 & 0 & -1 & 1 & -1 & 1 & -1 & 1 \\
0 & 6 & -6 & 0 & 0 & -6 & 6 & 0 & 0 & -2 & 4 & -2 & 0 & 6 & -6 & 0 & -3 & 3
\end{array}\right)
\]
from which we recover the quadrilateral $ABDC$ where
$\ABCD{A_1}{A_2}{B_1}{B_2}{C_1}{C_2}{D_1}{D_2}
=
\left(\begin{smallmatrix}
2 & 2 & -4 & -2 \\
0 & 6 & -6 & 0
\end{smallmatrix}\right)
$ and the maps $\vv$ and $\hh$.
The codomain of the maps $\vv$ and $\hh$ can be made nonnegative
by adding multiples of the first four rows of $(0|K)$ to them.
After dividing by the gcd of the entries of each row, we obtain: 
\[
\footnotesize
\def\arraycolsep{1mm}
\left(\ABCD{A_1}{A_2}{B_1}{B_2}{C_1}{C_2}{D_1}{D_2}|V'|H'\right)
=
\left(\begin{array}{rrrr|rrrrr|rrr|rrrr|rr}
1 & 1 & -2 & -1 & 1 & 0 & 4 & 3 & 2 & 0 & 2 & 1 & 0 & 1 & 0 & 1 & 0 & 1 \\
0 & 1 & -1 & 0 & 1 & 0 & 2 & 1 & 1 & 0 & 1 & 0 & 1 & 2 & 0 & 1 & 0 & 1
\end{array}\right)
.\]

Thus, the set of Wang tiles determines the quadrilateral $ABDC$ satisfying
    \begin{align*}
    \det\left(\begin{smallmatrix} A_1&D_1\\A_2&D_2 \end{smallmatrix}\right)
        &= \det\left(\begin{smallmatrix} 1&-1\\0&0 \end{smallmatrix}\right) =0,\\
    \det\left(\begin{smallmatrix} B_1&D_1\\B_2&D_2 \end{smallmatrix}\right)
        &=\det\left(\begin{smallmatrix} 1&-1\\1&0 \end{smallmatrix}\right)
        =1 \neq0,\\
    \det\left(\begin{smallmatrix} C_1&D_1\\C_2&D_2 \end{smallmatrix}\right)
       &=\det\left(\begin{smallmatrix} -2&-1\\-1&0 \end{smallmatrix}\right)
      =-1=-\det\left(\begin{smallmatrix} B_1&D_1\\B_2&D_2 \end{smallmatrix}\right).
    \end{align*}
The discriminant
$\Delta=(B_1+C_1)^2-4A_1D_1=(1-2)^2+4=5$
is positive and is not a perfect square.
Thus, we conclude from Theorem~\ref{thm:not-periodic} that
every valid configuration over these Wang tiles is non-periodic.
\end{proof}

\subsection{The set $T_7$ describing Jeandel-Rao Wang tiles}

In the description of the substitutive structure of the Jeandel-Rao tilings
made in \cite{MR4226493}, a set $T_7$ of 20 Wang tiles was considered.
Contrary to Jeandel-Rao tiles and intermediate sets of Wang tiles in the
description of the substitutive structure, the set $T_7$ has vertical and
horizontal stripes.
We may apply the method described here and prove that the set $T_7$ admits no
periodic tiling of the plane.

\begin{theorem}
    The set $\Tcal_7$ 
    \begin{center}
        \includegraphics{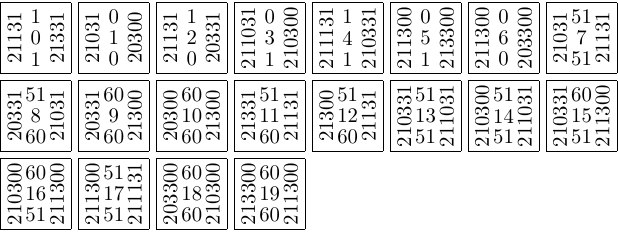}
    \end{center}
    of 20 Wang tiles admits no periodic tiling.
\end{theorem}

\begin{proof}
For the Jeandel-Rao $T_7$ example, using Lemma~\ref{lem:the-kernel-method}, 
we have the following matrices\\[5mm]
    \[
(-S|\Mcal)=
\begin{pNiceArray}{rrrr|rrrrrrr|rrrrrr|rr|rr}[small,first-row,first-col]\RowStyle{\rotate}
& \square & \boxminus & \boxbar & \boxplus & 203300 & 210300 & 210331 & 211031 & 211131 & 211300 & 213300 & 20300 & 20331 & 21031 & 21131 & 21300 & 21331 & 51 & 60 & 0 & 1 \\
\tau_{0} & 0 & 0 & 0 & -1 & 0 & 0 & 0 & 0 & 0 & 0 & 0 & 0 & 0 & 0 & -1 & 0 & 1 & 0 & 0 & 0 & 0 \\
\tau_{1} & 0 & 0 & 0 & -1 & 0 & 0 & 0 & 0 & 0 & 0 & 0 & 1 & 0 & -1 & 0 & 0 & 0 & 0 & 0 & 0 & 0 \\
\tau_{2} & 0 & 0 & 0 & -1 & 0 & 0 & 0 & 0 & 0 & 0 & 0 & 0 & 1 & 0 & -1 & 0 & 0 & 0 & 0 & -1 & 1 \\
\tau_{3} & 0 & 0 & -1 & 0 & 0 & 1 & 0 & -1 & 0 & 0 & 0 & 0 & 0 & 0 & 0 & 0 & 0 & 0 & 0 & 1 & -1 \\
\tau_{4} & 0 & 0 & -1 & 0 & 0 & 0 & 1 & 0 & -1 & 0 & 0 & 0 & 0 & 0 & 0 & 0 & 0 & 0 & 0 & 0 & 0 \\
\tau_{5} & 0 & 0 & -1 & 0 & 0 & 0 & 0 & 0 & 0 & -1 & 1 & 0 & 0 & 0 & 0 & 0 & 0 & 0 & 0 & 1 & -1 \\
\tau_{6} & 0 & 0 & -1 & 0 & 1 & 0 & 0 & 0 & 0 & -1 & 0 & 0 & 0 & 0 & 0 & 0 & 0 & 0 & 0 & 0 & 0 \\
\tau_{7} & 0 & -1 & 0 & 0 & 0 & 0 & 0 & 0 & 0 & 0 & 0 & 0 & 0 & -1 & 1 & 0 & 0 & 0 & 0 & 0 & 0 \\
\tau_{8} & 0 & -1 & 0 & 0 & 0 & 0 & 0 & 0 & 0 & 0 & 0 & 0 & -1 & 1 & 0 & 0 & 0 & 1 & -1 & 0 & 0 \\
\tau_{9} & 0 & -1 & 0 & 0 & 0 & 0 & 0 & 0 & 0 & 0 & 0 & 0 & -1 & 0 & 0 & 1 & 0 & 0 & 0 & 0 & 0 \\
\tau_{10} & 0 & -1 & 0 & 0 & 0 & 0 & 0 & 0 & 0 & 0 & 0 & -1 & 0 & 0 & 0 & 1 & 0 & 0 & 0 & 0 & 0 \\
\tau_{11} & 0 & -1 & 0 & 0 & 0 & 0 & 0 & 0 & 0 & 0 & 0 & 0 & 0 & 0 & 1 & 0 & -1 & 1 & -1 & 0 & 0 \\
\tau_{12} & 0 & -1 & 0 & 0 & 0 & 0 & 0 & 0 & 0 & 0 & 0 & 0 & 0 & 0 & 1 & -1 & 0 & 1 & -1 & 0 & 0 \\
\tau_{13} & -1 & 0 & 0 & 0 & 0 & 0 & -1 & 1 & 0 & 0 & 0 & 0 & 0 & 0 & 0 & 0 & 0 & 0 & 0 & 0 & 0 \\
\tau_{14} & -1 & 0 & 0 & 0 & 0 & -1 & 0 & 1 & 0 & 0 & 0 & 0 & 0 & 0 & 0 & 0 & 0 & 0 & 0 & 0 & 0 \\
\tau_{15} & -1 & 0 & 0 & 0 & 0 & 0 & -1 & 0 & 0 & 1 & 0 & 0 & 0 & 0 & 0 & 0 & 0 & -1 & 1 & 0 & 0 \\
\tau_{16} & -1 & 0 & 0 & 0 & 0 & -1 & 0 & 0 & 0 & 1 & 0 & 0 & 0 & 0 & 0 & 0 & 0 & -1 & 1 & 0 & 0 \\
\tau_{17} & -1 & 0 & 0 & 0 & 0 & 0 & 0 & 0 & 1 & -1 & 0 & 0 & 0 & 0 & 0 & 0 & 0 & 0 & 0 & 0 & 0 \\
\tau_{18} & -1 & 0 & 0 & 0 & -1 & 1 & 0 & 0 & 0 & 0 & 0 & 0 & 0 & 0 & 0 & 0 & 0 & 0 & 0 & 0 & 0 \\
\tau_{19} & -1 & 0 & 0 & 0 & 0 & 0 & 0 & 0 & 0 & 1 & -1 & 0 & 0 & 0 & 0 & 0 & 0 & 0 & 0 & 0 & 0 \\
\CodeAfter
\OverBrace[shorten,yshift=3pt]{0-5}{0-11}{I_0}
\OverBrace[shorten,yshift=3pt]{0-12}{0-17}{I_1}
\OverBrace[shorten,yshift=3pt]{0-18}{0-19}{J_0}
\OverBrace[shorten,yshift=3pt]{0-20}{0-21}{J_1}
\end{pNiceArray}
\]
satisfying
$\rank(-S|\Mcal)=15$,
$\rank(\Mcal)=13$,
$\nullity(-S|\Mcal)=6$ and
$\nullity(\Mcal)=4$.
The right kernel of $\Mcal$ is generated by the rows of the matrix $K$ where
\[
(0|K) =
\footnotesize
\def\arraycolsep{1mm}
\left(\begin{array}{rrrr|rrrrrrr|rrrrrr|rr|rr}
0 & 0 & 0 & 0 & 1 & 1 & 1 & 1 & 1 & 1 & 1 & 0 & 0 & 0 & 0 & 0 & 0 & 0 & 0 & 0 & 0 \\
0 & 0 & 0 & 0 & 0 & 0 & 0 & 0 & 0 & 0 & 0 & 1 & 1 & 1 & 1 & 1 & 1 & 0 & 0 & 0 & 0 \\
0 & 0 & 0 & 0 & 0 & 0 & 0 & 0 & 0 & 0 & 0 & 0 & 0 & 0 & 0 & 0 & 0 & 1 & 1 & 0 & 0 \\
0 & 0 & 0 & 0 & 0 & 0 & 0 & 0 & 0 & 0 & 0 & 0 & 0 & 0 & 0 & 0 & 0 & 0 & 0 & 1 & 1
\end{array}\right)
\]
The right kernel of 
$\left(\begin{smallmatrix}
-S & \Mcal\\ 0 & K
 \end{smallmatrix}\right)$
has dimension 2
and is generated by the rows of 
\[
\left(\ABCD{A_1}{A_2}{B_1}{B_2}{C_1}{C_2}{D_1}{D_2}|V|H\right)
=
\footnotesize
\def\arraycolsep{1mm}
\left(\begin{array}{rrrr|rrrrrrr|rrrrrr|rr|rr}
6 & 0 & -6 & -6 & -6 & 0 & 0 & 6 & 6 & 0 & -6 & -2 & -2 & 4 & 4 & -2 & -2 & -3 & 3 & 0 & 0 \\
0 & 14 & -14 & 0 & -6 & -6 & -6 & -6 & 8 & 8 & 8 & -7 & -7 & -7 & 7 & 7 & 7 & 7 & -7 & -7 & 7
\end{array}\right)
\]
from which we recover the quadrilateral $ABDC$ where
$\ABCD{A_1}{A_2}{B_1}{B_2}{C_1}{C_2}{D_1}{D_2}
=
\left(\begin{smallmatrix}
6 & 0 & -6 & -6 \\
0 & 14 & -14 & 0
\end{smallmatrix}\right)
$ and the maps $\vv$ and $\hh$.
The codomain of the maps $\vv$ and $\hh$ can be made nonnegative
by adding multiples of the first four rows of $(0|K)$ to them.
After dividing by the gcd of the entries of each row, we obtain: 
\[
\footnotesize
\def\arraycolsep{1mm}
\left(\ABCD{A_1}{A_2}{B_1}{B_2}{C_1}{C_2}{D_1}{D_2}|V'|H'\right)
=
\left(\begin{array}{rrrr|rrrrrrr|rrrrrr|rr|rr}
1 & 0 & -1 & -1 & 0 & 1 & 1 & 2 & 2 & 1 & 0 & 0 & 0 & 1 & 1 & 0 & 0 & 0 & 1 & 0 & 0 \\
0 & 1 & -1 & 0 & 0 & 0 & 0 & 0 & 1 & 1 & 1 & 0 & 0 & 0 & 1 & 1 & 1 & 1 & 0 & 0 & 1
\end{array}\right)
.\]

Thus, the set of Wang tiles determines the quadrilateral $ABDC$ satisfying
    \begin{align*}
    \det\left(\begin{smallmatrix} A_1&D_1\\A_2&D_2 \end{smallmatrix}\right)
        &= \det\left(\begin{smallmatrix} 1&-1\\0&0 \end{smallmatrix}\right) =0,\\
    \det\left(\begin{smallmatrix} B_1&D_1\\B_2&D_2 \end{smallmatrix}\right)
        &=\det\left(\begin{smallmatrix} 0&-1\\1&0 \end{smallmatrix}\right)
        =1 \neq0,\\
    \det\left(\begin{smallmatrix} C_1&D_1\\C_2&D_2 \end{smallmatrix}\right)
       &=\det\left(\begin{smallmatrix} -1&-1\\-1&0 \end{smallmatrix}\right)
       =-1
      =-\det\left(\begin{smallmatrix} B_1&D_1\\B_2&D_2 \end{smallmatrix}\right).
    \end{align*}
The discriminant
$\Delta=(B_1+C_1)^2-4A_1D_1=(0-1)^2+4=5$
is positive and is not a perfect square.
Thus, we conclude from Theorem~\ref{thm:not-periodic} that
every valid configuration over these Wang tiles is non-periodic.
\end{proof}

\section{The solutions to Equation~\eqref{alpha_beta_relations}}
\label{sec:solution-to-alpha-beta-equation}

In this section, we describe the set of solutions to
Equation~\eqref{alpha_beta_relations}. 
In the generic case, these solutions can be expressed in terms of the lines passing through the origin
that are tangent to some parabola.

\subsection{Symmetries}

Notice that the expression on the right of Equation~\eqref{alpha_beta_relations},
\begin{align*}
     (1-\alpha)\left[(1-\beta)\left(\begin{array}{c} A_1\\A_2 \end{array}\right) 
                    +\beta    \left(\begin{array}{c} B_1\\B_2 \end{array}\right) \right]
        +\alpha\left[(1-\beta)\left(\begin{array}{c} C_1\\C_2 \end{array}\right) 
                    +\beta    \left(\begin{array}{c} D_1\\D_2 \end{array}\right) \right],
\end{align*} 
can be interpreted as the double integral of a vector-valued piecewise constant function on $[0,1]^2$, with four regions separated by the lines $x=\alpha$ and $y=\beta$. 
\begin{figure}[H]
\centering
\includegraphics[scale=0.7]{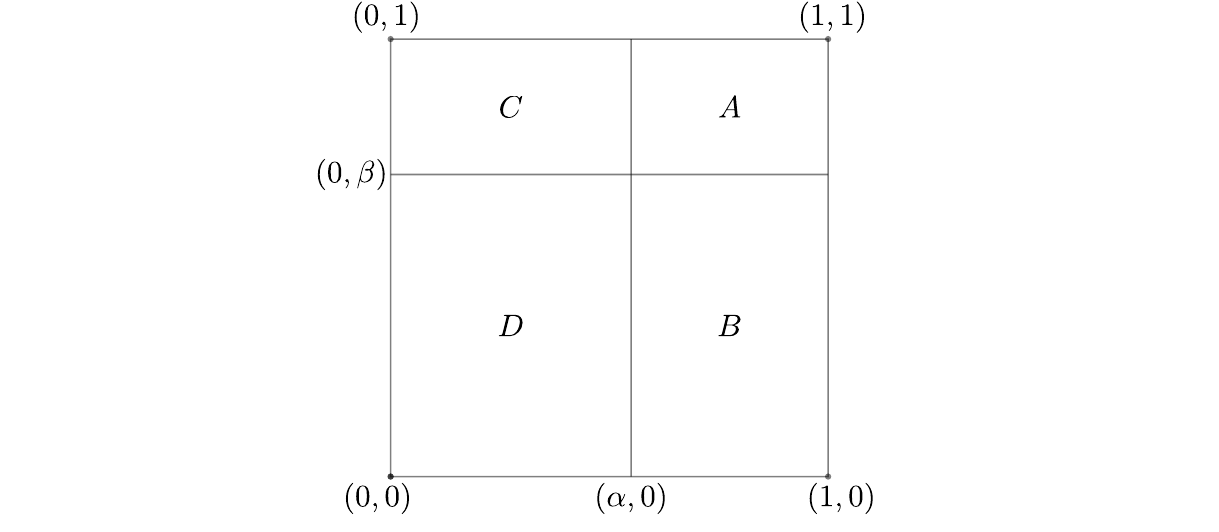}
\caption{A piecewise constant function $[0,1]^2 \to \R^2$}\label{fig:ABCD_on_the_unit_square}
\end{figure}
The group of affine transformations of $[0,1]^2$ generated by $(x,y)\mapsto (y,x)$ and $(x,y)\mapsto (1-x,y)$ is isomorphic to $D_4$, the dihedral group of order 8. Since the right action of this group on the set of functions $f:[0,1]^2 \to \R^2$ takes piecewise constant functions as above to functions of the same form, we have an induced action of $D_4$ on the set $\mathcal{Z}$ of solutions 
\[
    \left(\alpha, \beta, \ABCD{A_1}{A_2}{B_1}{B_2}{C_1}{C_2}{D_1}{D_2}\right)
    \in \C^2 \times \Z^{2\times4}
\] 
to \eqref{alpha_beta_relations}. 
Two generators for this action on $\mathcal{Z}$ are given by:
\begin{align*}
    \Theta_1:
    \left(\alpha, \beta, \ABCD{A_1}{A_2}{B_1}{B_2}{C_1}{C_2}{D_1}{D_2}\right)
    &\mapsto
    \left(\beta, \alpha, \ABCD{A_1}{A_2}{C_1}{C_2}{B_1}{B_2}{D_1}{D_2}\right),\\
    \Theta_2:
    \left(\alpha, \beta, \ABCD{A_1}{A_2}{B_1}{B_2}{C_1}{C_2}{D_1}{D_2}\right)
    &\mapsto
    \left(1-\alpha, \beta, \ABCD{C_1}{C_2}{D_1}{D_2}{A_1}{A_2}{B_1}{B_2}\right).
\end{align*}

\begin{lemma}\label{lem:Theta-action-on-solutions}
    Let
    $\left(\alpha, \beta, \ABCD{A_1}{A_2}{B_1}{B_2}{C_1}{C_2}{D_1}{D_2}\right)\in\mathcal{Z}$
    be a solution to Equation~\eqref{alpha_beta_relations}. 
    Then, for every 
    $\theta\in\langle\Theta_1,\Theta_2\rangle$,
    $\theta\left(\alpha, \beta, \ABCD{A_1}{A_2}{B_1}{B_2}{C_1}{C_2}{D_1}{D_2}\right)\in\mathcal{Z}$
    is also a solution.
\end{lemma}

\begin{proof}
    This is an easy observation.
\end{proof}

For example, in Table~\ref{table:solutions-alpha-beta-examples}, one may observe
that the solutions $(\alpha,\beta)$ 
for $\ABCD{-3}{-2}{4}{-2}{0}{3}{-1}{-2}$ are complex,
for $\ABCD{-3}{-2}{4}{-2}{-1}{-2}{0}{3}$ are rational,
while the solutions for $\ABCD{-3}{-2}{-1}{-2}{0}{3}{4}{-2}$
are real and irrational.
These three solutions are in three distinct orbits, each of size 8, under the 
action generated by $\langle\Theta_1,\Theta_2\rangle$.

\subsection{Solutions}

Write $E(\alpha,\beta)$ for the right hand side of Equation~\eqref{alpha_beta_relations}:
\begin{align}
E(\alpha, \beta) &= (1-\alpha)\left[(1-\beta)A +\beta B \right] + \alpha\left[(1-\beta)C+\beta D \right]\label{eq:alpha_line}\\
&= (1-\beta)\left[(1-\alpha)A +\alpha C \right] + \beta\left[(1-\alpha)B+\alpha D \right]\label{eq:beta_line}\\
&= A + \alpha(C-A) + \beta(B-A) + \alpha \beta(A -B - C + D). \label{eq:alpha_beta_alphabeta}
\end{align}
For fixed $\beta$, the representation \eqref{eq:alpha_line} shows that the points traced out by $E(\alpha, \beta)$ as $\alpha$ varies form a line passing through $(1-\beta)A +\beta B$ and $(1-\beta)C+\beta D$. Likewise, for fixed $\alpha$ and varying $\beta$, \eqref{eq:beta_line} shows that $E(\alpha, \beta)$ traces out a line passing through $(1-\alpha)A +\alpha C $ and $(1-\alpha)B+\alpha D$ (with the exception in both cases that if the two points coincide, $E(\alpha, \beta)$ is constant). Figure \ref{fig:lines_over_unit_intervals} illustrates this with light blue and purple lines corresponding to fixed $\alpha$ and fixed $\beta$, respectively, with the free variable restricted to the interval $[0,1]$.
\begin{figure}[h]
\centering
\includegraphics[scale=0.7]{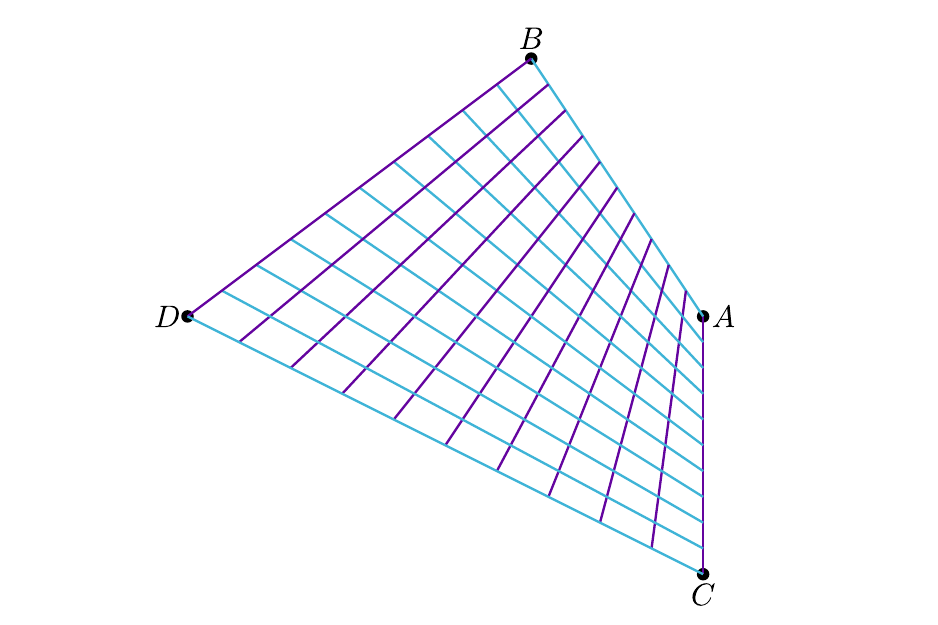}
\caption{Two representative sets of line segments traced out by $E(\alpha, \beta)$ for $\alpha$ and $\beta$ fixed separately}\label{fig:lines_over_unit_intervals}
\end{figure}
If the free variable is allowed to run through $\R$, we obtain two families of lines that sweep out the same subset of $\R^2$, so real solutions to Equation~\eqref{alpha_beta_relations} correspond to configurations where the origin lies inside the region swept out.
\begin{figure}[h]
\centering
\includegraphics[scale=0.7]{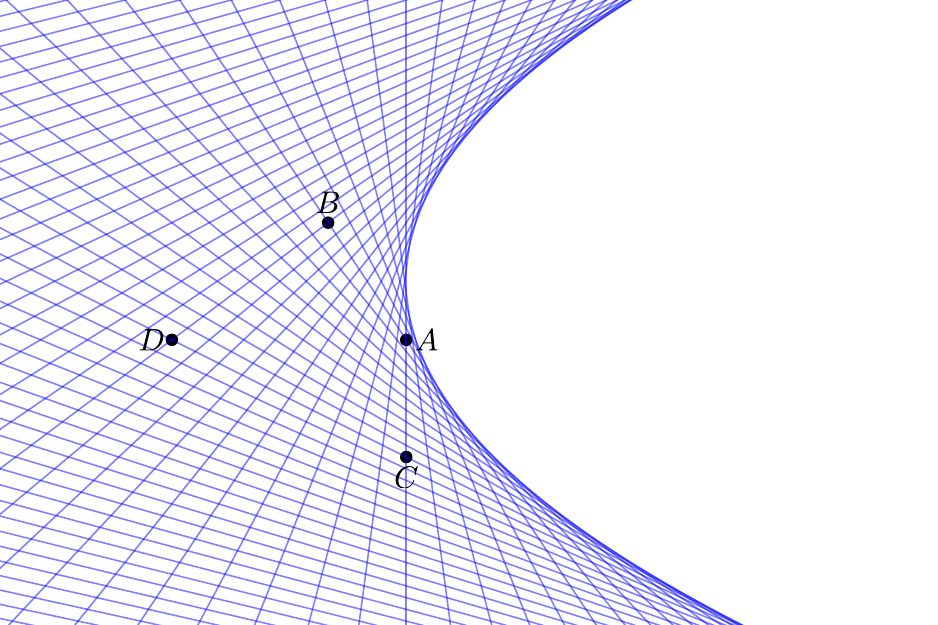}
\caption{The region swept out by lines in the generic case (vi) in Proposition \ref{prop:alpha_beta_solutions} below}\label{fig:parabolic_envelope}
\end{figure}

In the following proposition, it is convenient to express the solutions 
to \eqref{alpha_beta_relations}
in terms of the Pl\"ucker coordinates \cite{zbMATH00598493}
associated with the line in the projective 3-space $\mathbb{P}^3$
passing throught the points $(A_1:B_1:C_1:D_1)$ and $(A_2:B_2:C_2:D_2)$:
\begin{align*}
      p_{AB}  = \det\left(\begin{smallmatrix} A_1&B_1\\A_2&B_2 \end{smallmatrix}\right) &= A_1B_2 - A_2B_1, 
    & p_{BC}  = \det\left(\begin{smallmatrix} B_1&C_1\\B_2&C_2 \end{smallmatrix}\right) &= B_1C_2 - B_2C_1,\\
      p_{AC}  = \det\left(\begin{smallmatrix} A_1&C_1\\A_2&C_2 \end{smallmatrix}\right) &= A_1C_2 - A_2C_1, 
    & p_{BD} =  \det\left(\begin{smallmatrix} B_1&D_1\\B_2&D_2 \end{smallmatrix}\right) &= B_1D_2 - B_2D_1,\\
      p_{AD}  = \det\left(\begin{smallmatrix} A_1&D_1\\A_2&D_2 \end{smallmatrix}\right) &= A_1D_2 - A_2D_1, 
    & p_{CD}  = \det\left(\begin{smallmatrix} C_1&D_1\\C_2&D_2 \end{smallmatrix}\right) &= C_1D_2 - C_2D_1,
\end{align*}
satisfying $p_{ij}=-p_{ji}$ for every $i,j\in\{A,B,C,D\}$.
Note that we say that two vectors $V,W$ are \emph{parallel} if
$\text{span}(\{V,W\})$ is at most one-dimensional.  Thus the zero vector is
parallel to all vectors.

\begin{proposition}\label{prop:alpha_beta_solutions}
    Let $A,B,C,D \in \R^2$, and consider the quadrilateral $ABDC$.
    \begin{enumerate}[(i)]
    \item\label{item:points_coincide} If $A=B=C=D$, then Equation~\eqref{alpha_beta_relations} has 
            no solution unless $A=B=C=D=\mathbf{0}$,
            in which case $(\alpha, \beta)$ can be arbitrary.
    \item\label{item:points_collinear} If the points do not coincide and $A,B,C,D$ lie on a line $\ell$, then
            Equation~\eqref{alpha_beta_relations} has 
            no solution unless $\ell$
            goes through the origin, 
            in which case it has a one-dimensional solution space.
    \item\label{item:parallelogram} If the points are not collinear and $ABDC$ is a parallelogram, then
            Equation~\eqref{alpha_beta_relations} has a unique solution given by
            \begin{align}
            (\alpha, \beta) &= \frac{1}{p_{BC} + p_{CA} + p_{AB}}\left(p_{AB}, p_{CA}\right)
            \end{align}
    \item\label{item:trapezium} If $ABDC$ is a trapezium, with the edges $AC$ and $BD$ parallel, and $AB$ and $CD$ not parallel, let $P$ be the point of intersection of the lines through $AB$ and $CD$, and let $\ell$ be the line through $P$ parallel to $AC$ and $BD$. If $P=\mathbf{0}$, solutions to Equation~\eqref{alpha_beta_relations} are given by
            \begin{align}
            (\alpha, \beta) &= \left(\gamma, \frac{p_{AB}+p_{AC}+p_{DA}}{p_{AC} + p_{CB} +  p_{BD} + p_{DA}}\right),\;\;\gamma \in \R
            \end{align}
            Otherwise, if $P\neq \mathbf{0}$ and $\ell$ passes through $\mathbf{0}$, there are no solutions, while if $\ell$ does not pass through $\mathbf{0}$, there is a unique solution given by 
            \begin{align}
            (\alpha, \beta) &= \left(\frac{p_{AB}}{p_{AB} + p_{DC} }, \frac{p_{AB}+p_{AC}+p_{DA}}{p_{AC} + p_{CB} +  p_{BD} + p_{DA}}\right)
            \end{align}
    \item If the edges $AB$ and $CD$ are parallel, and $AC$ and $BD$ are not, the same conclusions as in \eqref{item:trapezium} hold, but with $\alpha$ and $\beta$, and $B$ and $C$ switched.
    \item\label{item:generic_position} If neither $AB$ and $CD$ nor $AC$ and $BD$ are parallel, let $\Delta = (p_{AD} + p_{CB})^2 - 4p_{AB}p_{CD}$. If $\Delta \geq 0$, the solutions to Equation~\eqref{alpha_beta_relations} are
    \begin{align}
    (\alpha, \beta) &= \left(\frac{ 2p_{AB} + p_{BC} + p_{DA} \pm \sqrt{\Delta}}{2(p_{AB} + p_{BC} + p_{CD} + p_{DA})}, \frac{ 2p_{AC} + p_{CB} + p_{DA} \mp \sqrt{\Delta}}{2(p_{AC} + p_{CB} + p_{BD} + p_{DA})} \right)
    \end{align}
    If $\Delta < 0$, there are no real solutions.
    \end{enumerate}
\end{proposition}

\begin{proof}
If $A=B=C=D$, then $E(\alpha, \beta) = A$, from which \eqref{item:points_coincide} follows immediately.

If $A,B,C,D$ do not all coincide, and lie on a line, we have
\begin{align*}
    A &= P + aV,\\
    B &= P + bV,\\
    C &= P + cV,\\
    D &= P + dV,
\end{align*} 
for some point $P$ on the line, a non-zero vector $V$ parallel to the line, and $a,b,c,d\in\R$, not all equal. 
Equation~\eqref{alpha_beta_relations} then reduces to 
$$\mathbf{0} = P + (a + \alpha(c-a) + \beta(b-a) + \alpha \beta(a -b - c + d))V,$$
which has no solutions if $P$ is not in the span of $\{V\}$ (i.e. if the line does not pass through the origin). 
Otherwise, we may assume the point is the origin; $P=\mathbf{0}$.
In this case, Equation~\eqref{alpha_beta_relations} has the solution space 
\begin{align}
    \{(\alpha,\beta)\in \R^2\colon a + \alpha(c-a) + \beta(b-a) + \alpha \beta(a -b - c + d)=0\}, \label{eq:one_dim_soln}
\end{align}
which is one dimensional since if the coefficients of $\alpha$, $\beta$ and $\alpha\beta$ were zero, $a,b,c,d$ would all be equal.

To prove \eqref{item:parallelogram} - \eqref{item:generic_position}, we will use the antisymmetric bilinear form $\langl\;,\;\rangl: (\R^2)^2 \to \R$ defined by
\[\langl V,W\rangl = \det\left(\begin{smallmatrix}
            V_1 & W_1 \\
            V_2 & W_2
        \end{smallmatrix}\right) = p_{V,W}.\]
Note that $\{V,W\}$ is a basis for $\R^2$ if and only if $\langl V,W\rangl \neq 0$.

Suppose $ABDC$ is a non-degenerate parallelogram. Then for $V=B-A$ and $W=C-A$, we have $\langl V, W \rangl \neq 0$ and 
\begin{align*}
B &= A + V, \\
C &= A + W, \\
D &= A + V + W, \\
E(\alpha, \beta) &= A + \beta V + \alpha W\quad\quad(\text{by \eqref{eq:alpha_beta_alphabeta}}).
\end{align*}
Since $\{V,W\}$ is a basis for $\R^2$, $E(\alpha, \beta) = \mathbf{0}$ if and only if
\[\langl E(\alpha, \beta), V\rangl =\langl E(\alpha, \beta), W\rangl = 0. \]
which is equivalent to the system of equations
\begin{align*}
0 = \langl A, V \rangl + \alpha \langl W, V \rangl = \langl A, B \rangl + \alpha \langl C-A, B-A \rangl, \\
0 = \langl A, W \rangl + \beta \langl V, W \rangl = \langl A, C \rangl + \beta \langl B-A, C-A \rangl,
\end{align*}
which has the unique solution 
\begin{align*}
(\alpha, \beta) &= \left(-\frac{\langl A, B \rangl}{\langl C-A, B-A \rangl}, -\frac{\langl A, C \rangl}{\langl B-A, C-A \rangl}\right) \\
&=\left(\frac{p_{AB}}{p_{BC} + p_{CA} + p_{AB}}, \frac{p_{AC}}{p_{CB} + p_{BA} + p_{AC}}\right)
\\
&=\frac{1}{p_{BC} + p_{CA} + p_{AB}}\left(p_{AB}, p_{CA}\right).
\end{align*}

Now suppose $ABDC$ is a trapezium, with the edges $AC$ and $BD$ parallel, and $AB$ and $CD$ not parallel, so that $\langl A-C,B-D\rangl=0$ and $\langl A-B,C-D\rangl\neq 0$. We note for future reference that
\[0= \langl A-C,B-D\rangl =p_{AB} - p_{AD}- p_{CB} + p_{CD},\]
which implies
\begin{align}
p_{CB} + p_{AD} &= p_{AB} + p_{CD}. \label{eq:trapezium_plucker_identity}
\end{align}
Since $A-C$ and $B-D$ are both parallel to their difference $V = A-B-C+D$, solutions to Equation~\eqref{alpha_beta_relations} satisfy
\begin{align}
0 &= \langl E(\alpha, \beta), V \rangl \nonumber\\
&=\langl A -\beta (A-B) -\alpha(A - C) + \alpha \beta V, V\rangl \nonumber\\
&= \langl A, V \rangl - \beta\langl A-B, A-B -C+D\rangl \nonumber\\
&= -\langl A, B + C - D \rangl + \beta\langl A-B, C-D \rangl, \nonumber
\end{align}
and since $\langl A-B,C-D\rangl\neq 0$,
\begin{align}\beta = \frac{\langl A, B + C - D \rangl}{\langl A-B, C-D \rangl}=\frac{p_{AB}+p_{AC}+p_{DA}}{p_{AC} + p_{CB} +  p_{BD} + p_{DA}} \label{eq:beta_trapezium}
\end{align}
Now let $W = A-B$, and note that $\{V,W\}$ is a basis for $\R^2$, since $\langl V, W\rangl = \langl A-B,C-D\rangl\neq 0$. Note also that
\begin{align}
 \beta \langl V, W\rangl &= p_{AB}+p_{AC}+p_{DA}. \label{eq:beta_VW}
\end{align}
With $\beta$ as in \eqref{eq:beta_trapezium},
\[\langl E(\alpha, \beta), V \rangl = 0 \]
so the solutions to $E(\alpha, \beta)=\mathbf{0}$ will be given by \eqref{eq:beta_trapezium} and the value(s) of $\alpha$ which satisfies
\[\langl E(\alpha, \beta), W \rangl = 0. \]
We calculate:
\begin{align}
0 &= \langl E(\alpha, \beta), W \rangl \nonumber\\
&=\langl A -\beta (A-B) -\alpha(A - C) + \alpha \beta V, W\rangl \nonumber\\
&= \langl A, W \rangl - \alpha\langl A-C, W\rangl + \alpha\beta \langl V, W\rangl \nonumber\\
&= -\langl A, B \rangl - \alpha(\langl A-C, A-B\rangl -\beta \langl V, W\rangl )\nonumber\\
&= -p_{AB} - \alpha(p_{AC}+p_{CB}+p_{BA} - (p_{AB}+p_{AC}+p_{DA} ))\;\;\;(\text{by }\eqref{eq:beta_VW})\nonumber\\
&= -p_{AB} - \alpha(2p_{BA} + p_{CB}+ p_{AD} )\nonumber\\
&= -p_{AB} - \alpha(p_{BA} + p_{CD} )\;\;\;(\text{by }\eqref{eq:trapezium_plucker_identity}).\nonumber
\end{align}
Thus if $p_{BA} + p_{CD} = 0$, there is no solution for $\alpha$ unless $p_{AB} = 0$, in which case $\alpha$ can be arbitrary. Note that $p_{BA} + p_{CD} = p_{AB}=0$ imply $p_{CD}=0$, so both $A$ and $B$, and $C$ and $D$ are parallel, hence the lines through them intersect at $\mathbf{0}$.

Otherwise, if  $p_{BA} + p_{CD} \neq 0$,
\[\alpha = \frac{p_{AB}}{p_{AB} + p_{DC} }.\]

Finally, suppose neither pair of opposite edges of $ABDB$ are parallel, so that $\langl A-C,B-D\rangl\neq 0$ and $\langl A-B,C-D\rangl\neq 0$.

If $(\alpha, \beta)$ is a solution to Equation~\eqref{alpha_beta_relations}, then the sets
\begin{align}
\{(1-\alpha)A + \alpha C, (1-\alpha)B + \alpha D\}, \label{eq:linearly dependent_alpha} \\
\{(1-\beta)A + \beta B, (1-\beta)C + \beta D\}
\end{align}
are both linearly dependent. The linear dependence of \eqref{eq:linearly dependent_alpha} is equivalent to
\begin{align}
0 &= \langl(1-\alpha)A + \alpha C, (1-\alpha)B + \alpha D\rangl \nonumber\\
&= (1-\alpha)^2\langl A , B\rangl + (1-\alpha)\alpha(\langl A, D\rangl + \langl C, B\rangl) + \alpha^2\langl C, D\rangl. %
\end{align}
Note that if $a -b + c \neq 0$, the equation 
\[0 = a(1-x)^2 + b(1-x)x + cx^2 = (a-b+c)x^2 - (2a - b)x + a\]
has no real solutions in $x$ if the discriminant
\[\Delta = (2a-b)^2 - 4(a-b+c)a = b^2 - 4ac\]
is negative, and otherwise
\[x = \frac{2a-b \pm \sqrt{\Delta}}{2(a -b + c)}.\]

Since
\[\langl A , B\rangl - (\langl A, D\rangl + \langl C, B\rangl) + \langl C, D\rangl = \langl A - C, B - D\rangl \neq 0,\]
we conclude that if 
\begin{align}
\Delta_1 &= (\langl A, D\rangl - \langl C, B\rangl)^2 - 4\langl A , B\rangl\langl C, D\rangl \nonumber \\
&= (p_{AD} + p_{CB})^2 - 4p_{AB}p_{CD} 
\end{align}
is non-negative, solutions to Equation~\eqref{alpha_beta_relations} must have
\begin{align}\alpha 
&= \frac{ 2\langl A , B\rangl - \langl A, D\rangl - \langl C, B\rangl + \sigma_1 \sqrt{\Delta_1}}{2(\langl A -C , B - D\rangl)} \label{eq:alpha_generic_bilinear} \\
& = \frac{ 2p_{AB} + p_{BC} + p_{DA} + \sigma_1 \sqrt{\Delta_1}}{2(p_{AB} + p_{BC} + p_{CD} + p_{DA})} \label{eq:alpha_generic_plucker} \\
\sigma_1 & \in \{1,-1\} \nonumber
\end{align}
and otherwise, if $\Delta_1<0$, there are no solutions.

By symmetry, if 
\begin{align}
\Delta_2 &= (\langl A, D\rangl - \langl B, C\rangl)^2 - 4\langl A , C\rangl\langl B, D\rangl \nonumber \\
&= (p_{AD} + p_{BC})^2 - 4p_{AC}p_{BD}
\end{align}
is non-negative, solutions to Equation~\eqref{alpha_beta_relations} must have
\begin{align}
\beta &= \frac{ 2\langl A , C\rangl - \langl A, D\rangl - \langl B, C\rangl + \sigma_2 \sqrt{\Delta_2}}{2(\langl  A - B , C - D\rangl)} \label{eq:beta_generic_bilinear} \\
&= \frac{ 2p_{AC} + p_{CB} + p_{DA} + \sigma_2 \sqrt{\Delta_2}}{2(p_{AC} + p_{CB} + p_{BD} + p_{DA})} \label{eq:beta_generic_plucker}\\
\sigma_2 & \in \{1,-1\} \nonumber
\end{align}
and otherwise, if $\Delta_2<0$, there are no solutions.

Observe that the Pl\"ucker relation $p_{AD}p_{BC} = p_{AB}p_{CD} + p_{AC}p_{BD}$ implies $\Delta_1 = \Delta_2$, so we can define $\Delta=\Delta_1=\Delta_2$.

If $\Delta > 0$, the two choices of sign in \eqref{eq:alpha_generic_plucker} and \eqref{eq:beta_generic_plucker} produce four potential solutions 
\begin{align}
(\alpha, \beta) &= \left(\frac{ 2p_{AB} + p_{BC} + p_{DA} + \sigma_1 \sqrt{\Delta}}{2(p_{AB} + p_{BC} + p_{CD} + p_{DA})}, \frac{ 2p_{AC} + p_{CB} + p_{DA} + \sigma_2 \sqrt{\Delta}}{2(p_{AC} + p_{CB} + p_{BD} + p_{DA})} \right)\\
\sigma_1, \sigma_2 &\in \{1,-1\} \nonumber
\end{align}
and it remains to determine which combinations of sign choices (if any) satisfy Equation~\eqref{alpha_beta_relations}.

For $\alpha$ defined by Equation~\eqref{eq:alpha_generic_plucker}, $(1-\alpha)A + \alpha C$ and $(1-\alpha)B + \alpha D$ lie on a common line $\ell$ through the origin, so for any vector $U$ not on the line, we have 
$$E(\alpha,\beta)= \mathbf{0} \iff \langl E(\alpha, \beta), U\rangl = 0.$$
We claim that $V = A - B - C + D$ is not on $\ell$. Note that
\[[(1-\alpha)A + \alpha C] - [(1-\alpha)B + \alpha D]  = A -B  -\alpha V\]
lies on $\ell$, and
\[\langl A -B  -\alpha V, V\rangl = \langl A-B, D-C\rangl \neq 0,\]
hence $V \not\in \ell$.

In the calculation below, we will need the fact that
\begin{align*}
\alpha\langl A-C, B - D\rangl &= \frac{1}{2}\left( 2\langl A , B\rangl - \langl A, D\rangl - \langl C, B\rangl + \sigma_1 \sqrt{\Delta}\right) \\
&= \frac{1}{2}\left(2p_{AB} + p_{BC} + p_{DA} + \sigma_1 \sqrt{\Delta}\right)\\
\beta\langl A-B, C - D\rangl &= \frac{1}{2}\left( 2\langl A , C\rangl - \langl A, D\rangl - \langl B, C\rangl + \sigma_2 \sqrt{\Delta}\right) \\
&= \frac{1}{2}\left(2p_{AC} + p_{CB} + p_{DA} + \sigma_2 \sqrt{\Delta}\right)
\end{align*}
which follows from \eqref{eq:alpha_generic_bilinear} and \eqref{eq:beta_generic_bilinear}.
\begin{align*}
0 &= \langle\langl E(\alpha, \beta), V \rangl \nonumber\\
&=\langl A -\alpha(A - C) -\beta (A-B) + \alpha \beta V, V\rangl \nonumber\\
&= \langl A, V \rangl - \alpha\langl A-C, V\rangl - \beta \langl A-B, V\rangl \nonumber\\
&= \langl A, -B-C+D \rangl + \alpha\langl A-C, B - D\rangl + \beta \langl A-B, C - D\rangl \nonumber\\
&= -(p_{AB} + p_{AC} + p_{DA}) + \frac{1}{2}\left(2p_{AB} + p_{BC} + p_{DA} + \sigma_1 \sqrt{\Delta}\right)\\
&\hspace{110pt} +\frac{1}{2}\left(2p_{AC} + p_{CB} + p_{DA} + \sigma_2 \sqrt{\Delta}\right) \nonumber\\
&= \frac{\sigma_1 + \sigma_2}{2} \sqrt{\Delta}.
\end{align*}
Thus $\sigma_1$ and $\sigma_2$ must be of opposite sign. 

If $\Delta=0$, there is only one potential solution, and the calculation above shows that it is valid.
\end{proof}

As illustrated in Table~\ref{table:solutions-alpha-beta-examples},
solutions $(\alpha,\beta)$ to Equation~\eqref{alpha_beta_relations} behave in many different ways:
sometimes the solutions are complex, 
sometimes they are real but not in $[0,1]^2$,
sometimes they are real but only one of them is in $[0,1]^2$,
sometimes they are real and both solutions are in $[0,1]^2$,
sometimes the two solutions are rational
and sometimes there is only one solution and it must be rational.

\begin{table}
\[
    \begin{array}{c|c|c}
        \ABCD{A_1}{A_2}{B_1}{B_2}{C_1}{C_2}{D_1}{D_2}
        &\Delta
		&\text{solutions } (\alpha,\beta)\\[1mm]
		\hline
		&&\\[-2mm]
        \ABCD{0}{2}{1}{-2}{1}{0}{0}{-1}
        & -4 &
        (1+i,\frac{3}{5}-\frac{1}{5} i) \in\C^2\setminus\R^2
        \\
		&&
        (1-i,\frac{3}{5}+\frac{1}{5} i) \in\C^2\setminus\R^2
		\\[2mm]
		\hline
		&&\\[-2mm]
        \ABCD{-1}{1}{4}{2}{3}{-2}{-3}{-4}
        & 72 &
		(\frac{19}{49}+\frac{3}{49} \, \sqrt{2},2+\frac{3}{2} \, \sqrt{2}) 
		\approx (0.47,4.12)\in\R^2\setminus[0,1]^2\\
		&&
		(\frac{19}{49}-\frac{3}{49} \, \sqrt{2},2-\frac{3}{2} \, \sqrt{2}) 
		\approx (0.30,-0.12)\in\R^2\setminus[0,1]^2
		\\[2mm]
		\hline
		&&\\[-2mm]
        \ABCD{-3}{-2}{-1}{-2}{0}{3}{4}{-2}
        & 481 & 
        (\frac{9}{50}+\frac{1}{50} \, \sqrt{481},\frac{29}{20}-\frac{1}{20} \, \sqrt{481})
        \approx (0.62,0.35)\in[0,1]^2\setminus\Q^2\\
		&=13\cdot 7&                                                         
        (\frac{9}{50}-\frac{1}{50} \, \sqrt{481},\frac{29}{20}+\frac{1}{20} \, \sqrt{481})
        \approx (-0.26,2.55)\in\R^2\setminus[0,1]^2
		\\[2mm]
		\hline
		&&\\[-2mm]
        \ABCD{-8}{6}{7}{-3}{2}{0}{-7}{-1}
        & 1792 & 
		(\frac{5}{8}+\frac{1}{8} \, \sqrt{7},\frac{5}{12}-\frac{1}{12} \, \sqrt{7}) 
		\approx (0.96,0.20)\in[0,1]^2\setminus\Q^2\\
		&=2^8\cdot 7&                                         
		(\frac{5}{8}-\frac{1}{8} \, \sqrt{7},\frac{5}{12}+\frac{1}{12} \, \sqrt{7}) 
		\approx (0.29,0.64)\in[0,1]^2\setminus\Q^2
		\\[2mm]
		\hline
		&&\\[-2mm]
        \ABCD{-1}{7}{-3}{5}{1}{-7}{5}{-9}
        & 100 &
		(\frac{8}{21},\frac{5}{6}) \approx (0.38,0.83) \in[0,1]^2\cap\Q^2\\
        &&
		(\frac{1}{2},0) \approx (0.50,0.00) \in[0,1]^2\cap\Q^2
		\\[2mm]
		\hline
		&&\\[-2mm]
        \ABCD{-3}{-2}{4}{-2}{-1}{-2}{0}{3}
        & 169 &
        (\frac{7}{10},\frac{4}{7}) \approx (0.70,0.57) \in[0,1]^2\cap\Q^2\\
        &&
        (2,\frac{1}{5}) \approx (2.00,0.20) \in\Q^2\setminus[0,1]^2
		\\[2mm]
		\hline
		&&\\[-2mm]
        \ABCD{-9}{-3}{-2}{8}{7}{9}{0}{-2}
        & 4096 &
		(\frac{39}{46},\frac{15}{16}) \approx (0.85,0.94) \in[0,1]^2\cap\Q^2\\
        &&
        \text{only one solution because $ABDC$ is a parallelogram}\\
		\hline
		&&\\[-2mm]
        \ABCD{0}{2}{2}{2}{1}{1}{2}{1}
        & 0 &
        \text{no solution because $ABDC$ is a trapezium with $AB\parallel CD$}\\
        &&\text{(case (v) in Prop.\ref{prop:alpha_beta_solutions})
              where $P=(2,0)$ and line $\ell$ passes through $\mathbf{0}$}\\
		\hline
		&&\\[-2mm]
        \ABCD{0}{0}{0}{0}{1}{0}{0}{1}
        & 0 &
        \{(0,\beta)\mid \beta\in\R\} 
        \text{ is a 1-dimensional solution space because $ABDC$ }\\
        &&
        \text{is a trapezium with $AB\parallel CD$
        (case (v) in Prop.\ref{prop:alpha_beta_solutions})
              where $P=\mathbf{0}$}\\
    \end{array}
\]
\caption{Solutions of Equation~\eqref{alpha_beta_relations} illustrating
some typical examples.}
\label{table:solutions-alpha-beta-examples}
\end{table}

We finish this section by observing that the quadrilateral $ABDC$ circumscribed
to a parabola reminds one of Poncelet's closure theorem \cite{zbMATH04030361}.
In geometry, this theorem states that whenever a polygon is inscribed in one conic
section and circumscribes another one, the polygon must
be part of an infinite family of polygons that are all inscribed in and
circumscribe the same two conics.
The special case of quadrilaterals, circles and parabolas is treated in more details in
\cite{zbMATH08141186}.

\subsection{Every pair of quadratic number field elements is a solution}
\label{sec:every-pair-of-quadratic-elements}
    
One can prove more generally that if $\alpha, \beta$ are two elements of the
same quadratic number field, then there exist $A,B,C,D \in \Z^2$ such that
$(\alpha,\beta)$ is a solution to Equation~\eqref{alpha_beta_relations}. 

\begin{lemma}\label{lem:a-solution-for-every-pair-of-quadratic-numbers}
For every pair $(\alpha,\beta)$ of elements of a quadratic number field $K$,
there exist four non-collinear points $A,B,C,D \in \Z^2$ such that
$(\alpha,\beta)$ is a solution to Equation~\eqref{alpha_beta_relations}. 

We can choose $A,B,C,D$ so that opposite sides of $ABDC$ are not parallel
if and only if $\alpha$ and $\beta$ are both irrational or both rational.
\end{lemma}

\begin{proof}
Let $\alpha,\beta\in K$ for some quadratic number field $K=\Q(\sqrt{d})$
for some square-free integer $d$ different from $0$ and $1$.
Let 
\[
    \xi = ((1-\alpha)(1-\beta), (1-\alpha)\beta, \alpha(1-\beta), \alpha\beta)^T \in K^4
\]
satisfying $(1,1,1,1)\cdot \xi = 1$.
Since $\{1, \sqrt{d} \}$ is a basis for $K$ considered as a rational vector space, 
there exists $V\in \Q^{4\times 2}$ such that $\xi=V\cdot (1,\sqrt{d})^T$.
The dimension of $\text{ColumnSpace}(V)^{\perp}$ is at least $4-2 = 2$. 
Thus, let $M=\ABCD{A_1}{A_2}{B_1}{B_2}{C_1}{C_2}{D_1}{D_2}\in \Z^{2\times 4}$ 
be any matrix whose rows are two linearly independent integer vectors 
in $\text{ColumnSpace}(V)^{\perp}$. 
The matrix $M$ has rank 2 and its row space is orthogonal to the column space of $V$, that is,
$ MV = \mathbf{0}\in \Q^{2\times2}$. Therefore,
Equation~\eqref{alpha_beta_relations} is satisfied:
\begin{align*}
    &(1-\alpha)\left[(1-\beta)\left(\begin{smallmatrix} A_1\\A_2 \end{smallmatrix}\right) 
                   +\beta    \left(\begin{smallmatrix} B_1\\B_2 \end{smallmatrix}\right) \right]
       +\alpha\left[(1-\beta)\left(\begin{smallmatrix} C_1\\C_2 \end{smallmatrix}\right) 
                   +\beta    \left(\begin{smallmatrix} D_1\\D_2 \end{smallmatrix}\right) \right]\\
    &=M\xi
     =M\cdot V\cdot\left(\begin{smallmatrix} 1\\\sqrt{d} \end{smallmatrix}\right)
     = \mathbf{0}\cdot\left(\begin{smallmatrix} 1\\\sqrt{d} \end{smallmatrix}\right)
     = \left(\begin{smallmatrix} 0\\0 \end{smallmatrix}\right).
\end{align*}
Suppose the columns $A,B,C,D$ of $M$ are collinear. 
Then the span of $\text{RowSpace}(M) \cup \{(1,1,1,1)\}$ has dimension less than 3, and since $\dim_\Q \text{RowSpace}(M) = 2$, 
we have $(1,1,1,1) \in \text{RowSpace}(M)$.
Thus, there exists $u\in\Z^{2\times1}$ such that $u\cdot M=(1,1,1,1)$.
However, 
\[
    1
    =(1,1,1,1)\cdot \xi 
    =u M\cdot V\left(\begin{smallmatrix} 1\\\sqrt{d} \end{smallmatrix}\right) 
    =u \cdot \mathbf{0}\cdot\left(\begin{smallmatrix} 1\\\sqrt{d} \end{smallmatrix}\right) 
    =0 
\]
a contradiction. We conclude that the columns of $M$ are non-collinear.

Now suppose that $\alpha$ and $\beta$ are both irrational. Since the columns of $M$ are non-collinear, cases (i) and (ii) of Proposition~\ref{prop:alpha_beta_solutions} do not hold. Since neither $\alpha$ nor $\beta$ is rational, we can exclude cases (iii) - (v), so (vi) must hold.

If both $\alpha$ and $\beta$ are rational, the first column of $V$ is $\xi$, and the second is zero. Thus the rows of $M$ are only required to be orthogonal to a single vector, $\xi$. Let $M \in \Z^{2\times 4}$ be arbitrary, and set $M' = M - M\cdot \xi\cdot (1,1,1,1)$.
Then 
\begin{align*}
M'\cdot \xi &= (M- M\cdot \xi\cdot (1,1,1,1))\cdot \xi = M\cdot \xi - M\cdot \xi\cdot (1,1,1,1)\cdot \xi = M\cdot \xi - M\cdot \xi\cdot 1 = \mathbf{0}.
\end{align*}
Since the columns of $M'$ are obtained by translating the columns of $M$ by $M\cdot \xi \in \Q^2$, the shape of the quadrilateral determined by the columns is unchanged, so we can certainly ensure that opposite sides are not parallel. By multiplying $M'$ by an integer, we can obtain a matrix in $\Z^{2\times 4}$ whose columns determine $A, B, C, D$.

Conversely, suppose exactly one of $\alpha$ and $\beta$ is rational; without loss of generality, we may assume it is $\beta$. From Proposition~\ref{prop:alpha_beta_solutions} and its proof, (iv) in the proposition holds, with the lines through $AB$ and $CD$ passing through the origin. Thus the quadrilateral $ABDC$ has a pair of parallel opposite sides.
\end{proof}

\section{Stripe densities always exist and satisfy Equation~\eqref{alpha_beta_relations}}
\label{sec:stripes-densities-satisfy-equation}

In this section, we prove a stronger result than Theorem~\ref{thm:not-periodic}.
Under the same hypothesis, we show that densities of stripes in valid
configurations over a set of Wang tiles that determines the quadrilateral $ABDC$ 
must exist and be a solution to Equation~\eqref{alpha_beta_relations}.

\begin{proposition}\label{prop:sufficient-condition-quadratic-equation}
Let $\Tcal\subset I\times J\times I\times J$ be a set of Wang tiles
having horizontal and vertical stripes with partition
$ \Tcal = \Tcal_\square \cup \Tcal_\boxbar \cup \Tcal_\boxminus \cup \Tcal_\boxplus$.
Suppose that $\Tcal$ determines the quadrilateral $ABDC$ for some
$A,B,C,D\in\R^2$.
If Equation~\eqref{alpha_beta_relations} has finitely many solutions,
then for every valid configuration $w:\Z^2\to\Tcal$ in $\Omega_\Tcal$,
the density $\alpha$ of tiles with vertical stripes (tiles in
    $\Tcal_\boxbar\cup\Tcal_\boxplus$)
    and
    the density $\beta$ of tiles with horizontal stripes (tiles in
    $\Tcal_\boxminus\cup\Tcal_\boxplus$)
    both exist, and $(\alpha,\beta)$ is a solution to
    Equation~\eqref{alpha_beta_relations}.
\end{proposition}
\begin{proof}
    Let $w:\Z^2\to\Tcal$ be a valid configuration over the set of tiles $\Tcal$.
    For every integer $K\geq0$, we 
    consider the subdomain $S_K=\{(i,j)\in\Z^2\colon-K\leq i,j\leq K\}$
    of cardinality $(2K+1)^2$.
    For every integer $K\geq0$, let 
    \[
        \alpha_K = \frac{\#\left(w^{-1}(\Tcal_\boxbar \cup \Tcal_\boxplus)\cap S_K\right)}
                        {(2K+1)^2}
        \quad
        \text{ and }
        \quad
        \beta_K  = \frac{\#\left(w^{-1}(\Tcal_\boxminus \cup \Tcal_\boxplus)\cap S_K\right)}
                        {(2K+1)^2}
    \]
    be respectively 
    the density of vertical stripes 
    (tiles in $\Tcal_\boxbar \cup \Tcal_\boxplus$) 
    and
    the density of horizontal stripes 
    (tiles in $\Tcal_\boxminus \cup \Tcal_\boxplus$) 
    in the configuration $w$
    restricted to the subdomain $S_K$.
    Thus, in the domain $S_K$, the density of tiles in the configuration $w$ 
    belonging to each subset of the partition 
    $ \Tcal = \Tcal_\square \cup \Tcal_\boxbar \cup \Tcal_\boxminus \cup \Tcal_\boxplus$
    is given by the following formulas:
    \begin{align*}
        \frac{\#\left(w^{-1}(\Tcal_\boxminus)\cap S_K\right)}{\# S_K}
                         &=(1-\alpha_K)\beta_K,
        &
        \frac{\#\left(w^{-1}(\Tcal_\square)\cap S_K\right)}{\# S_K}
        &=(1-\alpha_K)(1-\beta_K),\\
        \frac{\#\left(w^{-1}(\Tcal_\boxplus)\cap S_K\right)}{\# S_K}
        &=\alpha_K\beta_K,    
        &\frac{\#\left(w^{-1}(\Tcal_\boxbar)\cap S_K\right)}{\# S_K}
        &=\alpha_K(1-\beta_K).
    \end{align*}

    In what follows, we identify $w_n$ with its matrix encoding $M^{\vv,\hh}(w_n)$
    following Section~\ref{sec:wang-tiles-encoded-as-nx4-matrices}.
    On the one hand, for every integer $K\geq0$, we have
    \begin{equation}\label{eq:first-equation-in-proof}
\begin{aligned}
    \frac{1}{\#S_K}
    \sum_{n\in S_K}
        U\cdot w_n
        \left(\begin{smallmatrix}
        1\\1\\-1\\-1
        \end{smallmatrix}\right)
    &=
    \frac{1}{\#S_K}
    \Big(
        \left(\begin{smallmatrix}
        A_1\\A_2
        \end{smallmatrix}\right)
        \cdot \#\{n\in S_K\colon w_n\in\Tcal_\square\}  
        +
        \left(\begin{smallmatrix}
        B_1\\B_2
        \end{smallmatrix}\right)
        \cdot \#\{n\in S_K\colon w_n\in\Tcal_\boxminus\}  \\
        &\qquad+
        \left(\begin{smallmatrix}
        C_1\\C_2
        \end{smallmatrix}\right)
        \cdot \#\{n\in S_K\colon w_n\in\Tcal_\boxbar\}
        +
        \left(\begin{smallmatrix}
        D_1\\D_2
        \end{smallmatrix}\right)
        \cdot \#\{n\in S_K\colon w_n\in\Tcal_\boxplus\}\Big)\\
    &= 
     \left(\begin{smallmatrix} A_1\\A_2 \end{smallmatrix}\right)     (1-\alpha_K)(1-\beta_K)
    +\left(\begin{smallmatrix} B_1\\B_2 \end{smallmatrix}\right) (1-\alpha_K)\beta_K
    +\left(\begin{smallmatrix} C_1\\C_2 \end{smallmatrix}\right) \alpha_K(1-\beta_K)
    +\left(\begin{smallmatrix} D_1\\D_2 \end{smallmatrix}\right)     \alpha_K\beta_K\\
    &= 
     (1-\alpha_K)\left[\left(\begin{smallmatrix} A_1\\A_2 \end{smallmatrix}\right)     (1-\beta_K)
     +\left(\begin{smallmatrix} B_1\\B_2 \end{smallmatrix}\right) \beta_K\right]
     +\alpha_K\left[\left(\begin{smallmatrix} C_1\\C_2 \end{smallmatrix}\right) (1-\beta_K)
         +\left(\begin{smallmatrix} D_1\\D_2  \end{smallmatrix}\right) \beta_K\right].
\end{aligned}
\end{equation}

On the other hand,
    let $\{e_1,e_2,e_3,e_4\}$ be the canonical base of $\R^4$.
    Because $w$ is a valid configuration,
    Remark~\ref{rem:matrix-encoding-of-valid-configuration} applies
    and the following sum is telescoping:
    \begin{align*}
        \sum_{n\in S_K}
        w_{n} 
        \left(\begin{smallmatrix}
        1\\1\\-1\\-1
        \end{smallmatrix}\right)
        &=
        \sum_{-K\leq i,j \leq K}
          w_{i,j} e_1 
        - w_{i,j} e_3
        + w_{i,j} e_2
        - w_{i,j} e_4\\
        &=
        \sum_{-K\leq i,j \leq K}
          w_{i,j} e_1 
        - w_{i-1,j} e_1
        + w_{i,j} e_2
        - w_{i,j-1} e_2\\
        &=
        \sum_{j=-K}^{K}
        \sum_{i=-K}^{K}
        \left(
          w_{i,j} e_1 
        - w_{i-1,j} e_1
        \right)
        +
        \sum_{i=-K}^{K}
        \sum_{j=-K}^{K}
        \left(
          w_{i,j} e_2
        - w_{i,j-1} e_2
        \right)\\
        &=
        \sum_{j=-K}^{K}
        \left(
          w_{K,j} e_1 
        - w_{-K-1,j} e_1
        \right)
        +
        \sum_{i=-K}^{K}
        \left(
          w_{i,K} e_2
        - w_{i,-K-1} e_2
        \right).
    \end{align*}
    Let $\rho= \max_{\tau\in\Tcal} \Vert M_\tau \Vert$ be the maximum norm of the
    matrix encodings of the Wang tiles for the sup-norm 
    $\Vert\cdot\Vert=\Vert\cdot\Vert_\infty$
    and let $\Delta=4\cdot\rho\cdot\Vert U\Vert$.
    The following is bounded by a constant
    \begin{equation}\label{eq:sum_is_bounded_by_a_constant}
    \begin{aligned}
        &
        \frac{1}{\# S_K}
        \left\Vert \sum_{n\in S_K} U\cdot w_{n} 
        \left(\begin{smallmatrix}
        1\\1\\-1\\-1
        \end{smallmatrix}\right)
        \right\Vert\\
        &\leq
        \frac{\Vert U\Vert}{(2K+1)^2}
        \left\Vert
        \sum_{j=-K}^{K}
        \left(
          w_{K,j} e_1 
        - w_{-K-1,j} e_1
        \right)
        +
        \sum_{i=-K}^{K}
        \left(
          w_{i,K} e_2
        - w_{i,-K-1} e_2
        \right)
        \right\Vert\\
        &\leq
        \frac{\Vert U\Vert}{(2K+1)^2}
        \left(
        \sum_{j=-K}^{K}
        \left(
        \Vert w_{K,j} e_1\Vert + \Vert w_{-K-1,j} e_1\Vert
        \right)
        +
        \sum_{i=-K}^{K}
        \left(
        \Vert w_{i,K} e_2\Vert + \Vert w_{i,-K-1} e_2\Vert
        \right)\right)\\
        &\leq
        \frac{\Vert U\Vert}{2K+1}
        \left(
        \max_{\tau\in\Tcal}
        \Vert M_\tau e_1\Vert 
        +
        \max_{\tau\in\Tcal}
        \Vert M_\tau e_1\Vert 
        +
        \max_{\tau\in\Tcal}
        \Vert M_\tau e_2\Vert 
        +
        \max_{\tau\in\Tcal}
        \Vert M_\tau e_2\Vert \right)\\
        &=
        \frac{\Vert U\Vert}{2K+1}
        \left(
        2\max_{\tau\in\Tcal}
        \Vert M_\tau e_1\Vert 
        +
        2\max_{\tau\in\Tcal}
        \Vert M_\tau e_2\Vert\right)
        \leq
        \frac{4\rho\Vert U\Vert}{2K+1}
        =
        \frac{\Delta}{2K+1}.
    \end{aligned}
\end{equation}
Combining \eqref{eq:first-equation-in-proof} and \eqref{eq:sum_is_bounded_by_a_constant}, 
and assuming the norm is the sup-norm,
we obtain the pair of estimates
\begin{equation}\label{eq:two-nice-inequalities-bounded-by-Delta-over-2K}
    \begin{cases}
        \left|(1-\alpha_K)\left(A_1(1-\beta_K) +B_1 \beta_K\right)
        +\alpha_K\left(C_1 (1-\beta_K) +D_1   \beta_K\right)\right|
        = \Ocal(K^{-1}),\\
        \left|(1-\alpha_K)\left(A_2(1-\beta_K) +B_2 \beta_K\right)
        +\alpha_K\left(C_2 (1-\beta_K) +D_2   \beta_K\right)\right|
        = \Ocal(K^{-1}).
    \end{cases}
\end{equation}

Define $P: [0,1]^2 \to \R_{\geq 0}$ by
\begin{align*}
P(\alpha, \beta) =& \left|(1-\alpha)\left(A_1(1-\beta) +B_1 \beta\right)
        +\alpha\left(C_1 (1-\beta) +D_1   \beta\right)\right| \\
        & +
        \left|(1-\alpha)\left(A_2(1-\beta) +B_2 \beta\right)
        +\alpha\left(C_2 (1-\beta) +D_2   \beta\right)\right|
\end{align*}
Then $P$ is continuous, and its set of zeros $\mathcal{Z}_P = P^{-1}(0)$
is the set of solutions to Equation~\eqref{alpha_beta_relations}.
By hypothesis, Equation~\eqref{alpha_beta_relations} has finitely many solutions $(\alpha, \beta)$.
If we set $x_K = (\alpha_K, \beta_K) \in [0,1]^2$ for every integer $K\geq0$, 
then \eqref{eq:two-nice-inequalities-bounded-by-Delta-over-2K} implies $\lim_{K\to\infty}P(x_K) = 0$.

Since $\alpha_K = \frac{n_K}{2K+1}$, where $n_K$ is the number of vertical stripes in $S_K$, and since $n_{K+1} = n_K + d_K$ for some $d_K \in \{0,1,2\}$,
$$\alpha_{K+1} - \alpha_K = \frac{n_K + d_K}{2K+3} - \frac{n_K}{2K+1} = \frac{-2n_K + d_K(2K+1)}{(2K+1)(2K+3)} = \Ocal(1/K).$$
The same estimate holds for $\beta_{K+1} - \beta_K$, so $\lim_{K\to\infty}\|x_{K+1} - x_K\| = 0$.

Given $\epsilon > 0$, let $X = [0,1]^2 \smallsetminus \cup_{z\in \mathcal{Z}_P}B_\epsilon(z)$, and $\delta = \min_{y \in X}P(y),$
which is positive, by continuity and compactness. Then 
\[P(x_K)< \delta \implies x_K\in [0,1]^2 \smallsetminus X = \cup_{x\in \mathcal{Z}_P}B_\epsilon(x).\] For $K$ sufficiently large and $\epsilon$ sufficiently small, we cannot have $x_K \in B_\epsilon(z_1)$ and $x_{K+1} \in B_\epsilon(z_2)$ for distinct $z_1, z_2 \in \mathcal{Z}_P$, since $\mathcal{Z}_P$ is finite. Thus for some $z_0 \in \mathcal{Z}_P$, we have that for any $\epsilon > 0 $, $x_K\in B_\epsilon(z_0)$ for $K$ sufficiently large, so $\lim_{K\to\infty} x_K = z_0$.
\end{proof}

Following Proposition~\ref{prop:sufficient-condition-quadratic-equation},
for each example in the Section~\ref{sec:applications}, 
the stripe densities $(\alpha,\beta)$
must satisfy $\alpha=\beta$ and the equation
$0=(A_i-B_i-C_i+D_i)\alpha^2+(B_i+C_i-2A_i)\alpha + A_i$
for $i\in\{1,2\}$.
These equations and their solutions are listed in
Table~\ref{table:densities-for-all-examples}.

\begin{table}
    \[
    \begin{array}{l|c|c|c|c}
        \text{Example}
        &\ABCD{A_1}{A_2}{B_1}{B_2}{C_1}{C_2}{D_1}{D_2}
        &\Delta
        &\begin{array}{c}
            \text{Equations}\\
            \text{for }i\in\{1,2\}
            \end{array}
        &\begin{array}{c}
            \text{Stripe densities}\\
            (\alpha,\beta)
            \end{array}
            \\
		\hline
		\hline
        &&&\\[-2mm]
        \text{Penrose}
        &\ABCD{0}{ 1}{-2}{ 1}{2}{ 0}{0}{ -1} & 5 &
        \begin{array}{l}
        0= 0\alpha^2+0\alpha + 0\\
        0= -\alpha^2-\alpha + 1
        \end{array}
        &\alpha=\beta=\varphi^{-1}
        \\[3mm]
        \hline
        &&&\\[-2mm]
        \text{Ammann}
        &\ABCD{0}{1}{1}{1}{-1}{0}{0}{-1} & 5 &
        \begin{array}{l}
        0= 0\alpha^2+0\alpha + 0\\
        0= -\alpha^2-\alpha + 1
        \end{array}
        &\alpha=\beta=\varphi^{-1}
        \\[3mm]
        \hline
        &&&\\[-2mm]
        \text{19 self-similar $\Ucal$}
        &\ABCD{-1}{1}{0}{-2}{1}{1}{1}{-1} & 5 &
        \begin{array}{l}
        0= -\alpha^2+3\alpha - 1\\
        0=  \alpha^2-3\alpha + 1
        \end{array}
        &\alpha=\beta=\varphi^{-2}
        \\[3mm]
        \hline
        &&&\\[-2mm]
        \text{16 self-similar $\Zcal$}
        &\ABCD{-1}{-1}{-1}{2}{2}{-1}{1}{1} & 5 &
        \begin{array}{l}
        0= -\alpha^2+3\alpha - 1\\
        0= -\alpha^2+3\alpha - 1
        \end{array}
        &\alpha=\beta=\varphi^{-2}
        \\[3mm]
        \hline
        &&&\\[-2mm]
        \text{Jeandel-Rao $\Tcal_7$}
        &\ABCD{1}{1}{0}{1}{-1}{-2}{-1}{-1} & 5 &
        \begin{array}{l}
        0=  \alpha^2-3\alpha + 1\\
        0=  \alpha^2-3\alpha + 1
        \end{array}
        &\alpha=\beta=\varphi^{-2}
        \\[3mm]
        \hline
        &&&\\[-2mm]
        \text{Metallic mean}
        &\ABCD{1}{1}{1}{1-n}{1-n}{1}{-n}{-n} & n^2+4 &
        \begin{array}{l}
        0= -\alpha^2-n\alpha + 1\\
        0= -\alpha^2-n\alpha + 1
        \end{array}
        &\alpha=\beta=
        \tau_n^{-1}
        \\[3mm]
        \hline
    \end{array}
    \]
    \caption{
        Densities $(\alpha,\beta)$
        of vertical and horizontal stripes 
    for every example in Section~\ref{sec:applications}
    where $\varphi$ is the golden mean
    and $\tau_n$ is the $n$-th metallic mean.}
\label{table:densities-for-all-examples}
\end{table}

\section{New constructions of aperiodic sets of Wang tiles}
\label{sec:new-constructions}

In this section we design finite sets of Wang tiles with vertical and horizontal stripes that satisfy the linear constraints \eqref{alpha_beta_relations}. Such tiles can be constructed for any integer parameter values $A_i, B_i, C_i, D_i$ as long as \eqref{alpha_beta_relations} admits a solution $\alpha,\beta\in [0,1]$. The constructions provide new families of aperiodic tile sets.

\begin{theorem}
\label{thm:balancetheorem}
Let $A,B,C,D\in\Z^2$.
For every solution $\alpha,\beta\in [0,1]$ to Equation~\eqref{alpha_beta_relations},
there exists a finite Wang tile set
$\Tcal=\Tcal_\square\cup\Tcal_\boxbar\cup\Tcal_\boxminus\cup \Tcal_\boxplus$
with horizontal and vertical stripes such that
\begin{enumerate}
    \item[(i)] $\Tcal$ admits a tiling $w:\Z^2\to\Tcal$ where vertical
        stripes (\emph{i.e.}, tiles in
        $\Tcal_\boxbar\cup \Tcal_\boxplus$)
        have density $\alpha$, and the horizontal stripes (tiles in $\Tcal_\boxminus\cup \Tcal_\boxplus$) have density $\beta$, and
    \item[(ii)] $\Tcal$ determines the quadrilateral $ABDC$.
\end{enumerate}
\end{theorem}
\begin{proof}
The colors of the tiles are triplets $(s,t,u)$ of integers. The first coordinate $s$ is $0$ or $1$ depending on whether there is a stripe in the tile across the edge, with label $1$ indicating the presence of a stripe. The encoding functions $\vv$ and $\hh$ extract the last two coordinates of the vectors, i.e., they map $(s,t,u)\mapsto (t,u)$.

In the default tiling $w$, the second coordinate $t$ of a horizontal edge is a digit in a representation of a real number as a bi-infinite sequence of integers. As in~\cite{MR1417578}, we represent a real number $r$ as a bi-infinite sturmian or periodic sequence of integer symbols $\lfloor r\rfloor$ and $\lceil r\rceil$ with well-defined average $r$.  More precisely, for all $i\in\Z$, the $i$'th element of the sequence is
\begin{equation}
\label{eq:balancedrepresentation}
\bal{r}{i} = \lfloor (i+1)r\rfloor - \lfloor ir\rfloor.
\end{equation}
We call this the balanced representation of the number $r$, and $\bal{r}{i}$  is the $i$'th digit of the representation.

In our default tiling $w$, the second coordinates $t$ of the north and the south edges of a horizontal row of tiles form balanced representations of real numbers. If $r$ is the number represented by south edges then the north edges represent the number $r+\sigma_1$ if there is a horizontal stripe in the tiles of the row, and $r+\tau_1$ if there is no horizontal stripe, where
\begin{equation}
\label{eq:sigmatau}
\begin{array}{rcl}
\sigma_1 &=& (1-\alpha)B_1+\alpha D_1,\\
\tau_1 &=& (1-\alpha)A_1+\alpha C_1.
\end{array}
\end{equation}
Numbers $\sigma_1$ and $\tau_1$ are chosen this way because they satisfy
\begin{equation}
\label{eq:sigmatauaverage}
\sigma_1\cdot \beta + \tau_1\cdot (1-\beta)=0.
\end{equation}
This relation follows directly from \eqref{alpha_beta_relations} and \eqref{eq:sigmatau}.

The second coordinates $t$ of the vertical edges are ``carry'' symbols to facilitate the representations. The construction is similar to the one in~\cite{MR1417578} but, instead of multiplying real numbers by some rational constants, the tiles now add an irrational constant $\sigma_1$  or $\tau_1$. Another difference is that the present tiles also do a similar addition symmetrically in the vertical direction: the third coordinates $u$ of the east and the west edges of a vertical column form balanced representations of numbers, with corresponding ``carries'' given by the third components $u$ of the north and the south edges. In this direction, the constants added when moving from a column to the next are
\begin{equation}
\label{eq:sigmatau2}
\begin{array}{rcl}
\sigma_2 &=& (1-\beta)C_2+\beta D_2,\\
\tau_2 &=& (1-\beta)A_2+\beta B_2.
\end{array}
\end{equation}
By \eqref{alpha_beta_relations} and \eqref{eq:sigmatau2} these numbers satisfy the relation
\begin{equation}
\label{eq:sigmatauaverage2}
\sigma_2\cdot \alpha + \tau_2\cdot (1-\alpha)=0.
\end{equation}

Let us pick any real numbers $y_0$ and $z_0$ as the numbers to be represented on the south edges of the horizontal row 0 and the west edges of the vertical column 0, respectively. Define, for every $j\in\Z$,
\begin{equation}
\label{eq:numbers}
y_j = y_0+\sigma_1\cdot \lfloor j\beta\rfloor + \tau_1\cdot (j-\lfloor j\beta\rfloor).
\end{equation}
Note that the right-hand-side evaluates to $y_0$ when $j=0$, so that the notation is consistent. Note also that
\begin{equation}
\label{eq:increment}
y_{j+1}-y_{j} = \sigma_1\cdot \bal{\beta}{j} + \tau_1\cdot (1-\bal{\beta}{j})\\
=
\left\{
\begin{array}{ll}
\sigma_1, & \mbox{ if $\bal{\beta}{j}=1$,}\\
\tau_1, & \mbox{ if $\bal{\beta}{j}=0$.}
\end{array}
\right.
\end{equation}
In our default tiling $w$ the south edges of the horizontal row $j$ represent the number $y_j$,
and there is a horizontal stripe on row $j$ if and only if $\bal{\beta}{j}=1$.

Analogously for the columns, we define, for every $i\in\Z$,
\begin{equation}
\label{eq:numbers2}
z_i = z_0+\sigma_2\cdot \lfloor i\alpha\rfloor + \tau_2\cdot (i-\lfloor i\alpha\rfloor),
\end{equation}
so that
\begin{equation}
\label{eq:increment2}
z_{i+1}-z_{i}=
\left\{
\begin{array}{ll}
\sigma_2, & \mbox{ if $\bal{\alpha}{i}=1$,}\\
\tau_2, & \mbox{ if $\bal{\alpha}{i}=0$.}
\end{array}
\right.
\end{equation}

Let us next define the required ``carry'' symbols. For any $i,j\in\Z$ define
\begin{equation}
\label{eq:carry}
c_{i,j} =
\left\{
\begin{array}{ll}
\lfloor i y_j\rfloor-\lfloor i y_{j+1}\rfloor+
\lfloor i\alpha\rfloor D_1
+(i-\lfloor i\alpha\rfloor) B_1, & \mbox{ if $\bal{\beta}{j}=1$,}\\
\lfloor i y_j\rfloor-\lfloor i y_{j+1}\rfloor+
\lfloor i\alpha\rfloor C_1
+(i-\lfloor i\alpha\rfloor) A_1,
 & \mbox{ if $\bal{\beta}{j}=0$.}\\
\end{array}
\right.
\end{equation}
A direct calculation shows that, for all $i,j\in\Z$,
\begin{equation}
\label{eq:linearitycondition}
c_{i+1,j}+\bal{y_{j+1}}{i}-c_{i,j}-\bal{y_{j}}{i} =
\left\{
\begin{array}{ll}
A_1, & \mbox{ if $\bal{\alpha}{i}=0$ and $\bal{\beta}{j}=0$,}\\
B_1, & \mbox{ if $\bal{\alpha}{i}=0$ and $\bal{\beta}{j}=1$,}\\
C_1, & \mbox{ if $\bal{\alpha}{i}=1$ and $\bal{\beta}{j}=0$,}\\
D_1, & \mbox{ if $\bal{\alpha}{i}=1$ and $\bal{\beta}{j}=1$,}\\
\end{array}
\right.
\end{equation}
Analogously, the carries $d_{i,j}$ in the vertical direction are defined, for all $i,j\in\Z$, by
\begin{equation}
\label{eq:carry2}
d_{i,j} =
\left\{
\begin{array}{ll}
\lfloor j z_i\rfloor-\lfloor j z_{i+1}\rfloor+
\lfloor j\beta\rfloor D_2
+(j-\lfloor j\beta\rfloor) C_2, & \mbox{ if $\bal{\alpha}{i}=1$,}\\
\lfloor j z_i\rfloor-\lfloor j z_{i+1}\rfloor+
\lfloor j\beta\rfloor B_2
+(j-\lfloor j\beta\rfloor) A_2,
 & \mbox{ if $\bal{\alpha}{i}=0$.}\\
\end{array}
\right.
\end{equation}
As above, a direct calculation gives
\begin{equation}
\label{eq:linearitycondition2}
d_{i,j+1}+\bal{z_{i+1}}{j}-d_{i,j}-\bal{z_{i}}{j} =
\left\{
\begin{array}{ll}
A_2, & \mbox{ if $\bal{\alpha}{i}=0$ and $\bal{\beta}{j}=0$,}\\
B_2, & \mbox{ if $\bal{\alpha}{i}=0$ and $\bal{\beta}{j}=1$,}\\
C_2, & \mbox{ if $\bal{\alpha}{i}=1$ and $\bal{\beta}{j}=0$,}\\
D_2, & \mbox{ if $\bal{\alpha}{i}=1$ and $\bal{\beta}{j}=1$,}\\
\end{array}
\right.
\end{equation}

Now we are ready to define the tile set. We give the tiles in the matrix formalism, discussed in Section~\ref{sec:wang-tiles-encoded-as-nx4-matrices}, where the colors of the east, north, west and south edges of a tile are written, in this order, as the columns of a $3\times 4$ matrix. The tile set contains, for every $i,j\in\Z$, the Wang tile $t_{i,j}$ with the matrix representation
\begin{equation}
\label{eq:balancedtiles}
M_{t_{i,j}}
        = \left(\begin{array}{c|c|c|c}
            \bal{\beta}{j}     & \bal{\alpha}{i}       & \bal{\beta}{j} & \bal{\alpha}{i}\\
            \hline
            c_{i+1,j}      & \bal{y_{j+1}}{i}  & c_{i,j} & \bal{y_j}{i}\\
            \bal{z_{i+1}}{j} & d_{i,j+1} & \bal{z_i}{j} & d_{i,j}
        \end{array}\right).
\end{equation}
The tile $t_{i,j}$ is to be placed in cell $(i,j)$ in our default tiling $w$. This default tiling is valid because $\north(t_{i,j})=\south(t_{i,j+1})$ and $\east(t_{i,j})=\west(t_{i+1,j})$ for all $i,j\in\Z$.

The first row encodes the presence of stripes.
It follows that in the default tiling $w$ the density of vertical stripes is equal to $\alpha$ because the stripes are at columns $i$ with $\bal{\alpha}{i}=1$.
Similarly, there is a horizontal stripe on row $j$ if and only if $\bal{\beta}{j}=1$,
so that the horizontal stripes have density $\beta$ in $w$.

Because
$$
\IdTwoThree M_{t_{i,j}}  \OneOneMinusoneMinusOne
=
\left(\begin{array}{c}
        c_{i+1,j}+\bal{y_{j+1}}{i}-c_{i,j}-\bal{y_{j}}{i}\\
        d_{i,j+1}+\bal{z_{i+1}}{j}-d_{i,j}-\bal{z_{i}}{j}
        \end{array}\right),
$$
it follows from \eqref{eq:linearitycondition} and \eqref{eq:linearitycondition2} that
\eqref{eq:fourbulletpoints} holds. Thus, $\Tcal$ determines the quadrilateral $ABDC$.

Finally, we need to prove that we have only constructed a finite number of tiles. As all elements of matrices $M_{t_{i,j}}$ are integers, it is enough to prove that the absolute values of the elements are bounded by some constant.
\begin{itemize}
    \item The first row elements $\bal{\alpha}{i}$ and $\bal{\beta}{j}$ are in $\{0,1\}$.
    \item Elements $\bal{y_j}{i}$, for $i,j\in\Z$, are balanced digits of numbers $y_j$, and so it is enough to bound the numbers $y_j$. Erasing the floor functions in (\ref{eq:numbers}) gives an approximation
$$
y'_j=y_0+j(\sigma_1\cdot\beta+\tau_1\cdot (1-\beta))
$$
of $y_j$. By \eqref{eq:sigmatauaverage} we have that $y'_j=y_0$ is a constant, independent of $j$. As the difference of $y_j$ and $y'_j$ comes only from erasing the floor functions,
we have that
$|y_j-y_0| \leq |\sigma_1|+|\tau_1|$, proving that the numbers $|y_j|$ have a bound independent of $j$. More precisely, $y_j-y_0=(\tau_1-\sigma_1)\fr{j\beta}$ where, for any $r\in\R$, we denote $\fr{r}\in[0,1)$ for the fractional part $r-\lfloor r\rfloor$ of $r$.
An analogous reasoning shows that $\bal{z_i}{j}$ are bounded for $i,j\in\Z$, and that in fact
$z_i-z_0=(\tau_2-\sigma_2)\fr{i\alpha}$.

\item Finally, to bound $c_{i,j}$, for $i,j\in\Z$,  we use a similar calculation. Erasing the floor functions from the definition (\ref{eq:carry}) of numbers $c_{i,j}$ gives their approximations $c'_{i,j}$. If
$\bal{\beta}{j}=1$ we have
$$
c'_{i,j}
=i(y_j-y_{j+1}+\alpha D_1 +(1-\alpha) B_1)
=i(y_j-y_{j+1}+\sigma_1)=0,
$$
where we used (\ref{eq:sigmatau}), and also (\ref{eq:increment}) to observe that $y_j-y_{j+1}+\sigma_1=0$ when $\bal{\beta}{j}=1$. Similarly, if $\bal{\beta}{j}=0$ we have
$$
c'_{i,j}
=i(y_j-y_{j+1}+\alpha C_1 +(1-\alpha) A_1)
=i(y_j-y_{j+1}+\tau_1)=0.
$$
The change in $c_{i,j}$ caused by removing the floor functions in (\ref{eq:carry}) is at most $1+1+|D_1|+|B_1|$ when
$\bal{\beta}{j}=1$, and at most $1+1+|C_1|+|A_1|$ when
$\bal{\beta}{j}=0$. In any case, integers $c_{i,j}$ are bounded, independent of $i,j$.
More precisely, $c_{i,j}=\fr{iy_{j+1}}-\fr{iy_j}+\fr{i\alpha}(B_1-D_1)$ when $\bal{\beta}{j}=1$
and  $c_{i,j}=\fr{iy_{j+1}}-\fr{iy_j}+\fr{i\alpha}(A_1-C_1)$ when $\bal{\beta}{j}=0$.
An analogous reasoning shows that also the integers $d_{i,j}$ are bounded, with
$d_{i,j}=\fr{jz_{i+1}}-\fr{jz_i}-\fr{j\beta}(C_2-D_2)$ when $\bal{\alpha}{i}=1$ and
$d_{i,j}=\fr{jz_{i+1}}-\fr{jz_i}-\fr{j\beta}(A_2-B_2)$ when $\bal{\alpha}{i}=0$.
\end{itemize}
\end{proof}

\begin{remark}
\label{rem:effective}
Let $\Tcal$ be the tile set constructed in the proof of Theorem~\ref{thm:balancetheorem}. One may extend $\Tcal$
by adding more tiles represented by $3\times 4$ integer matrices that satisfy \eqref{eq:fourbulletpoints}, i.e., that determine the same quadrilateral ABDC as $\Tcal$ under the same  functions $\vv$ and $\hh$. The extended tile set still satisfies Theorem~\ref{thm:balancetheorem}. 
It follows that one can effectively find a tile set in Theorem~\ref{thm:balancetheorem}: calculate bounds on the elements of the matrices \eqref{eq:balancedtiles} as discussed in the end of the proof, and take in the tile set all matrices with so bounded elements that satisfy \eqref{eq:fourbulletpoints}. All tiles given by \eqref{eq:balancedtiles} for any $i,j\in\Z$  are guaranteed to be among the constructed tiles.
The tile set may contain tiles that are not given by \eqref{eq:balancedtiles} but the additional tiles do not break Theorem~\ref{thm:balancetheorem}. 
\end{remark}

In the following lemma, we provide  explicit upper bounds on the cardinalities of the
horizontal and vertical color sets with and without stripes in the constructed
tile set $\Tcal$. 
For simplicity, we do the calculations with the choice $y_0=z_0=0$. 

\begin{lemma}
Let $\Tcal$ be the set of Wang tiles constructed in Theorem~\ref{thm:balancetheorem}
appearing in the configuration defined with parameters $y_0=z_0=0$. 
The cardinalities of the horizontal and vertical color sets 
with and without stripes in $\Tcal$ satisfy
    \begin{equation}\label{eq:cardinalities-color-sets}
        \begin{array}{lcl}
        |I_0| & \leq &  (|A_1-C_1|+1)\lceil|R_2-S_2|+1\rceil,\\
        |I_1| & \leq &  (|B_1-D_1|+1)\lceil|R_2-S_2|+1\rceil,\\
        |J_0| & \leq &  (|A_2-B_2|+1) \lceil|P_1-Q_1|+1\rceil,\\
        |J_1| & \leq &  (|C_2-D_2|+1) \lceil|P_1-Q_1|+1\rceil,
        \end{array}
    \end{equation}
where $P=(1-\alpha)A+\alpha C$, $Q=(1-\alpha)B+\alpha D$,  $R=(1-\beta)A+\beta B$ and
$S=(1-\beta)C+\beta D$ are the points of the quadrilateral ABDC shown in Figure~\ref{fig:the-parabola-for-Ammann-case}.
\end{lemma}

\begin{proof}
Because $y_j=(\tau_1-\sigma_1)\fr{j\beta}$ we have that $y_j\in [0,\tau_1-\sigma_1)$ or $y_j\in (\tau_1-\sigma_1,0]$, depending on the sign of $\tau_1-\sigma_1$. The balanced digits of any $r\in\R$ are among $\lfloor r\rfloor$ and $\lceil r\rceil$ so that $\bal{y_j}{i}$ is an element of the set
$\{0,\dots \lceil\tau_1-\sigma_1\rceil\}$ or the set $\{-\lceil\sigma_1-\tau_1\rceil,\dots,0\}$. In any case, there are at most $\lceil|\tau_1-\sigma_1|\rceil+1$ possible symbols $\bal{y_j}{i}$. Using \eqref{eq:sigmatau} this number can also be written as
$$
\lceil|(1-\alpha)A_1+\alpha C_1-(1-\alpha)B_1-\alpha D_1)|+1\rceil = 
\lceil|P_1-Q_1|+1\rceil.
$$
Analogously, there are at most 
$$
\lceil|\tau_2-\sigma_2|\rceil+1 = 
\lceil|(1-\beta)A_2+\beta B_2-(1-\beta)C_2-\beta D_2)|+1\rceil = 
\lceil|R_2-S_2|+1\rceil
$$
possible values of $\bal{z_i}{j}$.

The number of possible values $c_{i,j}$ depends on whether $\bal{\beta}{j}$ is 0 or 1. 
If $\bal{\beta}{j}=1$ then $c_{i,j}=\fr{iy_{j+1}}-\fr{iy_j}+\fr{i\alpha}(B_1-D_1)$
so that $c_{i,j}$ is an integer in the open interval $(-1,1+B_1-D_1)$ or
$(-(1+B_1-D_1),1)$, depending on the sign of $B_1-D_1$. In either case, the number of possible values is at most $|B_1-D_1|+1$. If $\bal{\beta}{j}=0$ then there are at most
$|A_1-C_1|+1$ possible values of $c_{i,j}$. Analogously, the number of different values that $d_{i,j}$ attains is bounded by $|C_2-D_2|+1$ when $\bal{\alpha}{i}=1$ and by $|A_2-B_2|+1$ when 
$\bal{\alpha}{i}=0$. Because $\bal{\beta}{j}=1$ and $\bal{\alpha}{i}=1$ characterize the presence of horizontal and vertical stripes in the tile, respectively, we obtain the upper bounds on the cardinalities of the horizontal and vertical color sets given in \eqref{eq:cardinalities-color-sets}.
\end{proof}

We may now prove Theorem~\ref{mainthm:an-aperiodic-tileset-for-every-quadratic-number}.

\begin{THEOREMII}
    \StatementMainTHEOREMII
\end{THEOREMII}

\begin{proof}%
    (i) Let $\alpha,\beta\in[0,1]$ be elements of a quadratic number field $K$.
From Lemma~\ref{lem:a-solution-for-every-pair-of-quadratic-numbers},
there exist four non-collinear points $A,B,C,D \in \Z^2$ such that
$(\alpha,\beta)$ is a solution to Equation~\eqref{alpha_beta_relations}.

    (ii) and (iii)
    From Lemma~\ref{lem:a-solution-for-every-pair-of-quadratic-numbers},
    we can choose $A,B,C,D$ so that opposite sides of $ABDC$ are not parallel
    since $\alpha$ and $\beta$ are both irrational or both rational.
    From Theorem~\ref{thm:balancetheorem},
there exists a finite set
$\Tcal=\Tcal_\square\cup\Tcal_\boxbar\cup\Tcal_\boxminus\cup \Tcal_\boxplus$
of Wang tiles with horizontal and vertical stripes such that
$\Tcal$ admits a tiling $w:\Z^2\longrightarrow {\cal T}$ where horizontal
    stripes (\emph{i.e.}, tiles in $\Tcal_\boxminus\cup \Tcal_\boxplus$) have
    density $\beta$, and the vertical stripes (tiles in $\Tcal_\boxbar\cup
    \Tcal_\boxplus$) have density $\alpha$, and
$\Tcal$ determines the quadrilateral $ABDC$.

    (iv)
Since the opposite sides of $ABDC$ are not parallel,
we deduce from Proposition~\ref{prop:alpha_beta_solutions}
that Equation~\eqref{alpha_beta_relations} has finitely many solutions.
    From Proposition~\ref{prop:sufficient-condition-quadratic-equation},
    we deduce that in every valid tiling by $\Tcal$, the horizontal and the
    vertical densities of stripes exist and are solutions to Equation~\eqref{alpha_beta_relations}.

If $\beta$ and $\alpha$ are irrational, then
we deduce from Proposition~\ref{prop:alpha_beta_solutions}
that every solution to Equation~\eqref{alpha_beta_relations} is irrational.
Therefore, every valid configuration $\Z^2\to\Tcal$
    is non-periodic since it must have
irrational densities of vertical and horizontal stripes.
    From (iii), the set $\Tcal$ admits a valid configuration, so we
conclude that the tile set $\Tcal$ is aperiodic.
\end{proof}

\begin{remark}
Note that if exactly one of $\alpha$ and $\beta$ is rational, then the
number of solutions is infinite. 
However, we believe that Proposition~\ref{prop:sufficient-condition-quadratic-equation} 
can be modified to show that the stripe density corresponding to the rational
value exists and is equal to it.
\end{remark}

\begin{example}
Let $\ABCD{A_1}{A_2}{B_1}{B_2}{C_1}{C_2}{D_1}{D_2}
    =\ABCD{1}{1}{0}{1}{1}{0}{-1}{-1}$.
By Proposition~\ref{prop:alpha_beta_solutions}, 
Equation~\eqref{alpha_beta_relations} has two solutions. 
In fact, \eqref{alpha_beta_relations} is equivalent to
$$
\left\{
\begin{array}{rcl}
\beta &=& \alpha\\
\alpha^2+\alpha-1 &=& 0.
\end{array}
\right.
$$
Only the solution
$$
\alpha=\beta=\varphi^{-1}=\frac{\sqrt{5}-1}{2}
$$
satisfies  $\alpha,\beta\in [0,1]$. Let us construct a tile set $\Tcal$ that satisfies Theorem~\ref{thm:balancetheorem}.
From \eqref{eq:sigmatau} and \eqref{eq:sigmatau2} we obtain
$$
\begin{array}{rclcl}
\sigma_1 &=& \sigma_2 &=& -\varphi^{-1},\\
\tau_1 &=& \tau_2 &=& 1,
\end{array}
$$
so that $\tau_1-\sigma_1=\tau_2-\sigma_2=1+\varphi^{-1}=\varphi$.
Using the choice $y_0=0$ we have, for all $j\in\Z$, that $$y_j=(\tau_1-\sigma_1)\fr{j\beta}=\varphi\fr{j\beta}\in[0,\varphi).$$ This
means that $\bal{y_j}{i}\in\{0,1,2\}$ for all $i,j\in\Z$. 
Similarly, choosing $z_0=0$ gives $\bal{z_i}{j}\in\{0,1,2\}$ for all $i,j\in\Z$. 

The calculations bounding $c_{i,j}$ at the end of the proof of  Theorem~\ref{thm:balancetheorem} show that
$$
c_{i,j}=
\left\{
\begin{array}{ll}
\fr{iy_{j+1}}-\fr{iy_j}+\fr{i\alpha}(B_1-D_1)=\fr{iy_{j+1}}-\fr{iy_j}+\fr{i\alpha}\in (-1,2), & \mbox{ when $\bal{\beta}{j}=1$},\\
\fr{iy_{j+1}}-\fr{iy_j}+\fr{i\alpha}(A_1-C_1)=\fr{iy_{j+1}}-\fr{iy_j}\in (-1,1),
& \mbox{ when $\bal{\beta}{j}=0$},\\
\end{array}
\right. 
$$
so that $c_{i,j}\in \{0,1\}$ if $\bal{\beta}{j}=1$ and $c_{i,j}\in \{0\}$ if $\bal{\beta}{j}=0$. 
Similarly, $d_{i,j}\in\{0,1\}$ when $\bal{\alpha}{i}=1$ and $d_{i,j}\in\{0\}$ when $\bal{\alpha}{i}=0$.

All this combined, we have the following vertical and horizontal color sets, with and without stripes:
$$
\begin{array}{rcl}
I_0 &=& \{0\}\times \{0\}\times\{0,1,2\},\\
I_1 &=& \{1\}\times \{0\}\times\{0,1,2\}\ \cup\ \{1\}\times \{1\}\times\{0,1,2\},\\
J_0 &=& \{0\}\times\{0,1,2\}\times \{0\},\\
J_1 &=& \{1\}\times\{0,1,2\}\times \{0\}\ \cup\ \{1\}\times\{0,1,2\}\times \{1\}.\\
\end{array}
$$
So there are nine horizontal colors and nine vertical colors. Taking all tiles with these colors and satisfying  \eqref{eq:fourbulletpoints} yields an aperiodic Wang tile set with 108 tiles. As discussed in Remark~\ref{rem:effective}, not all these tiles may be necessary. In fact, computer experiments seem to indicate that only 80 of these 108 tiles come up from \eqref{eq:balancedtiles} for any $i,j\in\Z$.
\end{example}

\section{Open questions}
\label{sec:open}

We finish this article with some open questions.
First, one may wonder if statement (iv) of
Theorem~\ref{mainthm:an-aperiodic-tileset-for-every-quadratic-number}
can be improved. More precisely, we may ask the following question.

\begin{question}\label{question:improve-ThmB-iv}
    Let $\Tcal$ be a set of Wang tile constructed
    in Theorem~\ref{mainthm:an-aperiodic-tileset-for-every-quadratic-number}
    from some pair $(\alpha,\beta)$ of quadratic numbers
    using Theorem~\ref{thm:balancetheorem}.
    Is it true that for every configuration $w\in\Omega_\Tcal$,
    the vertical and the horizontal densities of stripes exist and are equal to
    $\alpha$ and $\beta$, respectively?
\end{question}

For instance, when
    $\ABCD{A_1}{A_2}{B_1}{B_2}{C_1}{C_2}{D_1}{D_2}
    =\ABCD{1}{3}{1}{0}{0}{-1}{-3}{-3}$,
Equation~\eqref{alpha_beta_relations} has two solutions in $[0,1]^2$: 
\[
    (\alpha,\beta)\approx(0.361325, 0.589197)
    \quad\text{ and }\quad
    (\alpha',\beta')\approx(0.638675, 0.18858).
\]
Let $\Tcal$ be the tile set constructed as in Theorem~\ref{thm:balancetheorem} from the solution $(\alpha,\beta)$.
From Theorem~\ref{thm:balancetheorem}, we know there exists a configuration in $\Omega_\Tcal$ such that 
the densities of vertical and horizontal stripes is $(\alpha,\beta)$.
But, is it possible to show that \emph{for every} configuration in $\Omega_\Tcal$,
the densities of vertical and horizontal stripes is $(\alpha,\beta)$?
From Proposition~\ref{prop:sufficient-condition-quadratic-equation}, 
we know that for every configuration in $\Omega_\Tcal$,
the densities of vertical and horizontal stripes is a solution to the equation.
So, one only needs to rule out the possibility of having densities of vertical
and horizontal stripes equal to the other solution $(\alpha',\beta')$.

Of course, to achieve that, Theorem~\ref{thm:balancetheorem} 
will need to become stronger, maybe
providing additional information about the tile set 
involving $\sigma_1$, $\tau_1$, $\sigma_2$ and $\tau_2$.
Because the union $\Tcal\cup\Tcal'$ is also a tile set which determines the same
quadrilateral $ABDC$ if 
$\Tcal$ is the tile set generated from $(\alpha,\beta)$
and $\Tcal'$ is the tile set generated from $(\alpha',\beta')$.
And there are configurations in the Wang shift
$\Omega_{\Tcal\cup\Tcal'}$ with both possible values of stripe densities. 
Bounding the elements of the matrices \eqref{eq:balancedtiles} is not sufficient to distinguish
 $\Tcal$ and $\Tcal'$ and, in fact, the tile sets constructed effectively for  $(\alpha,\beta)$ and  $(\alpha',\beta')$ in Remark~\ref{rem:effective}  are identical.

A more general question is about the Wang shifts generated by these tile sets.
In \cite{MR4944989}, minimality of the metallic mean Wang shifts was proved
using a criterion based on its self-similarity given as a 2-dimensional
rectangular substitution.
Can this method be used or extended to answer the following question?

\begin{question}\label{question:minimality}
    Let $\Tcal$ be a set of Wang tiles constructed
    in Theorem~\ref{mainthm:an-aperiodic-tileset-for-every-quadratic-number}
    from some pair of quadratic numbers
    using Theorem~\ref{thm:balancetheorem}.
    Is the Wang shift $\Omega_\Tcal$ minimal as a dynamical system?
\end{question}

Note that a positive answer to Question~\ref{question:minimality} implies
that Question~\ref{question:improve-ThmB-iv} also has a positive answer.

For Question~\ref{question:minimality},
the first subcases to consider are the solutions $(\alpha,\beta)$ for which the
associated Wang shift is expected to be self-similar, as in the case of
metallic mean Wang tiles.
Recall that a quadratic surd is said to be \emph{reduced} if $\zeta$ > 1 and its conjugate $\zeta'$
satisfies the inequalities $-1 < \zeta' < 0$.
Galois proved that if $\zeta$ is a reduced quadratic surd, 
then the continued fractions for $\zeta$ and for $(-1/\zeta')$ are both purely periodic.

\begin{question}\label{question:self-similarity}
    Let $\alpha,\beta$ be two reduced quadratic surds in the same quadratic number field
    and $\alpha',\beta'$ respectively be their Galois algebraic conjugates.
    Let $\Tcal$ be a set of Wang tiles constructed
    in Theorem~\ref{mainthm:an-aperiodic-tileset-for-every-quadratic-number}
    from the pair $(-\alpha',-\beta')\in[0,1]^2$
    using Theorem~\ref{thm:balancetheorem}.
    Is the Wang shift $\Omega_\Tcal$ self-similar?
    In other words,
    does there exist an expansive 2-dimensional substitution $\omega:\Omega_\Tcal\to\Omega_\Tcal$
    such that $\Omega_\Tcal=\overline{\omega(\Omega_\Tcal)}^\sigma$?
\end{question}

\section*{Acknowledgements}

This work was partly funded from France's Agence Nationale de la Recherche
(ANR) project IZES (ANR-22-CE40-0011) and
by the Research Council of Finland under project number  354965.
It was also supported by grants from the
\emph{Symbolic Dynamics and Arithmetic Expansions}
(SymDynAr) Project, co-funded by ANR (ANR-23-CE40-0024)
and FWF (\href{https://dx.doi.org/10.55776/I6750}{I 6750}),
the Austrian Science Fund.

The second author acknowledges Université de Bordeaux's program
``\textit{Mobilité internationale des personnels de recherche}''
partially supporting a one-year stay at CRM-CNRS in Montréal (2025-2026)
within the \emph{Laboratoire d'Algèbre, de Combinatoire et d'Informatique Mathématique}
(LACIM) at \emph{Université du Québec à Montréal}.

\bibliographystyle{alpha-first-name-initials-with-doi}

\bibliography{biblio}

\begin{thebibliography}{SMKGS24}

\bibitem[AGS92]{MR1156132}
R. Ammann, B. Gr\"unbaum, and G.~C. Shephard.
\newblock Aperiodic tiles.
\newblock {\em Discrete Comput. Geom.}, 8(1):1--25, 1992.

\bibitem[Aki12]{akiyama_note_2012}
S. Akiyama.
\newblock A {Note} on {Aperiodic} {Ammann} {Tiles}.
\newblock {\em Discrete Comput. Geom.}, 48(3):702--710, October 2012.
\newblock \doi{10.1007/s00454-012-9418-4}.

\bibitem[Ber66]{MR0216954}
R. Berger.
\newblock The undecidability of the domino problem.
\newblock {\em Mem. Amer. Math. Soc. No.}, 66:72, 1966.

\bibitem[BG13]{MR3136260}
M. Baake and U. Grimm.
\newblock {\em Aperiodic {O}rder. {V}ol. 1}, volume 149 of {\em Encyclopedia of
  Mathematics and its Applications}.
\newblock Cambridge University Press, Cambridge, 2013.
\newblock \doi{10.1017/CBO9781139025256}.

\bibitem[BKOR87]{zbMATH04030361}
H.~J.~M. Bos, C. Kers, F. Oort, and D.~W. Raven.
\newblock Poncelet's closure theorem.
\newblock {\em Expo. Math.}, 5:289--364, 1987.

\bibitem[BR10]{MR2742574}
V. Berth\'e and M. Rigo, editors.
\newblock {\em Combinatorics, automata and number theory}, volume 135 of {\em
  Encyclopedia of Mathematics and its Applications}.
\newblock Cambridge University Press, Cambridge, 2010.
\newblock \doi{10.1017/CBO9780511777653}.

\bibitem[Cul96]{MR1417576}
K. Culik, II.
\newblock An aperiodic set of {$13$} {W}ang tiles.
\newblock {\em Discrete Math.}, 160(1-3):245--251, 1996.
\newblock \doi{10.1016/S0012-365X(96)00118-5}.

\bibitem[DM25]{zbMATH08141186}
V. Dragovi{\'c} and M.~H. Murad.
\newblock Parable of the parabola.
\newblock {\em Expo. Math.}, 43(6):36, 2025.
\newblock \doi{10.1016/j.exmath.2025.125717}.
\newblock Id/No 125717.

\bibitem[ENP07]{MR2369448}
S. Eigen, J. Navarro, and V.~S. Prasad.
\newblock An aperiodic tiling using a dynamical system and {B}eatty sequences.
\newblock In {\em Dynamics, ergodic theory, and geometry}, volume~54 of {\em
  Math. Sci. Res. Inst. Publ.}, pages 223--241. Cambridge Univ. Press,
  Cambridge, 2007.
\newblock \doi{10.1017/CBO9780511755187.009}.

\bibitem[FGMn24]{MR4822285}
D. Frettl\"oh, A. Garber, and N. Ma\~nibo.
\newblock Substitution tilings with transcendental inflation factor.
\newblock {\em Discrete Anal.}, pages Paper No. 11, 24, 2024.
\newblock \doi{10.19086/da.1254}.

\bibitem[Fog02]{MR1970385}
N.~P. Fogg.
\newblock {\em Substitutions in Dynamics, Arithmetics and Combinatorics},
  volume 1794 of {\em Lecture Notes in Mathematics}.
\newblock Springer-Verlag, Berlin, 2002.
\newblock Edited by V. Berth\'{e}, S. Ferenczi, C. Mauduit and A. Siegel.
\newblock \doi{10.1007/b13861}.

\bibitem[Fre17]{MR3791847}
D. Frettl\"{o}h.
\newblock More inflation tilings.
\newblock In {\em Aperiodic order. {V}ol. 2}, volume 166 of {\em Encyclopedia
  Math. Appl.}, pages 1--37. Cambridge Univ. Press, Cambridge, 2017.

\bibitem[GS87]{MR857454}
B. Gr\"unbaum and G.~C. Shephard.
\newblock {\em Tilings and patterns}.
\newblock W. H. Freeman and Company, New York, 1987.

\bibitem[Hoc16]{MR3525488}
M. Hochman.
\newblock Multidimensional shifts of finite type and sofic shifts.
\newblock In {\em Combinatorics, words and symbolic dynamics}, volume 159 of
  {\em Encyclopedia Math. Appl.}, pages 296--358. Cambridge Univ. Press,
  Cambridge, 2016.
\newblock \doi{10.1017/CBO9781139924733.010}.

\bibitem[HP94]{zbMATH00598493}
W.~V.~D. Hodge and D. Pedoe.
\newblock {\em Methods of algebraic geometry. {Volume} {I}. {Book} {I}:
  {Algebraic} preliminaries. {Book} {II}: {Projective} space}.
\newblock Camb. Math. Libr. Cambridge: Cambridge University Press, 1994.

\bibitem[Jan21]{jang_directional_2021}
H. Jang.
\newblock {\em Directional {Expansivenss}}.
\newblock {PhD} {Thesis}, The George Washington University, August 2021.

\bibitem[JR21]{zbMATH07421483}
E. Jeandel and M. Rao.
\newblock An aperiodic set of 11 {Wang} tiles.
\newblock {\em Adv. Comb.}, 2021:37, 2021.
\newblock \doi{10.19086/aic.18614}.
\newblock Id/No 1.

\bibitem[JR25]{jang_directional_2025}
H. Jang and E.~A. Robinson, Jr.
\newblock Directional expansiveness for {Rd}-actions and for {Penrose} tilings,
  April 2025.
\newblock \arxiv{2504.10838}.

\bibitem[Kar96]{MR1417578}
J. Kari.
\newblock A small aperiodic set of {W}ang tiles.
\newblock {\em Discrete Math.}, 160(1-3):259--264, 1996.
\newblock \doi{10.1016/0012-365X(95)00120-L}.

\bibitem[Lab19]{MR3978536}
S. Labb\'{e}.
\newblock A self-similar aperiodic set of 19 {W}ang tiles.
\newblock {\em Geom. Dedicata}, 201:81--109, 2019.
\newblock \doi{10.1007/s10711-018-0384-8}.

\bibitem[Lab20]{labbe_three_2020}
S. Labbé.
\newblock Three characterizations of a self-similar aperiodic 2-dimensional
  subshift.
\newblock December 2020.
\newblock \arxiv{2012.03892}, 46 p., book chapter in preparation.

\bibitem[Lab21]{MR4226493}
S. Labb\'{e}.
\newblock Substitutive structure of {J}eandel-{R}ao aperiodic tilings.
\newblock {\em Discrete Comput. Geom.}, 65(3):800--855, 2021.
\newblock \doi{10.1007/s00454-019-00153-3}.

\bibitem[Lab25a]{MR4944989}
S. Labb\'e.
\newblock Metallic mean {W}ang tiles {I}: self-similarity, aperiodicity and
  minimality.
\newblock {\em Forum Math. Sigma}, 13:Paper No. e133, 68, 2025.
\newblock \doi{10.1017/fms.2025.10069}.

\bibitem[Lab25b]{MR4963140}
S. Labb\'e.
\newblock Metallic mean {W}ang tiles {II}: the dynamics of an aperiodic
  computer chip.
\newblock {\em Forum Math. Sigma}, 13:Paper No. e155, 51, 2025.
\newblock \doi{10.1017/fms.2025.10098}.

\bibitem[Lep24]{Lepsova-2024}
J. Lep\v{s}ov\'{a}.
\newblock {\em Substitutive structures in combinatorics, number theory, and
  discrete geometry}.
\newblock {PhD} {Thesis}, Universit\'e de Bordeaux and Czech Technical
  University in Prague, 2024.
\newblock Available online at \url{https://theses.fr/2024BORD0083} or
  \url{https://theses.hal.science/tel-04679032}.

\bibitem[Lin04]{MR2078846}
D. Lind.
\newblock Multi-dimensional symbolic dynamics.
\newblock In {\em Symbolic dynamics and its applications}, volume~60 of {\em
  Proc. Sympos. Appl. Math.}, pages 61--79. Amer. Math. Soc., Providence, RI,
  2004.
\newblock \doi{10.1090/psapm/060/2078846}.

\bibitem[LM95]{MR1369092}
D. Lind and B. Marcus.
\newblock {\em An Introduction to Symbolic Dynamics and Coding}.
\newblock Cambridge University Press, Cambridge, 1995.
\newblock \doi{10.1017/CBO9780511626302}.

\bibitem[ML71]{zbMATH03367095}
S. Mac~Lane.
\newblock {\em Categories for the working mathematician}, volume~5 of {\em
  Grad. Texts Math.}
\newblock Springer, Cham, 1971.
\newblock \doi{book/10.1007/978-1-4757-4721-8}.

\bibitem[Moz89]{MR1014984}
S. Mozes.
\newblock Tilings, substitution systems and dynamical systems generated by
  them.
\newblock {\em J. Analyse Math.}, 53:139--186, 1989.
\newblock \doi{10.1007/BF02793412}.

\bibitem[Mus37]{zbMATH02521565}
J.~R. Musselman.
\newblock On four lines and their associated parabola.
\newblock {\em Am. Math. Mon.}, 44:513--521, 1937.
\newblock \doi{10.2307/2301227}.

\bibitem[Pen74]{penrose_role_1974}
R. Penrose.
\newblock The r\^ole of aesthetics in pure and applied mathematical research.
\newblock {\em Bull. Inst. Math. Appl.}, 10(Jul-Aug):266--271, 1974.

\bibitem[Rob71]{MR0297572}
R.~M. Robinson.
\newblock Undecidability and nonperiodicity for tilings of the plane.
\newblock {\em Invent. Math.}, 12:177--209, 1971.
\newblock \doi{10.1007/BF01418780}.

\bibitem[Sch01]{MR1861953}
K. Schmidt.
\newblock Multi--dimensional symbolic dynamical systems.
\newblock In {\em Codes, Systems, and Graphical Models ({M}inneapolis, {MN},
  1999)}, volume 123 of {\em IMA Vol. Math. Appl.}, pages 67--82. Springer, New
  York, 2001.
\newblock \doi{10.1007/978-1-4613-0165-3\_3}.

\bibitem[Sen95]{zbMATH00768067}
M. Senechal.
\newblock {\em Quasicrystals and geometry}.
\newblock Cambridge: Cambridge Univ. Press, 1995.

\bibitem[SMKGS24]{MR4770585}
D. Smith, J.~S. Myers, C.~S. Kaplan, and C. Goodman-Strauss.
\newblock An aperiodic monotile.
\newblock {\em Comb. Theory}, 4(1):Paper No. 6, 91, 2024.
\newblock \doi{10.5070/C64163843}.

\bibitem[Wal82]{MR648108}
P. Walters.
\newblock {\em An Introduction to Ergodic Theory}, volume~79 of {\em GTM}.
\newblock Springer-Verlag, New York-Berlin, 1982.

\bibitem[Wan61]{wang_proving_1961}
H. Wang.
\newblock Proving theorems by pattern recognition -- {II}.
\newblock {\em Bell System Technical Journal}, 40(1):1--41, January 1961.
\newblock \doi{10.1002/j.1538-7305.1961.tb03975.x}.

\end{thebibliography}

\end{document}